\documentclass{amsart}
\usepackage{amssymb}
\usepackage{amsmath}
\numberwithin{equation}{section}
\usepackage[dvipsnames]{xcolor}
\usepackage[hidelinks,colorlinks=true,citecolor=ForestGreen,linkcolor=RubineRed]{hyperref}
\usepackage{amsthm}
\usepackage{appendix}
\usepackage{bookmark}
\usepackage{setspace}
\usepackage{mathrsfs}
\usepackage{mathtools}
\usepackage{tikz-cd}
\usepackage{manfnt}
\usepackage{stmaryrd}
\usepackage{comment}
\usepackage{fullpage}

\DeclareFontFamily{U}{wncy}{}
\DeclareFontShape{U}{wncy}{m}{n}{<->wncyr10}{}
\DeclareSymbolFont{mcy}{U}{wncy}{m}{n}
\DeclareMathSymbol{\Sha}{\mathord}{mcy}{"58}

\newtheorem{thm}{Theorem}
\numberwithin{thm}{section}
\newtheorem{prop}[thm]{Proposition}
\newtheorem{lem}[thm]{Lemma}
\newtheorem{cor}[thm]{Corollary}
\theoremstyle{remark}
\newtheorem{rem}[thm]{Remark}
\theoremstyle{definition}

\newtheorem{lthm}{Theorem}

\newcommand{\1}{\mathbf{1}}

\newcommand{\C}{\mathbf{C}}
\newcommand{\chx}[1]{\langle #1\rangle}
\newcommand{\cl}{\mathrm{Cl}}

\newcommand{\e}{\mathrm{e}}

\newcommand{\F}{\mathbf{F}}

\newcommand{\G}{\mathbf{G}}
\newcommand{\gal}{\mathrm{Gal}}

\newcommand{\id}{\mathrm{id}}

\newcommand{\im}{\mathrm{im}}

\newcommand{\mat}[4]{\left[\begin{matrix}
		#1 & #2\\
		#3 & #4\\
	\end{matrix}\right]}

\newcommand{\ord}{\mathrm{ord}}

\newcommand{\bP}{\mathbb{P}}

\newcommand{\Q}{\mathbf{Q}}

\newcommand{\SL}{\mathrm{SL}}

\newcommand{\Z}{\mathbf{Z}}

\newcommand{\Area}{\mathrm{Area}}
\newcommand{\BH}{\mathrm{BH}}
\newcommand{\can}{\mathrm{can}}

\newcommand{\Div}{\mathrm{div}}
\newcommand{\Eis}{\mathrm{Eis}}
\newcommand{\End}{\mathrm{End}}

\newcommand{\Katz}{\mathrm{Katz}}
\newcommand{\KL}{\mathrm{KL}}
\newcommand{\Mes}{\mathrm{Mes}}

\newcommand{\ur}{\mathrm{ur}}
\newcommand{\sgn}{\mathrm{sgn}}
\newcommand{\Tate}{\mathrm{Tate}}
\newcommand{\gen}{\mathrm{gen}}
\newcommand{\triv}{\mathrm{triv}}

\newcommand{\bV}{\mathbb{V}}

\newcommand{\cA}{\mathcal{A}}

\newcommand{\cE}{\mathcal{E}}
\newcommand{\cF}{\mathcal{F}}
\newcommand{\cH}{\mathcal{H}}
\newcommand{\cO}{\mathcal{O}}
\newcommand{\cP}{\mathcal{P}}

\newcommand{\fa}{\mathfrak{a}}
\newcommand{\fb}{\mathfrak{b}}

\newcommand{\fF}{\mathfrak{F}}
\newcommand{\fJ}{\mathfrak{J}}

\newcommand{\fo}{\mathfrak{o}}
\newcommand{\fp}{\mathfrak{p}}
\newcommand{\fP}{\mathfrak{P}}
\newcommand{\fG}{\mathfrak{G}}

\newcommand{\fH}{\mathfrak{H}}
\newcommand{\fx}{\mathfrak{x}}
\newcommand{\fy}{\mathfrak{y}}

\newcommand{\sA}{\mathscr{A}}

\newcommand{\sL}{\mathscr{L}}
\newcommand{\sR}{\mathscr{R}}

\newcommand{\barpi}{\bar{\varpi}}
\newcommand{\umega}{\underline{\omega}}
\newcommand{\EE}{\mathsf{E}}
\newcommand{\PP}{\mathsf{P}}
\newcommand{\HH}{\mathsf{H}}
\newcommand{\qq}{\mathsf{q}}
\newcommand{\uu}{\mathsf{u}}
\newcommand{\vv}{\mathsf{v}}
\newcommand{\RR}{\mathsf{R}}
\newcommand{\Our}{\hat{\fo}^{\rm ur}_{\fp}}

\newcommand{\llrrparen}[1]{(\!(#1)\!)}

\title{On the Bernoulli--Hurwitz periods}

\author{Luochen Zhao}

\date{Oct 20, 2025}



\address{Morningside Center of Mathematics\\No.55 Zhongguancun East Road\\Beijing\\100190\\China}
\email{luochenzhao@amss.ac.cn}
\begin{document}
	\maketitle
	
	\vspace{.4cm}
	\begin{center}
		\textit{To Antonio}
		
		\rule{2cm}{0.4pt}
	\end{center}
	\vspace{.5cm}
	\begin{abstract}
		Let $E$ be an elliptic curve having CM by the ring of integers of an imaginary quadratic field $K$ in which $p$ splits. Following Lichtenbaum, the Bernoulli--Hurwitz numbers of $E$ (i.e., values of Eisenstein series evaluated at $E$ up to normalization) admit integral representations given by a $p$-adic measure constructed from an elliptic function. We show that the periods of this measure are in fact special values of a family of weight one Eisenstein series at the CM curve $E$ equipped with certain level data, and explicitly relate it to Katz's one-variable $p$-adic Eisenstein measure, whereby we derive period formulas of the Bernoulli--Hurwitz measure attached to any ordinary elliptic curve $\cE$ defined over a local field. Moreover, by exploiting the modularity of these periods, and thanks to the existence of abundant weight one Hasse-type invariants, we present a novel approach to the interpolation property of the Bernoulli--Hurwitz $p$-adic zeta functions of the ordinary elliptic curve $\cE$, and obtain a $p$-adic Kronecker's first limit formula.
	\end{abstract}

	\makeatletter
	\def\l@subsection{\@tocline{2}{0pt}{2.5pc}{5pc}{}}
	\def\l@subsubsection{\@tocline{2}{0pt}{5pc}{7.5pc}{}}
	\makeatother
	\tableofcontents
	
	\addtocontents{toc}{\protect\setcounter{tocdepth}{1}}
	\section{Introduction}\label{sec:intro}
	
	Let $K$ be an imaginary quadratic field, $F$ be its Hilbert class field, and let $\cO$ (resp.~$\fo$) be the ring of integers of $F$ (resp.~$K$). Let $E$ be an elliptic curve over $F$ with $\End_F(E) = \fo$, and fix forever a Weierstrass model $y^2 = x^3 + Ax + B$ of $E$ with $A,B\in \cO$. Denote by $\omega_E = dx/2y$ the generator of $H^0(E,\Omega^1_{E/F})$ over $F$, and form the period lattice $L\subset \C$ spanned by $\int_{\gamma}\omega_E$ for $\gamma\in H_1(E(\C),\Z)$. In this way, the map
	\begin{align}
		\C \to \bP^2(\C), \quad z\mapsto (\wp(z,L),\wp'(z,L)/2,1)
	\end{align}
	induces a biholomorphic map $\C/L\xrightarrow{\sim} E(\C)$, where $\wp(z,L)$ is the Weierstrass $\wp$-function. For $k\in \Z_{\ge 1}$, the Bernoulli--Hurwitz numbers are defined to be \cite{katz-bernoulli-hurwitz}
	\begin{align}\label{eq:BH}
		\BH_k = \begin{cases}
			k!\sum_{0\ne \mu\in L} \mu^{-k} & k\in 2\Z_{\ge 2};\\
			0 & \text{otherwise}.
		\end{cases}
	\end{align}
	Then it is known that $\BH_k\in F$ by \textit{loc.~cit.}, and $\BH_k = 0$ unless $k$ is divisible by $|\fo^\times|$.
	
	Let $p\ge 5$ be a prime, and throughout we assume $p$ is split in $K$, say $p\fo = \fp\bar{\fp}$. Fix once and for all an embedding $\bar{\Q}\to \bar{K}_\fp$, and let $\fP$ be the prime of $F$ above $\fp$ corresponding to this embedding. We will further assume that $E$ has good reduction at $\fP$, and we require the Weierstrass model of $E$ to be minimal at $\fP$. So $y^2 = x^3 + Ax +B$ gives an integral model to the ordinary elliptic curve $E/\cO_\fP$, whereby $\omega_E$ generates $H^0(E,\Omega^1_{E/\cO_{\fP}})$ over $\cO_\fP$. The $p$-adic interpolation of Bernoulli--Hurwitz numbers of the ordinary CM curve $E$ has been studied by Katz \cite[\S\S3.8-3.9]{katz-eisenstein-measure} and Lichtenbaum \cite{lichtenbaum} (the case of generalized Bernoulli--Hurwitz numbers); see also works of Vi\v{s}ik--Manin \cite{visik-manin}, Katz \cite{katz-real-analytic-eisenstein}, Barsky \cite{barsky-bh}, Coates--Wiles \cite{coates-wiles-aus}, Cassou-Nogu\`es \cite{cassou-nogues-CM}, Yager \cite{yager84} and Colmez-Schneps \cite{colmez-schneps}. A precise formulation of such interpolation that suits us is the following: Denote by $\Our$ the ring of integers in the completion of the maximal unramified extension of $K_\fp$. Fixing auxiliary coprime integers $c,N\in \Z_{>1}$ that are prime to $p$, then there exists a $p$-adic measure $\mu_{\BH} =  \mu_{c,N}\in\Mes(\Z_p,\Our)$ such that, for all $k\in \Z_{\ge 0}$,
	\begin{align}\label{eq:interpol}
		\int_{\Z_p} x^k \mu_{\BH}(x) = \Omega_{p}^{-k}(1-N^{k+1})(c^2 - c^{k+1})\frac{\BH_{k+1}}{k+1}.
	\end{align}
	Here $\Omega_p$ is some $p$-adic period recalled in \S\ref{sec:measure} and uniquely chosen via an extra condition \eqref{ass:iota}. Below, we will reconstruct $\mu_{\BH}$ from an elliptic function in a fashion similar to \cite{lichtenbaum}, and will also explain the interpolation formula \eqref{eq:interpol} in detail.
	
	\subsection{Main results}
	
	Denote by $e$ the order of $\fp$ in the class group $\cl(K)$, and fix throughout a generator $\varpi$ of $\fp^e$. Our first goal is to explicitly compute the values $\mu_{\BH}(a+p^n\Z_p)$ for arbitrary $a\in \Z_p$ and $n\in \Z_{\ge 0}$, which turns out to be arithmetic in nature:
	\begin{lthm}[Theorem \ref{thm:epf}]\label{thm:main}
		Let $n\in \Z_{\ge 0}$ and $a\in \Z_p$. Write $L  = \Omega_\infty\fa$ for some $\Omega_\infty\in \C^\times$ and fractional ideal $\fa\subset K$ prime to $\fp$, with $\fa = \Z+\Z\varsigma$ for some $\im(\varsigma)>0$. Let $n'\ge n$ be the smallest integer divisible by $e$, and $s_{n'}\in \Z$ be congruent to $-\varsigma$ modulo $\fp^{n'}$. There exists a weight one Eisenstein series $\Psi_{a,n}$ of level $\Gamma^1(p^n)=\left\{\gamma \in \SL_2(\Z): \gamma \equiv \mat{1}{0}{*}{1}\bmod p^n\right\}$ defined over $\Z$, such that
		\begin{align}\label{eq:A}
			\mu_{\BH}(a\barpi^{-n'/e}+p^n\fo_\fp) = \Omega_\infty^{-1}\barpi^{-n'/e}\Psi_{a,n}\left(\frac{\varsigma+s_{n'}}{p^{n'-n}}\right),
		\end{align}
		where the right hand side is valued in $F_\fP(E[\bar{\fp}^n])$. Moreover, the constant term $\Psi_{a,n}(i\infty)$ is essentially the Kubota--Leopoldt period $\mu_{\rm KL}(a+p^n\Z_p)$, where $\mu_{\rm KL}\in \Mes(\Z_p,\Z_p)$ is characterized by the interpolation formula
		\begin{align}\label{eq:interpolation-KL}
			\int_{\Z_p}x^k\mu_{\rm KL}(x) = \begin{cases}
				(c^2 - c^{k+1})(1-N^{k+1})\zeta(-k), & \text{if }k\in \Z_{\ge 1};\\
				0, & \text{if }k=0.
			\end{cases}
		\end{align}
	\end{lthm}
	\begin{rem}
		The vanishing of $\mu_{\KL}(\Z_p)$ in \eqref{eq:interpolation-KL} is imperative: Taking $n=0$, Theorem \ref{thm:main} asserts the existence of a weight 1 full level modular form $\Psi_{a,0}$ whose special value at infinity gives $\mu_\KL(\Z_p)$. As there is no nonzero such forms, both $\Psi_{a,0}$ and $\mu_{\KL}(\Z_p)$ are forced to vanish.
	\end{rem}
	\begin{rem}
		It is worth noting that formula \eqref{eq:A} bears a close resemblance to Lemma 19 of \cite{colmez-schneps}. 
	\end{rem}
	
	The fact that Bernoulli--Hurwitz and Kubota--Leopoldt periods are specializations of a modular form $\Psi_{a,n}$ manifests Katz's philosophy of Bernoulli--Hurwitz numbers \cite{katz-bernoulli-hurwitz,katz-eisenstein-measure}. Consider now $\Phi_{a,n}(\tau) = \Psi_{a,n}(p^n\tau)$, which is of level $\Gamma_1(p^n) = \left\{\gamma \in \SL_2(\Z): \gamma \equiv \mat{1}{*}{0}{1}\bmod p^n\right\}$. Let $\mathbf{H}^{N,0}(x)$ be Katz's Eisenstein measure of trivial tame level \cite{katz-eisenstein-measure}, and denote by $\mu_{\Eis}(x) = 2c^2\mathbf{H}^{N,0}(x) - 2c\mathbf{H}^{N,0}(x/c)$ its $c$-regularization. Our second result (Corollary \ref{cor:BH-Phigen}, Theorem \ref{thm:Eis-Katz}) links the Bernoulli--Hurwitz measure with the Eisenstein measure, thereby providing an explicit comparison of the approaches of Lichtenbaum \cite{lichtenbaum} and Katz \cite{katz-eisenstein-measure}.
	\begin{lthm}
		Let $p\ge 5$ be a prime, and let $c,N\in \Z_{>1}$ be coprime and prime to $p$. Then
		\begin{align}
			\mu_{\Eis}(a+p^n\Z_p) = \Phi_{a,n}^{\gen},
		\end{align}
		where $\Phi_{a,n}^\gen$ is the generalized modular form in the sense of Katz \textit{loc.~cit} attached to $\Phi_{a,n}$. Moreover, we have the specialization
		\begin{align}\label{eq:bh-ka}
			\mu_{\BH}(a+p^n\Z_p) = \Omega_p\cdot \mu_\Eis(a+p^n\Z_p)(E,\iota^{-1}),
		\end{align}
		where $\iota: \hat{\G}_m\xrightarrow{\sim}\hat{E}$ is the isomorphism determined by $\Omega_p$. 
	\end{lthm}
	
	Now, let $\cE$ be an ordinary elliptic curve over $\Our$, $\varphi_{\cE}:\hat{\cE}\to \hat{\G}_m$ an isomorphism of formal groups, and $\omega_\cE$ a generator of $H^0(\cE,\Omega^1_{\cE/\Our})$. The data of $(\varphi_\cE,\omega_\cE)$ then defines a $p$-adic period $\Omega_{p,\cE}$ that coincides with $\Omega_p$ when $(\cE,\varphi_\cE,\omega_\cE) = (E,\iota^{-1},\omega_E)$. In view of \eqref{eq:bh-ka}, define the Bernoulli--Hurwitz measure of $(\cE,\varphi_\cE,\omega_\cE)$ by the period formula
	\begin{align}
		\mu_{\BH,\cE}(a+p^n\Z_p) = \Omega_{p,\cE}\cdot\mu_\Eis(a+p^n\Z_p)(\cE,\varphi_\cE).
	\end{align}
	Our final result, which is the crux of this article, concerns the special values of the Bernoulli--Hurwitz $p$-adic zeta function
	\begin{align}
		\zeta_{p,\cE}^\BH(s,\omega^i) = (c^2 - \omega^i(c)\chx{c}^{1-s})^{-1}(1-\omega^i(N)\chx{N}^{1-s})^{-1}\int_{\Z_p^\times}\omega^{i-1}(x)\chx{x}^{-s}\mu_{\BH,\cE}(x),
	\end{align}
	where $\omega$ is the Teichm\"uller character. We prove (Theorem \ref{thm:zeta-interpolation}, Theorem \ref{thm:first-limit-formula}):
	\begin{lthm}\label{thm:C}
		For all $k\in \Z_{\ge 0}$,
		\begin{align}\label{eq:interpol2}
			\zeta_{p,\cE}^\BH(-k,\omega^{k+1}) = \Omega_{p,\cE}^{-k}k!G^*_{k+1}(\cE,\omega_\cE),
		\end{align}
		where $(k-1)!G_k$ is the ($p$-adic if $k=2$) modular form whose $\qq$-expansion is
		\begin{align}
			\zeta(1-k)\1_{k>1} + \sum_{m\ge 1}\qq^m \sum_{d\mid m}\sgn(d)d^{k-1},
		\end{align}
		and $G_k^*$ is the full level $p$-adic modular form $G_k(\qq) - p^{k-1}G_k(\qq^p)$. Furthermore, let $\Delta^{(p)}$ be the $p$-adic modular form $\prod_{m\ge 1}(1-\qq^m)^{24p}/(1-\qq^{mp})^{24}$ and $\gamma_p = \lim_{s\to 1}[(1-1/p)^{-1}\zeta_p(s) - (s-1)^{-1}]$, where $\zeta_p(s)$ is the Kubota--Leopoldt $p$-adic zeta function with trivial character. Then we have the $p$-adic first limit formula
		\begin{align}
			\zeta_{p,\cE}^\BH(s) = \zeta_{p,\cE}^\BH(s,\omega^0) = \Omega_{p,\cE}\left\{(1-1/p)\left[\frac{1}{s-1} + \gamma_p + \log_p\Omega_{p,\cE}\right] -\frac{1}{12p}\log_p \Delta^{(p)}(\cE,\omega_\cE) + O(s-1)\right\}.
		\end{align}
	\end{lthm}
	\begin{rem}
		When $\cE$ is the CM curve $E$ and $\fp$ is principal, we can replace $G^*_{k+1}(E,\omega_E)$ in the interpolation formula by $(1-\varpi^{k+1}/p)G_{k+1}(E,\omega_E)$; see Corollary \ref{cor:interpolation-cm}. Similarly we have a variant of the $p$-adic first limit formula in the CM case (Corollary \ref{cor:FLF-cm}):
		\begin{align}
			\zeta_{p,E}^\BH(s) =
			\Omega_p\left\{ (1-1/p)
			\left[\frac{1}{s-1} + \gamma_p + \log_p\Omega_p + \frac{1}{12}\log_p\Delta(E,\omega_E)\right]
			+\frac{1}{p}\log_p\barpi
			+ O(s-1) \right\}.
		\end{align}
	\end{rem}
	\begin{rem}
		Different kinds of $p$-adic first limit formulas appear in the thesis of N.~Simard \cite[\S8.5]{simard} and the work of Bannai--Furusho--Kobayashi \cite[Proposition 1.3]{bannai-furusho-kobayashi}.
	\end{rem}
	\begin{rem}
		We do not know whether the periods $\mu_{\BH}(a+p^n\Z_p)$ satisfy an almost-periodicity in the sense of Barsky \cite[D\'efinition 4]{barsky77}. As such, it is not clear if $\zeta_p^\BH(s,\omega^i)$ has an infinite sum expression as in \cite{delbourgo-aus,knospe-washington,sum1,sum2}.
	\end{rem}
	
	It is worth emphasizing that we are less interested in the ends than the means, as, remarkably, the main reason we are able to prove a result of this kind is that some weight one Hasse-type modular invariants reveal themselves in a rather dramatic way. More precisely, our main contribution is a novel approach to the study of the special values, including the non-critical point $s = 1$, by exploiting the modularity of the Bernoulli--Hurwitz periods along with congruences of the $\qq$-expansions (after Swinnerton-Dyer, Serre, ...). Since the congruences are among modular forms of different weights, the weight one invariants are exactly the missing pieces to the jigsaw that allow us to invoke the $\qq$-expansion principle (see Proposition \ref{prop:q-exp-principle}), whence we can propagate the congruences of $\qq$-expansions to those of special values. It should be apparent to the reader that our strategy owes its very root to Katz's proof of the congruences of Bernoulli--Hurwitz numbers \cite{katz-bernoulli-hurwitz}, where the classical Hasse invariant was put into play. Finally, since this approach is completely free of elliptic units, our hope is that it would shed some new light on Gross's factorization \cite{gross-factorization} and thus the $p$-adic Chowla--Selberg formula, as well as $p$-adic limit formulas in the general CM field setting. We plan to investigate these questions in the near future.
	
	\subsection{Organization}
	The article can be divided into two parts. In the first, which is \S\S\ref{sec:measure}-\ref{sec:computation}, we explain the construction of the Bernoulli--Hurwitz measure that is a slight modification of \cite{lichtenbaum}, and then reduce the question of computing the periods into finding the $q$-expansions of the corresponding elliptic functions, which we do transcendentally over $\C$. In the second part, \S\S\ref{sec:modularity}-\ref{sec:FLF}, we study the arithmetic of Bernoulli--Hurwitz periods. After establishing their modularity in \S\ref{sec:modularity}, we proceed in \S\ref{sec:katz} to massage the shape of the periods and eventually realize them as specializations of the Eisenstein measure. Finally, forgetting all the elliptic functions from part one, in \S\S\ref{sec:interpolation}-\ref{sec:FLF} we bootstrap the Bernoulli--Hurwitz $p$-adic zeta functions for any ordinary elliptic curve, and prove their interpolation properties, as well as the $p$-adic first limit formula, via congruences of Eisenstein series.
	
	We also provide some supplementary materials in the appendix. In Appendix \ref{app:bernoulli}, we recount the period formulas of the Kubota--Leopoldt measure \cite{Lang}, and validate the interpolation \eqref{eq:interpolation-KL}. In Appendix \ref{app:hasse}, we provide some background on the ordinary loci, and define weight-one modular invariants that play prominent roles in the proof of Theorem \ref{thm:C}.

	\subsection{Notation}\label{subsec:notation}
	Throughout we fix the $\fP$-minimal Weierstrass model of $E$ above, as well as the lattice $L \subset \C$. Write $\omega_E = dx/2y\in H^0(E,\Omega^1_{E/K})$  for the one-form that corresponds to the holomorphic differential $dz$ of $\C/L$. Let $e$ denote the order of $\fp$ in $\cl(K)$, and fix a generator $\varpi$ of $\fp^e$. Fix auxiliary integers $c,N\in \Z_{>1}$ that are coprime and prime to $p$. Fix an algebraic closure $\bar{\Q}_p$ of $K_\fp$ and write $\C_p$ for its completion and $\cO_{\C_p}$ for the ring of integers in $\C_p$. Write $\hat{K}^\ur_\fp\subset \C_p$ for the completion of the maximal unramified extension of $K_\fp$, and $\Our$ for its ring of integers; often we think of $\Our$ as the ring of Witt vectors $W = W(\bar{\F}_p)$. Fix also an algebraic closure $\bar{\Q}$ of $\Q$, which we implicitly regard as a subfield of both $\C$ and $\C_p$ through fixed embeddings. The letter $\zeta$ is reserved for a root of unity in $\mu_{p^\infty}$, and we let $\zeta_n\in \mu_{p^n}(\bar{\Q})$ be the element corresponding to $\e^{2\pi i/p^n}\in \mu_{p^n}(\C)$. We denote  by $\cH$ the upper half plane $\{\tau\in \C:\im(\tau)>0\}$; if $\tau\in \cH$, write  $\qq = \e^{2\pi i\tau}$. Finally, write $\1_A$ for the indicator function, whose value is $1$ if the condition $A$ is met, and otherwise zero.
	
	\subsection{Conventions of modular forms}
	\label{subsec:conventions-modform}
	In this article, following Kubert--Lang \cite{kubert-lang}, we will frequently think of a modular form $f$ as a function on the set $\{(\omega_1,\omega_2)\in (\C^\times)^2\colon \im(\omega_1/\omega_2)>0\}$ enjoying certain properties (see p.~26 \textit{op.~cit.}). Still, when it is convenient, we shall also regard it as a function on the complex upper half plane $\cH$, as well as a function on the triples of elliptic curves with extra data, as we now explain. As a function on $\cH$, we adopt the convention $f(\tau) = f(\tau,1)$. As for the moduli interpretation, after \cite{katz-modular-scheme}, we can view the modular forms of level $\Gamma^1(p^n)$ or $\Gamma_1(p^n)$ for $n\in \Z_{>0}$ as functions on the triples $(\EE,P,\omega)$, where $\EE$ is an elliptic curve, $P\in \EE[p^n]\setminus \EE[p^{n-1}]$ and $\omega$ is a nowhere vanishing global differential on $\EE$. To translate, consider the two cases: When $f$ has level $\Gamma^1(p^n)$, we have
	\begin{align}\label{eq:moduli-interpretation}
		f(\C/(\Z\omega_1+\Z\omega_2),\omega_1/p^n,dz) = f(\omega_1,\omega_2).
	\end{align}
	It is straightforward to verify that, if another complex tuple $(\omega_1',\omega_2')\in (\C^\times)^2$ is such that $\im(\omega_1'/\omega_2')>0$, $\Z\omega_1+\Z\omega_2 = \Z\omega'_1+\Z\omega_2'$ and $\omega_1'\equiv \omega_1\bmod p^n(\Z\omega_1+\Z\omega_2)$, then
	\begin{align}
		f(\omega_1,\omega_2) = f(\omega_1',\omega_2').
	\end{align}
	So starting with a $\Gamma^1(p^n)$-form on complex tuples, \eqref{eq:moduli-interpretation} allows us to define a genuine function on the moduli of triples, and vice versa. Note that for $\tau\in \cH$, the $\qq$-expansion at $i\infty$ of a weight-$k$ form is given by
	\begin{align}
		(2\pi i)^{-k}f(\tau) = f(2\pi i\tau,2\pi i) = f(\C/(\Z2\pi i\tau + \Z2\pi i),2\pi i\tau/p^n,dz) = f(\Tate(\qq),\qq^{1/p^n},\omega_{\can}).
	\end{align}
	where $\Tate(\qq) \simeq \C^\times/\qq^{\Z}$ is the Tate curve over $\C$ and $\omega_{\can} = du/u$, with $u = \e^{2\pi i z}$ the variable of $\C^\times$.
	
	Similarly, in the level $\Gamma_1(p^n)$ case, we have
	\begin{align}
		f(\C/(\Z\omega_1+\Z\omega_2),\omega_2/p^n,dz) = f(\omega_1,\omega_2),
	\end{align}
	and the $\qq$-expansion at $i\infty$ of a weight-$k$ form is
	\begin{align}
		(2\pi i)^{-k}f(\tau) = f(\Tate(\qq),\e^{2\pi i/p^n},\omega_{\can}).
	\end{align}
	
	In Appendix \ref{app:hasse}, we will also consider modular forms of level $\Gamma_{00}(p^n)$ (denoted by $\Gamma_{00}(p^n)^{\rm arith}$ in \cite[\S2.0]{katz-real-analytic-eisenstein}), and we direct the reader to there for a detailed account.
	
	\subsection{Acknowledgment}
	The author is thankful to Pierre Colmez, Yangyu Fan, Mahesh Kakde, Antonio Lei, Yongxiong Li, David Lilienfeldt, Ari Shnidman, Arnaud Vanhaecke, Jan Vonk and Xin Wan for helpful comments/discussions. I owe special thanks to Yongxiong Li for many instructive conversations regarding \cite{rubin99}, to David Savitt for organizing a classic papers reading seminar that led me to understand Katz's work, and to Jan Vonk for sharing his insight and encouraging me to think beyond nine imaginary quadratic fields. Thanks are also due to Barsky for pointing out his paper \cite{barsky77}; I deeply apologize for being completely unaware of and not acknowledging it in my prior works \cite{sum1,sum2}. During the development of the paper, the author is supported by Xin Wan's grant (National Natural Science Foundation of China, E211014301) for conference travels. Last but far from  least, I dedicate this paper to Antonio Lei, without whose kindness and guidance this work would never have come  to life.

	\addtocontents{toc}{\protect\setcounter{tocdepth}{2}}
	\section{Construction of the measure}\label{sec:measure}
	We explain the construction of the measure $\mu_{c,N}$ following Lichtenbaum \cite{lichtenbaum}. 
	
	\subsection{Some elliptic functions}
	\label{subsec:elliptic}
	We start with the complex analysis. Let $\sL\subset\C$ be a lattice, and we put
	\begin{itemize}
		\item $s_k(\sL) = \lim_{s\to 0^+} \sum_{0\ne \mu\in \sL}\mu^{-k}|\mu|^{-2s}$ for $k\in \Z_{\ge 1}$ (the limit exists by \cite[Lemma 3.3]{shimura73}).
		\item $g(\sL) = \pi/\Area(\C/\sL)$.
		\item $\Delta(\sL) = \Omega(\sL)^{-12}\qq\prod_{n\ge 1}(1-\qq^n)^{24}$, if $\sL = \Omega(\sL)\cdot (\Z2\pi i\tau+\Z2\pi i)$ with $\Omega(\sL)\in \C$, $\im(\tau)>0$ and $\qq=\e^{2\pi i\tau}$.
	\end{itemize}
	Let $\theta(z,\sL)$ be the complex function defined by $\Delta(\sL)\e^{-6s_2(\sL)z^2}\sigma^{12}(z,\sL)$, where
	\begin{align*}
		\sigma(z,\sL) = z\prod_{0\ne \mu\in \sL}(1-z/\mu)\e^{z/\mu+(z/\mu)^2/2}.
	\end{align*}
	As $\Delta(\lambda \sL) = \lambda^{-12}\Delta(\sL)$ for $\lambda\in \C^\times$, we have $\theta(\lambda z,\lambda \sL) = \theta(z,\sL)$. For $c\in \Z_{>1}$, denote by $\Theta_c(z,\sL)$ the quotient
	\begin{align}\label{eq:Theta}
		\Theta_c(z,\sL) = \theta(z,\sL)^{c^2}/\theta(cz,\sL),
	\end{align}
	then it is known that $\Theta_c(z,\sL)$ is double periodic with respect to $\sL$ by the product formula \cite[Corollary 2.6.(ii)]{lichtenbaum}
	\begin{align}\label{eq:prod-formula-theta}
		\prod_{0\ne Q\in c^{-1}\sL/\sL} (\wp(z) - \wp(Q))^{-6} = c^{12}\Delta(\sL)^{1-c^2}\Theta_c(z,\sL).
	\end{align}
	Next, for an integer $N\in \Z_{>1}$ prime to $c$, we define
	\begin{align}
		\Lambda_{c,N}(z,\sL) = \prod_{0\ne \rho\in N^{-1}\sL/\sL} \Theta_c(z+\rho, \sL).
	\end{align}
	We will now specialize to $\sL = L$. Denote by $F_{c,N}(z)$ the $z$-expansion of $\frac{d}{dz}\log\Lambda_{c,N}(z,L) = \frac{\Lambda_{c,N}'(z,L)}{\Lambda_{c,N}(z,L)}$ at zero. Fix a basis $\{\kappa_1,\kappa_2\}$ of $N^{-1}L$, so the set $\sR(N) = \{d\cdot \kappa = d_1\kappa_1+d_2\kappa_2\colon d\in \Z^2\setminus \{(0,0)\}, 0\le d_1,d_2<N\}$ simultaneously gives complete sets of representatives of $(N^{-1}L/L)\setminus \{0\}$ and $(N^{-1}c^{-1}L/c^{-1}L)\setminus\{0\}$.
	\begin{prop}\label{prop:z-interpol}
		The power series $F_{c,N}(z)$ admits an expansion in $F[[z]]$:
		\begin{align}\label{eq:z-interpol}
			F_{c,N}(z) = 12 \sum_{k\ge 0}(c^2 - c^{k+1})(1-N^{k+1})\frac{\BH_{k+1}}{(k+1)!} \cdot z^k.
		\end{align}
	\end{prop}
	\begin{proof}
		We adapt the proof of \cite[Lemma 21]{coates-wiles}. As in \textit{loc.~cit.}, the rationality of $F_{c,N}(z)$ can be seen by the $\gal(\bar{\Q}/F)$-invariance of the left hand side of \eqref{eq:prod-formula-theta} with $\sL = L$. Now we turn to \eqref{eq:z-interpol}. From the definition, we can write
		\begin{align}
			F_{c,N}(z) = \sum_{k\ge 0}[c^2d_{k+1}(L)-d_{k+1}(c^{-1}L)]z^k,
		\end{align}
		where for $\sL\in \{L,c^{-1}L\}$,
		\begin{align}
			d_{k+1}(\sL) = \begin{cases}
				12\sum_{\rho\in\sR(N)}[\zeta(\rho,\sL)-s_2(\sL)\rho] &\text{ if }k=0;\\
				-12\sum_{\rho\in\sR(N)}[\wp(\rho,\sL) + s_2(\sL)] & \text{ if }k=1;\\
				-12\sum_{\rho\in\sR(N)}\wp^{(k-1)}(\rho,\sL)/k!&\text{ if }k\ge 2.
			\end{cases}
		\end{align}
		We will show
		\begin{align}\label{eq:d_k-BH}
			c^2d_{k+1}(L) - d_{k+1}(c^{-1}L) = 12(c^2-c^{k+1})(1-N^{k+1})s_{k+1}(L),
		\end{align}
		where recall $s_k(L) =\lim_{s\to 0^+}\sum_{0\ne \mu\in L} \mu^{-k}|\mu|^{-2s}$. Note that \eqref{eq:d_k-BH} is enough for us, since $s_{k+1}(L) = \frac{\BH_{k+1}}{(k+1)!}$ when $k\ge 3$ or $k\in 2\Z_{\ge 0}$. For $k=1$, there is an auxiliary zero caused by $c^2 - c^{k+1}$, so $(c^2 - c^{k+1})s_{k+1}(L) = (c^2-c^{k+1})\BH_{k+1}/(k+1)!$ is always valid.
		
		\textit{Case 1: $k\ge 2$.} Let $\sL\in \{L,c^{-1}L\}$, and put $\Omega_{\infty}(\sL) = \Omega_\infty$ if $\sL = L$ and otherwise $\Omega_\infty/c$. We have
		\begin{align}
			d_{k+1}(\sL) &= -12\sum_{\rho\in\sR(N)} \wp^{(k-1)}(\rho,\sL)/k!\\
			&=-12(-1)^{k-1}
			N^{k+1}\sum_{0\ne \rho\in \sL/N\sL}\sum_{\mu\in N\sL} \frac{1}{(\rho+\mu)^{k+1}},\\
			&=-12(-1)^{k-1}
			N^{k+1}\sum_{\mu\in\sL\setminus N\sL} \frac{1}{\mu^{k+1}}\\
			&=-12(-1)^{k-1}N^{k+1}(1-N^{-k-1})
			\sum_{0\ne \mu\in\sL} \frac{1}{\mu^{k+1}};
		\end{align}
		here the second equality is due to the double-periodicity of $\wp^{(k-1)}(z)$. Since $k\ge 2$, we see that
		\begin{align}
			\sum_{0\ne \mu\in L} \frac{1}{\mu^{k+1}} = \lim_{s\to 0^+}\sum_{0\ne \mu\in L} \frac{1}{\mu^{k+1}|\mu|^{2s}},
		\end{align}
		and the sum is zero except when $w\mid (k+1)$. Hence,
		\begin{align}
			c^2d_{k+1}(L) - d_{k+1}(c^{-1}L) &= -12c^2(N^{k+1}-1)s_{k+1}(L) + 12(N^{k+1}-1)c^{k+1}s_{k+1}(L)\\
			&= 12(1-N^{k+1})(c^2-c^{k+1})s_{k+1}(L).
		\end{align}
		
		\textit{Case 2: k = 1.} For $\sL\in\{L,c^{-1}L\}$, we have
		\begin{align}
			\wp(z,\sL) = \lim_{s\to 0^+} \left[\sum_{\mu\in \sL}\frac{1}{(z-\mu)^2|z-\mu|^{2s}} - \sum_{0\ne \mu\in \sL}\frac{1}{\mu^2|\mu|^{2s}}\right] = \lim_{s\to 0^+} \sum_{\mu\in \sL}\frac{1}{(z-\mu)^2|z-\mu|^{2s}} - s_2(\sL).
		\end{align}
		It follows that
		\begin{align}
			d_2(\sL) = -12\lim_{s\to 0^+}\sum_{0\ne \rho\in N^{-1}\sL/\sL} \sum_{\mu\in \sL} \frac{1}{(\rho+\mu)^2|\rho+\mu|^{2s}}.
		\end{align}
		Repeating the argument in Case 1, we have $d_2(\sL)= 0$ unless $w = 2$, in which case
		\begin{align}
			d_2(\sL) = -12(N^2-1)\lim_{s\to 0^+}\sum_{0\ne \mu\in \sL} \frac{1}{\mu^2|\mu|^{2s}}.
		\end{align}
		One may then proceed as in Case 1.
		
		\textit{Case 3: k=0.} Let $\sL\in \{L,c^{-1}L\}$. Recall for $\mu\in \sL$ (see, e.g., \cite[pp.~84-85]{bsd}),
		\begin{align}
			\zeta(z+\mu,\sL) = \zeta(z,\sL) + s_2(\sL)\mu + g_2(\sL)\bar{\mu}.
		\end{align}
		For $\rho = d_1\kappa_1+d_2\kappa_2\in \sR(N)$, put
		\begin{align}
			\mu(\rho) = \begin{cases}
				N(\kappa_1+\kappa_2) &\text{if }d_1\ne 0,d_2\ne 0;\\
				N\kappa_2 & \text{if }d_1=0;\\
				N\kappa_1 & \text{if }d_2=0.
			\end{cases}
		\end{align}
		Then $\mu(\rho)\in \sL$, and
		\begin{align}
			\zeta(\mu(\rho) - \rho,\sL) - s_2(\sL)[\mu(\rho)- \rho] = \zeta(-\rho) - s_2(\sL)(-\rho) + g_2(\sL)\overline{\mu(\rho)}.
		\end{align}
		Hence
		\begin{align}
			c^2d_1(L) - d_1(c^{-1}L) =& 6 c^2\sum_{\rho\in\sR(N)}[\zeta(\rho,L) + \zeta(-\rho,L) + g_2(L)\overline{\mu(\rho)}]\\
			&-6\sum_{\rho\in\sR(N)}[\zeta(\rho,c^{-1}L) + \zeta(-\rho,c^{-1}L) + g_2(c^{-1}L)\overline{\mu(\rho)}].
		\end{align}
		Since $\zeta(z,\sL)$ is an odd function, and $g_2(c^{-1}L) = c^2g_2(L)$, we see the above quantity vanishes.
	\end{proof}
	In \S\ref{subsec:mu-cN} we will explain the construction of a power series in $\cA_{\hat{E}} = \Our[[t-1]]$ from $F_{c,N}$; once this is done, the $p$-adic Fourier transform (after Amice, Mazur, ...) then yields the $p$-adic measure $\mu_{c,N}$ alluded to in \S\ref{sec:intro}. To proceed with the construction, we will however need some background on the change of variables that will be discussed in \S\ref{subsec:change-of-variable}.
	
	If $f,g$ are two nonzero meromorphic functions whose quotient is a scalar in $\C^\times$, we denote it by $f\sim g$. We conclude this subsection with the following rewrite
	\begin{prop}\label{prop:lambda-lambda}
		We have
		\begin{align}
			\Lambda_{c,N}(z,L)\sim \frac{\Theta_c(N z,L)}{\Theta_c(z,L)}.
		\end{align}
	\end{prop}
	\begin{lem}\label{lem:prod-formula}
		Let $\sL\subset \C$ be a lattice closed under multiplication by $\fo$. Let $h: \C/\sL \to \bP^1(\C)$ be an elliptic function and $\beta\in \fo$ be such that $[\beta]_*\Div(h) = \Div(h)$. Then
		\begin{align}
			\prod_{P\in \beta^{-1}\sL/\sL} h(z+P) \sim h(\beta z).
		\end{align}
	\end{lem}
	\begin{proof}
		Since both sides are elliptic functions on $\C/\sL$, it suffices to show their divisors coincide. Write $\Div(h) = \sum_{a} n_a[a]$, then the divisor of the left hand side is $\sum_{P\in\beta^{-1}\sL/\sL} \sum_a n_a[a+P]$. Next, since $\beta_*\Div(h) = \Div(h)$, for all $a\in \C/\sL$ with $n_a\ne 0$, there exists a unique $a'\in \C/\sL$ with $\beta\cdot a' = a$, and $n_{a'} = n_a$. It follows that the divisor of $h(\beta z)$ is $\sum_{a}\sum_{P\in\beta^{-1}\sL/\sL} n_a[a'+P] = \sum_a\sum_{P\in\beta^{-1}\sL/\sL} n_a[a+P]$.
	\end{proof}
	\begin{proof}[Proof of Proposition \ref{prop:lambda-lambda}]
		We contend that
		\begin{align}
			\Div(\Theta_c(z,L)) = 12c^2[0] - 12\sum_{Q\in c^{-1}L/L}[Q]
		\end{align}
		is preserved under $[N]_*$, since for all $Q\in c^{-1}L/L$, $Q = N(uQ)$ if $u\in \fo$ is such that $uN \equiv 1\bmod c$ ($u$ exists because $\gcd(N,c) =1$). Applying Lemma \ref{lem:prod-formula}, we find
		\begin{align}
			\Lambda_{c,N}(z,L) = \prod_{0\ne \rho\in N^{-1}L/L} \Theta_c(z+\rho,L) \sim \frac{\Theta_c(N z,L)}{\Theta_c(z,L)}.
		\end{align}
	\end{proof}
	
	\subsection{A change-of-variable diagram}
	\label{subsec:change-of-variable}
	Let now $p\ge 5$ be a prime split in $K$, and let $\fp$ (resp.~$\fP$) be the prime above $p$ in $K$ (resp.~$F$) as fixed in \S\ref{sec:intro}. We shall present a triangle diagram \eqref{diag:triangle} of one-parameter formal groups that are isomorphic to each other. To do this, first recall the formal group $\hat{E}$ of $E/\cO_\fP$, which as a formal Lie group corresponds to the completed algebra $\cO_\fP[[w]]$ with the group law given by that of $E$ via the substitution $w = -x/y$ \cite[\S IV.1]{silverman}. We will also consider the formal groups $\hat{\G}_a$ and $\hat{\G}_m$, the group algebras of which will be written as $\cA_{\hat{\G}_a} = \Z_p[[z]]$ and $\cA_{\hat{\G}_m} = \Z_p[[t-1]]$, respectively. Naturally, the identification $w=-x/y$ induces an isomorphism of formal groups $\xi: \hat{\G}_a \to \hat{E}$ over $F_\fP$, which can be described as
	\begin{align}\label{eq:xi}
		\xi^*: F_\fP[[w]] \to F_\fP[[z]],\ w\mapsto -\frac{\wp(z,L)}{\wp'(z,L)/2} = z + O(z^2).
	\end{align}
	Here, that $\xi^*$ is defined over $F_\fP$ is because the $z$-expansions of $\wp(z,L)$ and $\wp'(z,L)$ have coefficients in $F$.
	
	In another direction, since $p$ is split in $K$, $E$ has good ordinary reduction at $\fP$. By \cite[Corollary 4.3.3]{lubin}, we have an isomorphism $\iota: \hat{\G}_m\xrightarrow{\sim}\hat{E}$ over $\Our$. The composition
	\begin{align}
		\varepsilon: \hat{\G}_a \xrightarrow[\ \sim\ ]{\xi} \hat{E} \xrightarrow[\ \sim\ ]{\iota^{-1}} \hat{\G}_m
	\end{align}
	then gives an isomorphism over $\hat{K}^\ur_\fp$, which must be of the form $z\mapsto \e^{\Omega_{p} z}$ for some $\Omega_{p} \in (\hat{K}^\ur_\fp)^\times$.
	\begin{lem}
		The constant $\Omega_{p}\in (\Our)^\times$.
	\end{lem}
	\begin{proof}
		This already appears as \cite[Corollary 4.2]{lichtenbaum}; here we give a streamlined proof. Since $\e^{\Omega_{p} z} -1 = \Omega_{p} z + O(z^2)$, it suffices to show the leading term of $\e^{\Omega_p z}-1 =(\iota^{-1}\xi)^*(t-1)$ has leading term with coefficient in $(\Our)^\times$. Indeed, from construction we have
		\begin{align}
			(\iota^{-1})^*(t-1) = uw + O(w^2)
		\end{align}
		for some $u\in (\Our)^\times$. Then by \eqref{eq:xi}
		\begin{align}
			\xi^*(\iota^{-1})^*(t-1) = -u\frac{2\wp(z)}{\wp'(z)} + O(z^2) = uz + O(z^2),
		\end{align}
		which shows $\Omega_{p} = u$ is in $(\Our)^\times$.
	\end{proof}
	\begin{rem}\label{rem:omega-p}
		We note that $\Omega_p = (\iota^{-1})^*\left(\frac{dt}{t}\right)/\big(\frac{dx}{2y}\big)$. The quickest way to see this is to pull both sides back to $\hat{\G}_a$ via $\xi^*$: The numerator becomes $\frac{d\e^{\Omega_p z}}{\e^{\Omega_p z}} = \Omega_p dz$, whereas the denominator reads $d\wp(z,L)/\wp'(z,L) = dz$.
	\end{rem}
	
	To summarize, we have a diagram of formal groups with arrows being isomorphisms:
	\begin{align}\label{diag:triangle}
		\begin{tikzcd}[ampersand replacement=\&]
			\hat{\G}_a \ar[rr,"\varepsilon:z\mapsto e^{\Omega_{p} z}"] \ar[rd,"\xi"] \& \& \hat{\G}_m \ar[dl,"\iota"]\\
			\& \hat{E} \&
		\end{tikzcd},
	\end{align}
	where $\xi$ is defined over $F_\fP$, $\iota$ is defined over $\Our$, and $\Omega_{p} \in (\Our)^\times$.
	
	\subsection{The measure $\mu_{c,N}$}\label{subsec:mu-cN}
	Let $c,N\in \Z_{>1}$ be coprime integers as in \S\ref{subsec:elliptic}; assume in addition that $\gcd(p,cN)=1$. We start with a brief recall on the $p$-adic Fourier transform; the details can be found in \cite[\S 3]{Hida}. Let $R\subseteq \cO_{\C_p}$ be a $p$-adically complete valuation ring. Then there is an isomorphism 
	\begin{align}
		\cF: \Mes(\Z_p,R) \longrightarrow R[[t-1]],\  \mu\mapsto \sA_{\mu}(t):= \int_{\Z_p} t^x \mu(x) = \sum_{n\ge 0}(t-1)^n\left[\int_{\Z_p} \binom{x}{n}\mu(x)\right].
	\end{align}
	We record the basic properties of $\cF$ here (the proof can be found \textit{loc.~cit.}):
	\begin{prop}\label{prop:fourier-properties}
		Let $\mu\in \Mes(\Z_p,R)$. We have
		\begin{enumerate}
			\item[(i)] For all $a\in \Z_p$ and $n\in \Z_{\ge 0}$,
			\begin{align}
				\mu(a+p^n\Z_p) = \frac{1}{p^n} \sum_{\zeta^{p^n}=1} \sA_\mu(\zeta)\zeta^{-a}.
			\end{align}
			\item[(ii)] For all $n\in \Z_{\ge 0}$, write $x^n\mu$ for the measure with $x^n\mu(S) = \int_S x^n\mu(x)$, where $S\subseteq \Z_p$ is compact open. Then
			\begin{align}
				\sA_{x^n\mu}(t) = \left[\left(t\frac{d}{dt}\right)^n\sA_\mu\right](t).
			\end{align}
		\end{enumerate}
	\end{prop}
	
	We are now set to construct the measure. Recall from Proposition \ref{prop:z-interpol} that $F_{c,N}\in F[[z]]$ is the $z$-expansion of $\frac{d}{dz}\log\Lambda_{c,N}(z,L)$ at zero. Put
	\begin{align}
		f_{c,N}(t) = \frac{1}{12}(\varepsilon^{-1})^*F_{c,N}(z),
	\end{align}
	which \textit{a priori} is in $\hat{K}^\ur_\fp[[t-1]]$.
	\begin{prop}\label{prop:der-interpol}
		We have $f_{c,N}(t) \in \Our[[t-1]]$. Moreover, for all $k\in \Z_{\ge 0}$,
		\begin{align}\label{eq:der-interpol}
			\left.\left(t\frac{d}{dt}\right)^k f_{c,N}(t)\right|_{t=1} =  \Omega_{p}^{-k}(1-N^{k+1})(c^2-c^{k+1})\frac{\BH_{k+1}}{k+1}.
		\end{align}
	\end{prop}
	\begin{proof}
		We first verify \eqref{eq:der-interpol}. By the chain rule, we have
		\begin{align}
			\left.\left(t\frac{d}{dt}\right)^k f_{c,N}(t)\right|_{t=1} = \frac{1}{12\Omega_{p}^k}\frac{d^kF_{c,N}}{dz^k}(0).
		\end{align}
		The equation then follows from Proposition \ref{prop:z-interpol}.

		Now, we turn to the integrality of $f_{c,N}(t)$, or equivalently $12f_{c,N}$ as $p\nmid 6$. Essentially, this can be established by combining \cite[Lemma 23]{coates-wiles} and \cite[Lemma 5]{cassou-nogues-CM}. We record the proof here since our setting is slightly different. To simplify the notation, below we will mostly drop the pullback notation when performing the change of variables according to \eqref{diag:triangle}.
		
		\textbf{Step 1.} \textit{For $0\ne \rho\in N^{-1}L/L$, $\Theta_c(z+\rho,L)^{-1}\in \cO_{\C_p}[[w]]$.}
		By \eqref{eq:prod-formula-theta}, it suffices to show $\wp(z+\rho) - \wp(Q)\in \cO_{\C_p}[[w]]$ for all $0\ne Q\in c^{-1}L/L$, since $\gcd(c,p)=1$ and $\Delta(L)$ is only divisible by primes of bad reduction. By the addition formula \cite[\S20.31]{WW}
		\begin{align}\label{eq:addition-formula}
			\wp(z+\rho) - \wp(Q) = \frac{1}{4}\left(\frac{\wp'(z) - \wp'(\rho)}{\wp(z) - \wp(\rho)}\right)^2 - \wp(z) - \wp(\rho) -  \wp(Q),
		\end{align}
		it is enough to show the integrality of all the terms on the right hand side. Since $cN$ is prime to $p$, we know the $x$-coordinates of $E[cN]$ are $p$-integral \cite[Lemma 7.3.(ii)]{rubin99}, so $\wp(\rho),\wp'(\rho),\wp(Q)\in \cO_{\C_p}$. Next, by \cite[\S IV.1]{silverman}, we have
		\begin{align}\label{eq:wp-w}
			\wp(z) = w^{-2}(1+A_1w+A_2w^2+\cdots)^{-1}, \ \wp'(z) = -2w^{-3}(1+A_1w+A_2w^2+\cdots)^{-1},
		\end{align}
		for some $A_i\in \Z[A,B]$, where $A,B\in \cO$ are the coefficients of the Weierstrass equation $y^2 = x^3 + Ax + B$. As $p\ne 2$, plugging these expansions into \eqref{eq:addition-formula}, we see the right hand side is in $w^{-2}\cO_{\C_p}[[w]]$, therefore in $\cO_{\C_p}[[w]]$ since $\wp(z+\rho)-\wp(Q)$ has not pole at $w = 0$.
		
		\textbf{Step 2.} \textit{$\Lambda_{c,N}(z,L)\in \cO_\fP[[w]]$, and its constant term in the $w$-expansion belongs to $\cO_\fP^\times$.}
		First we note that the constant term of $\Theta_c(z+\rho,L)$ is invertible in $\cO_{\C_p}$ for all $0\ne \rho\in N^{-1}L/L$. Indeed, since $c,N$ are relatively prime, and that $p\nmid cN$, we have $\Theta_c(\rho,L)\in \cO_{\C_p}^\times$ by \cite[Theorem 7.4.(iii)]{rubin99}. It then follows from Step 1 that $\Theta_c(z+\rho,L)\in \cO_{\C_p}[[w]]$, and that its constant term is a unit. Taking the product over $\rho$, we see the same holds for $\Lambda_{c,N}(z,L)$. That $\Lambda_{c,N}(z,L)\in \cO_\fP[[w]] = F[[w]]\cap \cO_{\C_p}[[w]]$ then results from considering the Galois action on \eqref{eq:prod-formula-theta}.
		
		\textbf{Step 3.} \textit{$F_{c,N}(z) = \frac{d}{dz}\log\Lambda_{c,N}(z,L)\in \cO_\fP[[w]]$, and hence is in $\Our[[t-1]]$.} Using the chain rule, we find
		\begin{align}
			(\xi^{-1})^*\left(\frac{d}{dz}\log\Lambda_{c,N}(z,L)\right) = \frac{1}{z'(w)}\frac{d}{dw}\log\left[(\xi^{-1})^*\Lambda_{c,N}(z,L)\right].
		\end{align}
		By Step 2, we have $\frac{d}{dw}\log\left[(\xi^{-1})^*\Lambda_{c,N}(z,L)\right]\in \cO_\fP[[w]]$. Furthermore,
		\begin{align}
			\frac{dz}{dw} = \frac{1}{\wp'(z)}\frac{d\wp(z)}{dw}= \frac{(1+A_1w+\cdots)^{-1} - \frac{w}{2}\frac{d}{dw}[(1+A_1w+\cdots)^{-1}]}{(1+A_1w+\cdots)^{-1}} \in 1+ w\cO_\fP[[w]],
		\end{align}
		where the second equality is by \eqref{eq:wp-w}.
	\end{proof}
	
	Thanks to the above proposition, we may speak of the measure $\mu_{c,N}\in \Mes(\Z_p,\Our)$ which is the Fourier inverse of $f_{c,N}(t)$. The interpolation formula \eqref{eq:interpol} then follows from Proposition \ref{prop:fourier-properties}.(ii) with $t=1$; for this reason, we propose to call $\mu_{c,N}$ the $p$-adic Bernoulli--Hurwitz measure attached to the data of $(E,\iota,c,N)$. When $c,N$ are fixed, we will sometimes write $\mu_{\BH}$ in place of $\mu_{c,N}$.
	
	\section{Computations of the periods}
	\label{sec:computation}
	
	In this section, for all $a\in \Z_p$ and $n\in \Z_{\ge 0}$, we shall relate the period $\mu_{c,N}(a+p^n\Z_p)$ to the special value of some modular form $\Psi_{a,n}$. Let $\fp$ be the prime above $p$ in $K$, and let $e\in \Z_{\ge 1}$ be the order of $\fp$ in the class group $\cl(K)$. Throughout, we will fix a generator $\varpi$ of $\fp^e$; hence $\bar{\fp}^e$ is generated by $\barpi$, the complex conjugate of $\varpi$.
	
	Now, by Proposition \ref{prop:fourier-properties}.(i), $\mu_{c,N}(a+p^n\Z_p)$ is computed by
	\begin{align}\label{eq:pre-period}
		\frac{1}{p^n}\sum_{\zeta^{p^n}=1}f_{c,N}(\zeta)\zeta^{-a}.
	\end{align}
	As in \cite{sum1,sum2}, the method of computation is transcendental, by converting everything into complex numbers.
	
	\subsection{Reduction to $\C/L$}\label{subsec:reduction-torus}
	
	We start with a lemma (\textit{cf.}, \cite[proof of Proposition 5.6]{lichtenbaum})
	\begin{lem}\label{lem:formal-fp}
		The inclusion $\hat{E}(\C_p)\hookrightarrow E(\C_p)$ identifies $\hat{E}[p^\infty](\C_p)$ with $E[\fp^\infty](\C_p)$.
	\end{lem}
	\begin{proof}
		It suffices to prove that $E[\fp^{en}] = E[\varpi^n]$ is contained in the image of $\hat{E}[p^{en}]\subset E[p^{en}]$, which then must be an isomorphism since both groups are of order $p^{en}$. By the main theorem of complex multiplication \cite[Corollary 5.16.(iii)]{rubin99}, modulo $\fP$, the endomorphism $[\varpi]$ of $E$ coincides with the Frobenius up to a unit, and is therefore purely inseparable. It follows that $E[\varpi^n](\C_p)\subset \hat{E}(\C_p)$, thereby contained in $\hat{E}[p^{en}](\C_p)$.
	\end{proof}
	Thanks to Lemma \ref{lem:formal-fp}, we can identify the values of $f_{c,N}(t)\in \cA_{\hat{\G}_m}\hat{\otimes}\Our$ at $\hat{\G}_m[p^n]$ as those of the elliptic function $\frac{1}{12}\frac{d}{dz}\log\Lambda(z,L)$ at $(\C/L)[\fp^n]$. More precisely, we now have a string of group isomorphisms:
	\begin{align}\label{diag:mu-varpi}
		\mu_{p^n}(\bar{\Q}) \xrightarrow{\ \iota\ }
		\hat{E}[p^n](\C_p) =\joinrel= E[\fp^n](\bar{\Q}) \xrightarrow{\ \Xi\ } \fp^{-n}L/L,
	\end{align}
	where $\Xi: E(\C) \to \C/L$ is the Abel--Jacobi map, namely the inverse to the Weierstrass parametrization $z\mapsto (\wp(z),\wp'(z)/2)$. Let $z_\zeta\in \fp^{-n}L/L$ be the point corresponding to $\zeta\in \mu_{p^n}(\bar{\Q})$. Then from Proposition \ref{prop:lambda-lambda}, $f_{c,N}(\zeta) = \lim_{z\to 0} \frac{1}{12} \frac{d}{dz}\log\left(\Theta_c(N z+z_{\zeta^N},L)/\Theta_c(z+z_{\zeta},L)\right)$. As such,
	\begin{align}\label{eq:pre-period-2}
		\frac{1}{p^n}\sum_{\zeta^{p^n}=1} f_{c,N}(\zeta)\zeta^{-a} = \lim_{z\to 0} \frac{1}{12p^n}\sum_{\zeta^{p^n}=1}\left\{-\zeta^{-a}\left[\frac{d}{dz}\log\Theta_c\right](z+z_\zeta,L) + N \zeta^{-a/N}\left[\frac{d}{dz}\log\Theta_c\right](N z+ z_\zeta,L)\right\}.
	\end{align}

	\subsection{Reduction to $\C^\times/q^{\Z}$}
	\label{subsec:Ctimes}
	Our next task is to extend the diagram \eqref{diag:mu-varpi}, for which we need some fresh notation first.
	\begin{lem}
		There exist $\Omega_\infty\in \C^\times$ and a fractional ideal $\fa$ prime to $\fp$ and is of the form $\fa = \Z+\Z\varsigma$ for some $\varsigma\in \cH$, such that $L = \Omega_\infty\fa$.
	\end{lem}
	\begin{proof}
		Choose $x\in L\setminus \fp L$ such that $L/\Z x$ is free, then put $L = \Z x + \Z y$ for some $y\in \C^\times$ with $\im(y/x)>0$. Then $L = xL'$ with $L' = \Z+\Z(y/x)$. Since $\fo\cdot L'\subset L'$, we see that $L'$ is a fractional ideal of $K$. We contend that $L'$ is prime to $\fp$. Indeed, since $\fo\subset L'$, there are $r\in\Z_{\ge0}$ and an integral ideal $\fb$ prime to $\fp$ such that $L' = \fp^{-r}\fb^{-1}$. Thus $x\fb^{-1} = \fp^rL$, and so $x\in \fp^r L$, forcing $r=0$.
	\end{proof}
	Fix now and forever $\Omega_\infty$, $\fa$ and $\varsigma$ as above. Put $q = \e^{2\pi i\varsigma}$, which has complex norm $<1$. Note that for all $n\in \Z_{>0}$, $\fp^{-en}\fa/\fa\simeq \Z/p^{en}$ is generated by $\varpi^{-n}$. After Jacobi, consider the biholomorphic map
	\begin{align}
		\jmath: \C/L \to \C^\times/q^{\Z},\quad z\mapsto \e^{2\pi iz/\Omega_\infty}.
	\end{align}
	The composition
	\begin{align}\label{eq:J-fJ}
		\C^\times/q^{\Z} \xrightarrow{\jmath^{-1}} \C/L \xrightarrow{[\barpi^n]} \C/L
	\end{align}
	induces an isomorphism
	\begin{align}
		\mu_{p^{en}}(\bar{\Q}) \xrightarrow{\jmath^{-1}}  (p^{-en}\Z/\Z)\cdot \Omega_\infty \xrightarrow[\sim]{\times \barpi^n} (\fp^{-en}\fa/\fa)\cdot \Omega_\infty.
	\end{align}
	Fusing it with \eqref{diag:mu-varpi}, we have thus an elongated diagram
	\begin{align}\label{diag:hex}
		\begin{tikzcd}[ampersand replacement=\&, column sep = 0.5pt]
			\&\hat{E}[p^{en}](\C_p) \ar[rd,equal] \&\\
			\mu_{p^{en}}(\bar{\Q}) \ar[ru,"\iota"] \ar[d,dashed,"h_n"] \& \& E[\fp^{en}](\bar{\Q}) \ar[d,"\Xi"]\\
			\mu_{p^{en}}(\bar{\Q}) \& \& \ar[ld,"{[\barpi^n]}^{-1}"] (\fp^{-en}\fa/\fa)\cdot \Omega_\infty\\
			\& \ar[lu,"\jmath"] (p^{-en}\Z/\Z)\cdot \Omega_\infty\&
		\end{tikzcd}
	\end{align}
	where $h_n$ is the composition of arrows. The next lemma determines $h_n$ by studying the compatibility of these hexagons when $n$ varies. To state it, let $\varrho\in \Z_p^\times$ be the element corresponding to $\barpi$ under the identification $\Z_p\simeq \fo_\fp$.
	\begin{lem}\label{lem:hex-prism}
		For $n\in \Z_{>0}$, the following diagram is commutative:
		\[
		\begin{tikzcd}[column sep = 3pt]
			&&\mu_{p^{en}}(\bar{\Q})\ar[lld, "\varrho^nh_n"]
			\ar[rrr]\ar[ddd]&&&\hat{E}[p^{en}](\C_p) \ar[rd]\ar[ddd]&\\
			\mu_{p^{en}}(\bar{\Q}) \ar[ddd] &&&&&&E[\fp^{en}](\bar{\Q}) \ar[ddd]\\
			&(p^{-en}\Z/\Z)\cdot\Omega_\infty \ar[lu] &&&(\fp^{-en}\fa/\fa)\cdot \Omega_\infty \ar[lll,crossing over, "\varrho^n{[\barpi^n]}^{-1}" fill=white] \ar[from=rru, crossing over] &&\\
			&&\mu_{p^{e(n+1)}}(\bar{\Q}) \ar[lld, "\varrho^{n+1}h_{n+1}"]\ar[rrr]&&&\hat{E}[p^{e(n+1)}](\C_p) \ar[rd]&\\
			\mu_{p^{e(n+1)}}(\bar{\Q}) &&&&&&E[\fp^{e(n+1)}](\bar{\Q}) \ar[lld]\\
			&(p^{-e(n+1)}\Z/\Z)\cdot\Omega_\infty \ar[lu] \ar[from = uuu, crossing over] &&&(\fp^{-e(n+1)}\fa/\fa)\cdot \Omega_\infty \ar[lll, "\varrho^{n+1}{[\barpi^{n+1}]}^{-1}"] \ar[from=uuu, crossing over]&&
		\end{tikzcd},
		\]
		where the vertical arrows are natural inclusions, and the unmarked planar arrows are the same as in \eqref{diag:hex}.
	\end{lem}
	\begin{proof}
		The commutativity amounts to verifying that of
		\[
		\begin{tikzcd}[column sep= 2cm]
			p^{-en}\Z/\Z \ar[r,"\varrho^{-n}{[\barpi^n]}"] \ar[d] &  \fp^{-en}\fa/\fa \ar[d]\\
			p^{-e(n+1)}\Z/\Z \ar[r,"\varrho^{-n-1}{[\barpi^{n+1}]}"] & 
			\fp^{-e(n+1)}\fa/\fa
		\end{tikzcd}.
		\]
		That is, for all $i\in \Z/p^n$,
		\begin{align}
			\barpi^n\varrho^{-n} \frac{i}{p^{en}}\equiv \barpi^{n+1}\varrho^{-n-1}\frac{p^e i}{p^{en+e}} \pmod \fa.
		\end{align}
		This in turn follows from the definition of $\varrho$.
	\end{proof}
	As a consequence to Lemma \ref{lem:hex-prism}, replacing $\iota$ by $u\iota$ for some $u\in\Z_p^\times$ if necessarily, we may and will assume from now on that
	\begin{align}\label{ass:iota}
		\varrho^n h_n = \id\quad\text{for all }n\in\Z_{>0}.
	\end{align}
	\begin{rem}\label{rem:pi_n}
		Chasing the diagram, we see that assumption \eqref{ass:iota} gives a sequence of generators $\pi_n = \iota(\zeta_n)$ of $\fp^{-n}\fa/\fa$ for $n\in \Z_{\ge 0}$, such that
		\begin{enumerate}
			\item $\pi_{en} = \varrho^{-n}\varpi^{-n}$;
			\item $p\pi_{n+1}\equiv \pi_n \bmod \fa$.
		\end{enumerate}
	\end{rem}
	Keeping \eqref{eq:J-fJ} in mind, put
	\begin{align}
		\fJ_n = \frac{1}{12} ([\barpi^n]\circ\jmath^{-1})^*\left[\frac{d}{dz}\log\Theta_c(z,L)\right],
	\end{align}
	which we regard as a $q$-periodic meromorphic function on $\C^\times$. Let $X = e^{2\pi iz/\Omega_\infty}$ be the variable of $\C^\times$, so that if $f(z)$ is an elliptic function on $\C/L$, $(\jmath^*f)(X) = f(z)$. By \eqref{ass:iota} and replacing $n$ by $en$, \eqref{eq:pre-period-2} reads
	\begin{align}\label{eq:pre-epf}
		\begin{split}
			&\lim_{z\to 0} \frac{1}{12p^{en}}\sum_{\zeta^{p^{en}}=1}\left\{-\zeta^{-a}\left[\frac{d}{dz}\log\Theta_c\right](\barpi^n z+z_\zeta,L) + N \zeta^{-a/N}\left[\frac{d}{dz}\log\Theta_c\right](N \barpi^n z+ z_\zeta,L)\right\}\\
			=&\lim_{X\to 1} \frac{1}{p^{en}}\sum_{\zeta^{p^{en}}=1}\left[-\zeta^{-a}\fJ_n(X\zeta^{\varrho^{-n}}) + N\zeta^{-a/N}\fJ_n(X^{N}\zeta^{\varrho^{-n}})\right].
		\end{split}
	\end{align}
	
	\subsection{Explicit period formulas}\label{subsec:epf}
	
	Our task now is to explicitly determine the value of \eqref{eq:pre-epf}. Put now
	\begin{align}
		\fG_n(X) = \left[([\barpi^n]\jmath^{-1})^*\Theta_c\right](X) = \Theta_c\left(\barpi^n\frac{\Omega_\infty}{2\pi i}\log X,L\right),
	\end{align}
	which is a $q$-periodic function on $\C^\times$. By Lemma \ref{lem:prod-formula}, we have $\fG_n(X) \sim \prod_{v} \fG(vX)$, with $\fG(X)= \Theta_c((\Omega_\infty/2\pi i)\log X,L)$ and the product is taken over
	\begin{align}\label{eq:Rn}
		S_n = \left\{v\in \C^\times\colon v^{\barpi^n}\in q^{\Z} \right\}/q^{\Z} = 
		\left\{v\in \C^\times\colon v^{\alpha}\in q^{\Z}\text{ for all }\alpha\in\bar{\fp}^{en}=\barpi^n\fo \right\}/q^{\Z};
	\end{align}
	this makes sense since $\fG(vqX) = \fG(vX)$. As in \cite[\S2]{roquette}, consider the function on $\C^\times$:
	\begin{align}\label{eq:fH}
		\fH(X) = (1-1/X)\prod_{m\ge 1} (1-q^mX)(1-q^m/X).
	\end{align}
	\begin{lem}
		For all $v\in \C^\times$, we have
		\begin{align}
			\fG(vX) \sim \left[X^{c(c-1)/2}\frac{\fH(vX)^{c^2}}{\fH(v^cX^c)}\right]^{12}.
		\end{align}
	\end{lem}
	\begin{proof}
		Since both sides have the same divisors in the annulus $\{q\le |X|< 1\}$, it suffices to show that the right hand side is $q$-periodic. Write $\fH_v(X) = \fH(vX)$. It is readily verified that $\fH(X/q) = -X\fH(X)$. So
		\begin{align}
			\frac{\fH_v(X/q)^{c^2}}{\fH_{v^c}(X^c/q^c)} = \frac{(-vX)^{c^2}\fH_v(X)^{c^2}}{(-v^cX^c/q^{c-1})\cdots(-v^cX^c)\fH_{v^c}(X^c)} = q^{c(c-1)/2}\frac{\fH_v(X)^{c^2}}{\fH_{v^c}(X^c)}.
		\end{align}
		From this the $q$-periodicity follows.
	\end{proof}
	
	Let $D$ be the operator $X\frac{d}{dX}\log$. Then it can be directly verified that $([\barpi^n]\jmath^{-1})^*(\frac{d}{dz}\log \Theta_c) = \frac{2\pi i}{\Omega_\infty\barpi^n}\sum_v D\fG(vX)$. Thus
	\begin{align}\label{eq:J-H}
		\fJ_n(X) = \frac{2\pi i c}{\Omega_\infty\barpi^n}\sum_{v\in S_n} \left[\frac{c-1}{2} + cD\fH(vX) - D\fH(v^cX^c)\right].
	\end{align}
	For $t\in \Z_p$, denote by $\fF_{t,n}(X)$ the average $\frac{1}{p^n}\sum_{\zeta^{p^n}=1}D\fH(X\zeta)\zeta^{-t}$.
	\begin{prop}\label{prop:Ft}
		Let $t\in \Z_p$ and $n\in \Z_{\ge 0}$. We have
		\begin{align}
			\fF_{t,n}(X) = \frac{1}{p^n}\sum_{\zeta^{p^n}=1}D\fH(X\zeta)\zeta^{-t} = \frac{X^{t^\flat}}{X^{p^n}-1} + \sum_{m\ge 1}\left[\frac{(q^mX)^{t^\sharp}}{(q^mX)^{p^n}-1} + \frac{(q^m/X)^{(-t)^\sharp}}{1-(q^m/X)^{p^n}}\right],
		\end{align}
		where $t^\flat = t^\flat_{p^n}$ (resp.~$t^\sharp = t^\sharp_{p^n}$) is the unique integer in $[0,p^n)$ (resp.~$(0,p^n]$) such that $t^\flat \equiv t^\sharp \equiv t \bmod p^n$.
	\end{prop}
	
	\begin{proof}
		This can be directly verified from \eqref{eq:fH} and the formula (see, e.g., \cite[proof of Lemma 3.1]{sum1})
		\begin{align}
			\frac{1}{p^n}\sum_{\zeta^{p^n}=1} \frac{1}{1-y\zeta}\zeta^{-t} = \frac{y^{t^\flat}}{1-y^{p^n}}.
		\end{align}
	\end{proof}
	\begin{lem}\label{lem:qn}
		For all $n\in \Z_{\ge 0}$, there exists a unique $q_n\in \C^\times$ with $q_n^{\alpha}\in q^{\Z}$ for all $\alpha\in \fp^n$ and $q_n^{p^n} = q$. In particular, for $0\le i<p^n$, $|q|<|q_n^i|\le 1$.
	\end{lem}
	\begin{proof}
		It suffices to construct the corresponding element in $\C/\fa$. Consider $z_n = \frac{\varsigma + (-\varsigma)^\flat_{\fp^n}}{p^n}$. Here $(-\varsigma)^\flat_{\fp^n}$ is the unique integer in $[0,p^n)$ that is congruent to $-\varsigma$ modulo $\fp^n\fa$, which is possible since the natural map $\Z/p^n\to \fa/\fa\fp^n$ is a bijection. Now, by construction, $z_n\bar{\fp}^n\subseteq\fo$ and $z_np^n \in \varsigma + \Z$. So we take $q_n = \e^{2\pi i z_n}$.
	\end{proof}	
	Thanks to Lemma \ref{lem:qn}, we can consider a particular set of $v$'s. Write $\sR_{n} = \{q_n^i\colon 0\le i<p^n\}$ for $q_n$ as above, so $\sR_{en}$ gives a specific set of representatives of the set $S_n$ in \eqref{eq:Rn}, all of whose elements are in the annulus $\{|q|<u\le 1\}$. We are now ready to compute the average in \eqref{eq:pre-epf}.
	
	\begin{thm}\label{thm:epf}
		Let $n\in \Z_{\ge 0}$ and $a\in \Z_p$. Consider the divisor function $\sigma_{a,n}(m) = \sigma_{a,n,c,N}(m)$ defined by (all the congruences here are modulo $p^n$)
		\begin{align}
			\begin{cases}
				-\frac{c(c-1)(N-1)}{2}\1_{a\not\equiv 0} + \frac{-c^2 a^\flat_{p^n} + c^2 N(a/N)^\flat_{p^n} + c(a/c)^\flat_{p^n} - cN(a/cN)^\flat_{p^n}}{p^n}, & m = 0;\\
				c^2\sum_{\substack{d\mid m\\ d\equiv a}}\sgn(d) - c^2N \sum_{\substack{d\mid m\\ d\equiv a/N}}\sgn(d) - 
				c\sum_{\substack{d\mid m\\ d\equiv a/c}}\sgn(d)
				+cN \sum_{\substack{d\mid m\\ d\equiv a/cN}}\sgn(d), & m\ge 1.
			\end{cases}
		\end{align}
		Then, there exists a modular form $\Psi_{a,n}$ of weight $1$ and level $\Gamma^1(p^n)$ defined over $\Z$, such that for $\tau\in \cH$, we have the expansion
		\begin{align}\label{eq:q-exp}
			\Psi_{a,n}(\tau) = (2\pi i)\Psi_{a,n}(\C/2\pi i(\Z\tau+\Z),2\pi i\tau/p^n,dz) = 2\pi i\sum_{m\ge 0} \sigma_{a,n}(m)\e^{2\pi im\tau/p^n}.
		\end{align}
		Moreover, let $n'\ge n$ be the smallest integer divisible by $e$ and, for $m\in \Z_{\ge 0}$, write $\cO_{\fp,m}$ for the ring of integers in $F_\fP(E[\bar{\fp}^{m}])$. Then
		\begin{align}\label{eq:epf}
			\mu_{c,N}(a\varrho^{-n'/e}+p^n\Z_p) = \Omega_\infty^{-1}\barpi^{-n'/e} \Psi_{a,n}\left(\frac{\varsigma+(-\varsigma)^\flat_{\fp^{n'}}}{p^{n'-n}}\right) \in \cO_{\fp,n}.
		\end{align}
	\end{thm}
	
	\begin{lem}\label{lem:vc-reduction}
		For $n\in\Z_{\ge 0}$ and $\lambda\in \Z$, we have
		\begin{align}
			D\fH(q_n^{\lambda c} X) = D\fH(q_n^{(\lambda c)^\flat_{p^n}}X) - \frac{\lambda c - (\lambda c)^\flat_{p^n}}{p^n}.
		\end{align}
		Therefore,
		\begin{align}\label{eq:J-H-2}
			\fJ_n(X) = \frac{2\pi i c}{\Omega_\infty\barpi^n}\left[ \frac{(c-1)(2p^{en}-1)}{2} + c\sum_{v\in \sR_{en}} D\fH(vX) - \sum_{v\in \sR_{en}} D\fH(vX^c) \right].
		\end{align}
	\end{lem}
	\begin{proof}
		As mentioned earlier, $\fH(vX/q) = -vX\fH(vX)$. So $D\fH(vX/q) = D\fH(vX) + 1$. The rest follows from \eqref{eq:J-H}.
	\end{proof}
	\begin{lem}\label{lem:Ft2}
		Let $t\in \Z_p$. Then
		\begin{enumerate}
			\item[(I)]
			\begin{align}
				\sum_{1\ne v\in \sR_n} \fF_{t,n}(v) = (1-p^n)\1_{t\equiv 0\bmod p^n} -\sum_{m\ge 1} q_n^m\sum_{\substack{d\mid m\\ d\equiv t\bmod p^n\\ m/d\not\equiv 0\bmod p^n}} \sgn(d).
			\end{align}
			\item[(II)] We have $\lim_{X\to 1}\left[\fF_{t,n}(X) - N \fF_{t/N,n}(X^N)\right]$ is equal to
			\begin{align}
				\frac{N-1}{2} + \frac{t^\flat_{p^n} - N(t/N)^\flat_{p^n}}{p^n} - \sum_{m\ge 1}q_n^{m}\sum_{\substack{d\mid m\\ d\equiv t\bmod p^n\\ m/d\equiv 0\bmod p^n}}\sgn(d) 
				+N \sum_{m\ge 1}q_n^{m}\sum_{\substack{d\mid m\\ d\equiv t/N\bmod p^n\\ m/d\equiv 0\bmod p^n}}\sgn(d).
			\end{align}
		\end{enumerate}
	\end{lem}
	\begin{proof}
		Starting from Proposition \ref{prop:Ft}, we find $\sum_{1\ne v\in \sR_n}\fF_{t,n}(v)$ equates
		\begin{align}
			&\sum_{1\ne v\in \sR_n}\left\{-\1_{t\in p^n\Z_p} +  \frac{v^{t^\sharp}}{v^{p^n}-1} + \sum_{m\ge 1}\left[\frac{(q^mv)^{t^\sharp}}{(q^mv)^{p^n}-1} + \frac{(q^m/v)^{(-t)^\sharp}}{1-(q^m/v)^{p^n}}\right] \right\}\\
			=& (1-p^n)\1_{t\in p^n\Z_p} + \sum_{1\le \lambda<p^n}\left\{
			-\sum_{l\ge 0} q_n^{\lambda(t^\sharp+lp^n)}
			-\sum_{m\ge 1}\sum_{l\ge 0} q_n^{(\lambda+mp^n)(t^\sharp+lp^n)}
			+\sum_{m\ge 1}\sum_{l\ge 0}
			q_n^{(mp^n-\lambda)((-t)^\sharp+lp^n)} \right\}\\
			=& (1-p^n)\1_{t\in p^n\Z_p} - \sum_{\substack{d>0\\ d\equiv t\bmod p^n}}\sum_{\substack{m\ge 1\\ m\not\equiv 0\bmod p^n}} q_n^{dm} + \sum_{\substack{d>0\\ d\equiv -t\bmod p^n}}\sum_{\substack{m\ge 1\\ m\not\equiv 0\bmod p^n}} q_n^{dm}.
		\end{align}
		
		As for (II), by Proposition \ref{prop:Ft}, we have
		\begin{align}\label{eq:v=1}
			\begin{split}
				&\lim_{X\to 1}\left[\fF_{t,n}(X) - N \fF_{t/N,n}(X^N)\right]\\
				=& \lim_{X\to 1}\left[\frac{X^{t^\flat}}{X^{p^n}-1} - N \frac{X^{N(t/N)^\flat}}{X^{N p^n}-1}\right] + \sum_{m\ge 1}
				\left[ \frac{q^{mt^\sharp}}{q^{mp^n}-1} + \frac{q^{m(-t)^\sharp}}{1-q^{mp^n}} -N \frac{q^{m(t/N)^\sharp}}{q^{mp^n}-1} - N\frac{q^{m(-t/N)^\sharp}}{1-q^{mp^n}}\right].
			\end{split}
		\end{align}
		The limit of the last equation is
		\begin{align}
			\lim_{X\to 1} \left[\frac{\sum_{1\le j<N}(X^{jp^n}-1)}{X^{N p^n}-1} + \frac{X^{t^\flat}-1}{X^{p^n}-1} - N\frac{X^{N(t/N)^\flat}-1}{X^{N p^n}-1}\right]
			= \frac{N-1}{2} + \frac{t^\flat}{p^n} - \frac{N(t/N)^\flat}{p^n}.
		\end{align}
		By a computation similar to that of (I), the rest of \eqref{eq:v=1} is
		\begin{align}
			&-\sum_{m\ge 1}q^m\sum_{\substack{d\mid m\\ d\equiv t\bmod p^n}}\sgn(d) + N \sum_{m\ge 1}q^m\sum_{\substack{d\mid m\\ d\equiv t/N\bmod p^n}}\sgn(d)\\
			=&-\sum_{m\ge 1}q_n^m\sum_{\substack{d\mid m\\ d\equiv t\bmod p^n\\m/d\equiv 0\bmod p^n}}\sgn(d) + N \sum_{m\ge 1}q_n^m\sum_{\substack{d\mid m\\ d\equiv t/N\bmod p^n\\m/d\equiv 0\bmod p^n}}\sgn(d).
		\end{align}
	\end{proof}
	\begin{proof}[Proof of Theorem \ref{thm:epf}]
		We first show that for $a\in \Z_p$ and $n\in \Z_{\ge 0}$,
		\begin{align}\label{eq:preepf-e}
			\mu_{c,N}(a\varrho^{-n}+p^{en}\Z_p) = \frac{2\pi i}{\Omega_\infty\barpi^n}\sum_{m\ge 0}\sigma_{a,en}(m)q_{en}^{m}.
		\end{align}
		Combining \eqref{eq:pre-period}, \eqref{eq:pre-period-2}, \eqref{eq:pre-epf}, \eqref{eq:J-H-2}, $\mu_{c,N}(a\varrho^{-n}+p^{en}\Z_p)$ is equal to the limit at $X=1$ of
		\begin{align}
			\begin{split}
				\frac{2\pi ic}{\Omega_\infty\barpi^n}&\Bigg[\frac{(N-1)(c-1)(2p^{en}-1)}{2}\1_{a\in p^{en}\Z_p}
				- \frac{c}{p^{en}}\sum_{\zeta^{p^{en}}=1}\sum_{v\in \sR_{en}}D\fH(vX\zeta)\zeta^{-a} + \frac{1}{p^{en}}\sum_{\zeta^{p^{en}}=1}\sum_{v\in \sR_{en}}D\fH(vX^c\zeta^c)\zeta^{-a}\\
				&+\frac{cN}{p^{en}}\sum_{\zeta^{p^{en}}=1}\sum_{v\in \sR_{en}}D\fH(vX^N\zeta)\zeta^{-a/N} - \frac{N}{p^{en}}\sum_{\zeta^{p^{en}}=1}\sum_{v\in \sR_{en}} D\fH(vX^{N c}\zeta^c)\zeta^{-a/N}\Bigg],
			\end{split}
		\end{align}
		which we arrange to
		\begin{align}
			\frac{2\pi ic}{\Omega_\infty\barpi^n}&\Bigg[\frac{(N-1)(c-1)(2p^{en}-1)}{2}\1_{a\in p^{en}\Z_p}\\
			&- c\sum_{v\in \sR_{en}}\fF_{a,en}(vX) + \sum_{v\in \sR_{en}}\fF_{a/c,en}(vX^c) + cN \sum_{v\in \sR_{en}}\fF_{a/N,en}(vX^N) - N\sum_{v\in \sR_{en}}\fF_{a/(cN),en}(vX^{cN})\Bigg].
		\end{align}
		By Lemma \ref{lem:Ft2}, we find
		\begin{align}
			&\lim_{X\to 1}\sum_{v\in \sR_{en}}\left[-c\fF_{a,en}(vX) + \fF_{a/c,en}(vX^c) + cN \fF_{a/N,en}(vX^N) -N \fF_{a/(cN),en}(vX^{cN})\right]\\
			=&(c-1)(N-1)(1-p^{en})\1_{p^{en}|a} +\frac{(1-c)(N-1)}{2} + \frac{-c a^\flat_{p^{en}} + cN(a/N)^\flat_{p^{en}} + (a/c)^\flat_{p^{en}} - N(a/cN)^\flat_{p^{en}}}{p^{en}} + \sum_{m\ge 1} q_{en}^m \sigma_{a,en}(m).
		\end{align}
		Putting everything together, we then get \eqref{eq:preepf-e}. Let now $n\in \Z_{\ge 0}$ and $n'\ge n$ be the smallest integer divisible by $e$. Note that for all $m\in \Z_{\ge 0}$ and $t\ge 0$,
		\begin{align}
			\sum_{0\le b<p}\sigma_{a+bp^t,t+1}(m) = \sigma_{a,t}(m);
		\end{align}
		for $m=0$ this can be verified by the measure interpretation from Appendix \ref{app:bernoulli}. As such, we find
		\begin{align}
			\mu_{c,N}(a\varrho^{-n'/e}+p^n\Z_p) &= \sum_{a'\equiv a\bmod p^{n'}} \mu_{c,N}(a'\varrho^{-n'/e}+p^{n'}\Z_p)\\
			&=\frac{2\pi i}{\Omega_\infty\barpi^{n'/e}}\sum_{m\ge 0} \sigma_{a,n}(m)\e^{2\pi im(\varsigma + (-\varsigma)^\flat_{\fp^{n'}})/p^{n'}}.
		\end{align}
		We emphasize that this is a complex identity and is made possible since the left hand side can be regarded as $\C$-valued, as explained in \S\ref{subsec:reduction-torus}.
		
		In \S\ref{sec:modularity} we will show the modularity of $\Psi_{a,n}$ as well as \eqref{eq:q-exp}. Assume them for now, and we consider the question of rationality. As the $\qq$-expansion is integral, we see that $\Psi_{a,n}$ is defined over $\Z$. As in \S\ref{subsec:conventions-modform}, thinking of it as a function on the moduli space of triples $(\EE,P,\omega)$ where $\EE$ is an elliptic curve defined over some algebra $R$, $P$ is a point of $E(R)$ of exact order $p^n$, and $\omega\in H^0(\EE,\Omega^1_{\EE/R})$ is a basis. Then, writing $z_{n'} = (\varsigma+(-\varsigma)^\flat_{\fp^{n'}})/p^{n'}$ and $L' = (\Z\cdot p^n z_{n'}+\Z)\cdot\Omega_\infty$, we see that
		\begin{align}
			\Psi_{a,n}(\C/L',\Omega_{\infty}\cdot z_{n'},dz) = \Omega_\infty^{-1}\Psi_{a,n}\left(\C/(\Z\cdot p^nz_{n'}+\Z), z_{n'},dz\right) = \Omega_\infty^{-1}\Psi_{a,n}\left(\frac{\varsigma+(-\varsigma)^\flat_{\fp^{n'}}}{p^{n'-n}}\right).
		\end{align}
		Since both $\Psi_{a,n}$ and $(\C/L',\Omega_{\infty}\cdot z_{n'},dz)$ are defined over $F_\fP(E[\bar{\fp}^{n'}])$, we see that $\Omega_\infty^{-1}\Psi_{a,n}(p^n z_{n'})$ is valued in the same field, and thus $\cO_{\fp,n'}$ since it is integral. We will prove the stronger rationality $\mu_\BH(a+p^n\Z_p) \in \cO_{\fP,n}$ in Proposition \ref{prop:Psi-Phi}.
	\end{proof}

	
	\section{Modularity of the $q$-expansions}\label{sec:modularity}
	
	We will now study the automorphic property of $\Psi_{a,n}$ for $a\in \Z_p$ and $n\in \Z_{\ge 0}$; in fact, we will also show in Proposition \ref{prop:addition} below that the distribution relation of $\mu_{c,N}$ can be directly deduced from the modularity. As in \S\ref{subsec:notation}, denote by $\zeta = \zeta_n\in \mu_{p^n}(\bar{\Q})$ the primitive $p^n$-th root of unity corresponding to $\e^{2\pi i/p^n}$ under the fixed embedding $\bar{\Q}\to \C$, so that $\zeta_{n+1}^p = \zeta_n$ for all $n\in \Z_{\ge 0}$. Following the conventions of \cite{kubert-lang}, we start by defining $\Psi_{a,n}$ as a function on oriented tuples $(\omega_1,\omega_2)$ (treated as column vectors) with $\im(\omega_1/\omega_2) > 0$. In fact we do a bit more, for $k\in \Z_{\ge 0}$, put
	\begin{align}
		\Psi^{(k)}_{a,n}(\omega_1,\omega_2) = \frac{1}{12p^n}\sum_{0\le l_1,l_2<p^n} \frac{d^{k+1}}{dz^{k+1}}\log\Lambda_{c,N}(l_1\omega_1/p^n + l_2\omega_2/p^n, \Z\omega_1+\Z\omega_2)\zeta^{-al_2}.
	\end{align}
	Later in Proposition \ref{prop:q-expansion}, we shall see that $\Psi_{a,n}^{(0)}$ recovers $\Psi_{a,n}$. Note that by definition, $\Psi^{(k)}_{a,n} = \Psi^{(k)}_{a',n}$ whenever $a\equiv a'\bmod p^n$. 
	
	We start by verifying the three laws listed on p.~26 \textit{op.~cit}. 
	\begin{prop}\label{prop:modularity}
		The function $\Psi^{(k)}_{a,n}: \{\umega = (\omega_1,\omega_2)\in \C^2 \colon \im(\omega_1/\omega_2)>0\} \to \C$ enjoys the following properties:
		\begin{enumerate}
			\item[(MF1)] For all $\lambda\in \C^\times$,
			\begin{align}
				\Psi^{(k)}_{a,n}(\lambda\umega) = \lambda^{-k-1}\Psi^{(k)}_{a,n}(\umega).
			\end{align}
			
			\item[(MF2)]
			For all $u\in (\Z/p^n)^\times$ and $\gamma\in \SL_2(\Z)$ with $\gamma\equiv \mat{u}{0}{*}{u^{-1}}\bmod p^n$,
			\begin{align}
				\Psi^{(k)}_{a,n}(\gamma \umega) = \Psi^{(k)}_{ua,n}(\umega).
			\end{align}
			
			\item[(MF3)] For all $\gamma \in \SL_2(\Z)$, there exists $a_m\in \C$ for $m\in \Z_{\ge 0}$ with $|a_m| = O(m^k)$, such that for all $\tau\in \cH$,
			\begin{align}
				\Psi^{(k)}_{a,n}(\gamma (\tau,1)) = \sum_{m\ge 0} a_m \e^{2\pi im\tau/p^n}.
			\end{align}
		\end{enumerate}
	\end{prop}
	\begin{proof}
		Recall that, for any lattice $\sL\subset \C$ and $\lambda\in \C^\times$, $\theta(\lambda z,\lambda \sL) = \theta(z,\sL)$, and thus $\Lambda_{c,N}(\lambda z,\lambda \sL) = \Lambda_{c,N}(z,\sL)$. Taking the logarithm and the $(k+1)$-st derivative, we get (MF1). 
		
		Next, we turn to (MF2). Write $\gamma = \mat{a}{b}{c}{d}\in \SL_2(\Z)$ with $\gamma \equiv \mat{u}{0}{*}{u^{-1}}$ and $\sL = \Z\omega_1+\Z\omega_2$. Then
		\begin{align}
			\Psi^{(k)}_{a,n}(\gamma\umega) &= \frac{1}{12p^n}\sum_{0\le l_1,l_2<p^n}
			\frac{d^{k+1}}{dz^{k+1}}\log\Lambda_{c,N}\left(\frac{l_1(a\omega_1+b\omega_2)}{p^n} + \frac{l_2(c\omega_1+d\omega_2)}{p^n},\sL\right)\zeta^{-al_2}\\
			&=\frac{1}{12p^n}\sum_{0\le l_1,l_2<p^n}
			\frac{d^{k+1}}{dz^{k+1}}\log\Lambda_{c,N}\left(\frac{(l_1a+l_2c)\omega_1}{p^n} + (u^{-1})^\flat_{p^n} \frac{l_2\omega_2}{p^n},\sL\right)\zeta^{-al_2}\\
			&=\frac{1}{12p^n}\sum_{0\le l_1,l_2<p^n}
			\frac{d^{k+1}}{dz^{k+1}}\log\Lambda_{c,N}\left(\frac{l_1\omega_1}{p^n} + \frac{l_2\omega_2}{p^n},\sL\right)\zeta^{-ual_2},
		\end{align}
		where in the second equality we used the facts $d-u^{-1}\equiv b\equiv 0\bmod p^n$ and that $\Lambda_{c,N}$ is $\sL$-periodic.
		
		Finally, we deal with (MF3), which boils down to computations of the $q$-expansions essentially contained in \S\ref{sec:computation}. Let $\sL\subset \C$ be the lattice generated by $1$ and $\tau\in \cH$. With $X = \e^{2\pi iz}$ and $\qq = \e^{2\pi i \tau}$, following the same reasoning in \S\ref{subsec:epf}, we see that
		\begin{align}
			\Lambda_{c,N}(z,\sL) \sim \frac{\fG(X^N,\qq)}{\fG(X,\qq)} \sim \left[X^{(N-1)c(c-1)/2}\frac{\fH(X^N,\qq)^{c^2}\fH(X^c,\qq)}{\fH(X^{cN},\qq)\fH(X,\qq)^{c^2}}\right]^{12},
		\end{align}
		where
		\begin{align}
			\fG(X,\qq) &= \Theta_c(z,\sL);\\
			\fH(X,\qq) &= (1-1/X)\prod_{m\ge 1}(1-\qq^mX)(1-\qq^m/X).
		\end{align}
		It follows that
		\begin{align}\label{eq:Lambda-DH}
			\frac{1}{12}\frac{d}{dz}\log\Lambda_{c,N}(z,\sL) = 
			2\pi i\left[\frac{(N-1)(c^2-c)}{2} -c^2 D\fH(X,\qq) + c^2ND\fH(X^N,\qq) + cD\fH(X^c,\qq) - cND\fH(X^{cN},\qq)\right],
		\end{align}
		and if $k\in\Z_{>1}$,
		\begin{align}\label{eq:Lambda-DH2}
			\frac{1}{12}\frac{d^k}{dz^k}\log\Lambda_{c,N}(z,\sL) = (2\pi i)^k\left[-c^2 D^k\fH(X,\qq) + c^2N^kD^k\fH(X^N,\qq) + c^kD^k\fH(X^c,\qq) - (cN)^kD^k\fH(X^{cN},\qq)\right].
		\end{align}
		Thus (MF3) can be seen by virtue of the lemma below.
	\end{proof}
	\begin{lem}
		Suppose $\uu\in \C^\times$ is such that $\uu^{p^n}\in \qq^{\Z}$, then for all $k\in\Z_{> 0}$, 
		\begin{align}
			D^k\fH(\uu,\qq) - N^kD^k\fH(\uu^N,\qq) = \sum_{m\ge 0} a_m \qq^{m/p^n} 
		\end{align}
		for some $a_m\in \C$, $m\in \Z_{\ge 0}$ and $|a_m| = O(m^k)$. Here $\qq^{1/p^n} = \e^{2\pi i\tau/p^n}$.
	\end{lem}
	\begin{proof}
		We can write $\uu = \qq^{s/p^n}\zeta^t$ for some $s,t\in \Z$ and $\zeta$ the fixed primitive $p^n$-th root of unity. Since $D\fH(\uu/\qq,\qq) = 1 + D\fH(\uu,\qq)$, it suffices to work with $0\le s,t< p^n$; in particular $|\qq|<|\uu|\le 1$. Now, suppose $|\qq|<|X|\le 1$, then
		\begin{align}
			D^k\fH(X,\qq) &= D^{k-1}\Big[\frac{1}{X-1} + \sum_{m\ge 1} \frac{\qq^m X}{\qq^m X-1} + \frac{\qq^m/X}{1-\qq^m/X}\Big]\\
			&= D^{k-1}\Big[\frac{1}{X-1}\Big] - \sum_{m,l\ge 1} l^{k-1}\qq^{ml}X^l + \sum_{m,l\ge 1}(-l)^{k-1}\qq^{ml}X^{-l}.
		\end{align}
		Suppose first that $|\uu|\ne 1$. Plugging in $X = \uu = \qq^{s/p^n}\zeta^t$, we find
		\begin{align}\label{eq:DH-q-expansion}
			D^k\fH(\uu,\qq) &= -\sum_{l\ge 0}l^{k-1}\uu^l - \sum_{m,l\ge 1} l^{k-1}\qq^{ml}\uu^l + \sum_{m,l\ge 1}(-l)^{k-1}\qq^{ml}\uu^{-l}\\
			&= -\sum_{l\ge 0}l^{k-1}\zeta^{tl}\qq^{sl/p^n}
			- \sum_{m,l\ge 1}l^{k-1}\zeta^{tl}\qq^{(p^nml+sl)/p^n}
			+ \sum_{m,l\ge 1}(-l)^{k-1} \zeta^{-tl}\qq^{(p^nm-s)l/p^n}.
		\end{align}
		Hence all the exponents of $\qq$ are nonnegative, and the coefficient of $\qq^{m/p^n}$ is bounded by $O\left(\sum_{d\mid n}d^{k-1}\right) = O(n^k)$, as desired.
		
		When $|\uu|=1$, the expansion \eqref{eq:DH-q-expansion} diverges at the constant term $D^{k-1}\big[\frac{1}{X-1}\big]|_{X=\uu}$. Still, the constant term of $D^k\fH(\uu,\qq) - N^kD^k\fH(\uu^N,\qq)$ is finite. This is because $\frac{1}{X-1} - \frac{N}{X^N-1}$, and thus $D^{k-1}\big[\frac{1}{X-1} - \frac{N}{X^N-1}\big]$, has no pole at $\{\zeta^t\colon 0\le t<p^n\}$.
	\end{proof}
	
	Our next task is to justify the new definition of $\Psi^{(k)}_{a,n}$, which, by our presiding conventions \S\ref{subsec:conventions-modform} and specializing to $k=0$, gives \eqref{eq:q-exp}. To begin with, for $k\in \Z_{\ge 0}$, $a\in \Z_p$ and $n\in \Z_{\ge 0}$, define $\sigma^{(k)}_{a,n}(m) = \sigma^{(k)}_{a,n,c,N}(m)$ to be
	\begin{align}
		c^2\sum_{\substack{d\mid m\\ d\equiv a}} \sgn(d)d^k
		-c^2N^{k+1}\sum_{\substack{d\mid m\\ d\equiv a/N}}\sgn(d)d^k - c^{k+1}\sum_{\substack{d\mid m\\ d\equiv a/c}}\sgn(d)d^k
		+(cN)^{k+1}\sum_{\substack{d\mid m\\ d\equiv a/(cN)}} \sgn(d)d^k
	\end{align}
	if $m\ge 1$, and
	\begin{align}
		\sigma^{(k)}_{a,n}(0) = \int_{a+p^n\Z_p}x^k\mu_{\KL}(x),
	\end{align}
	where $\mu_\KL$ is the $(c,N)$-regularized Kubota--Leopoldt measure determined by the interpolation formula \eqref{eq:interpolation-KL}; see Appendix \ref{app:bernoulli} for a detailed discussion. We note that $\sigma^{(k)}_{a,n}(0)$ is valued in $\Q$ and is $p$-integral. 
	
	\begin{prop}\label{prop:q-expansion}
		For $\tau\in \cH$ and $\qq^{1/p^n} = \e^{2\pi i\tau/p^n}$, we have 
		\begin{align}
			\Psi^{(k)}_{a,n}(2\pi i\tau, 2\pi i) = \sum_{m\ge 0} \sigma^{(k)}_{a,n}(m) \qq^{m/p^n}.
		\end{align}
		In particular, by the $\qq$-expansion principle, $\Psi_{a,n}^{(k)}$ is a modular form over $\Z_{(p)}$.
	\end{prop}
	\begin{proof}
		Write $\sL = \Z\tau+\Z$. We have $\Psi^{(k)}_{a,n}(2\pi i\tau,2\pi i) = (2\pi i)^{-(k+1)}\Psi^{(k)}_{a,n}(\tau,1)$ by (MF1). Denote by $\RR_n$ the set $\{\qq^{j/p^n}:0\le j<p^n\}$. Using \eqref{eq:Lambda-DH}, \eqref{eq:Lambda-DH2} and writing $X = \e^{2\pi iz}$, for $k\in \Z_{\ge 1}$ we find
		\begin{align}
			&\frac{1}{12(2\pi i)^{k+1}}\sum_{0\le l_1<p^n} \frac{d^{k+1}}{dz^{k+1}}\log\Lambda_{c,N}(z+l_1\tau/p^n,\sL)\\
			=&\frac{(N-1)(c^2-c)}{2}p^n \1_{k=0}
			+\sum_{\vv\in\RR_n}\Big[-c^2 D^{k+1}\fH(\vv X,\qq) + c^2N^{k+1}D^{k+1}\fH(\vv^NX^N,\qq) + c^{k+1}D^{k+1}\fH(\vv^cX^c,\qq)\\
			&- (cN)^{k+1}D^{k+1}\fH(\vv^{cN}X^{cN},\qq)\Big]\\
			=&\frac{(N-1)(c^2-c)(2p^n-1)}{2}\1_{k=0}
			+\sum_{\vv\in\RR_n}\Big[-c^2 D^{k+1}\fH(\vv X,\qq) + c^2N^{k+1}D^{k+1}\fH(\vv X^N,\qq) + c^{k+1}D^{k+1}\fH(\vv X^c,\qq)\\
			&- (cN)^{k+1}D^{k+1}\fH(\vv X^{cN},\qq)\Big].
		\end{align}
		Here in the second equality we used a calculation similar to Lemma \ref{lem:vc-reduction}. We may then proceed with the same calculations in the proof of Theorem \ref{thm:epf}. Specifically, it can be shown that
		\begin{enumerate}
			\item Put $\fF^{(k)}_{t,n}(X,\qq) = \frac{1}{p^n}\sum_{\zeta^{p^n}=1}D^{k+1}\fH(\zeta X,\qq)\zeta^{-t}$. Then
			\begin{align}
				\fF_{t,n}^{(k)}(X,\qq) = D^k\Big[\frac{X^{t^\flat}}{X^{p^n}-1}\Big] -\sum_{m\ge 1}\sum_{\substack{l\ge 1\\ l\equiv t}} \qq^{ml}l^k X^l
				+\sum_{m\ge 1}\sum_{\substack{l\ge 1\\ l\equiv -t}}\qq^{ml}X^{-l}(-l)^k.
			\end{align}
			
			\item We have
			\begin{align}
				\sum_{1\ne v\in \RR_n} \fF^{(k)}_{t,n}(v,\qq) = (1-p^n)\1_{t\equiv 0,k=0} - \sum_{m\ge 1}\qq^{m/p^n}\sum_{\substack{d\mid m\\ d\equiv t\bmod p^n\\ m/d\not\equiv 0\bmod p^n}}\sgn(d)d^k
			\end{align}
			and $\lim_{X\to 1^-}\big[\fF_{t,n}^{(k)}(X,\qq) - N^{k+1}\fF^{(k)}_{t/N,n}(X^N,\qq)\big]$ is equal to
			\begin{align}
				\int_{t+p^n\Z_p}x^k\mu_N(x) - \sum_{m\ge 1}\qq^{m/p^n}\sum_{\substack{d\mid m\\ d\equiv t\bmod p^n\\ m/d\equiv 0\bmod p^n}}\sgn(d)d^k +N^{k+1}\sum_{m\ge 1}\qq^{m/p^n}\sum_{\substack{d\mid m\\ d\equiv t/N\bmod p^n\\ m/d\equiv 0\bmod p^n}}\sgn(d)d^k,
			\end{align}
			where $\mu_N$ is the measure recalled in Appendix \ref{app:bernoulli}. We remark that to get the constant term in the above form, one needs the identity
			\begin{align}
				\left.D^k\Bigg(\frac{X^{t^\flat}}{X^{p^n}-1} - \frac{X^{N(t/N)^\flat}}{X^{Np^n}-1}\Bigg)\right|_{X=1} = \int_{t+p^n\Z_p} x^k\mu_N(x),
			\end{align}
			which can be worked out using the $p$-adic Fourier transform.
		\end{enumerate}
		Combining these identities, one then concludes that
		\begin{align}
			\Psi_{a,n}^{(k)}(2\pi i\tau,2\pi i) =& \frac{(2\pi i)^{-(k+1)}}{12p^n}\sum_{0\le l_1,l_2<p^n}
			\frac{d^{k+1}}{dz^{k+1}}\log\Lambda_{c,N}\left(\frac{l_1\tau+l_2}{p^n},\sL\right)\zeta^{-al_2}\\
			=&\frac{(N-1)(c^2-c)(2p^n-1)}{2}\1_{k=0,a\equiv 0} 
			+ \lim_{X\to 1^-}\sum_{v\in\RR_n} \Big[-c^2\fF^{(k)}_{a,n}(vX,\qq) +c^2N^{k+1}\fF^{(k)}_{a/N,n}(vX,\qq)\\
			&+ c^{k+1}\fF^{(k)}_{a/c,n}(vX^c,\qq)-(cN)^{k+1}\fF^{(k)}_{a/(cN),n}(vX^{cN},\qq)\Big]\\
			=&\int_{a+p^n\Z_p}x^k\mu_{\KL}(x)+ \sum_{m\ge 1}\qq^{m/p^n}\sigma^{(k)}_{a,n}(m).
		\end{align}
		Here, in the last equality, we used the fact that ($\delta_0$ is the Dirac measure supported at $0\in \Z_p$)
		\begin{align}\label{eq:KL-mazur2}
			\mu_{\KL}(x) = -c^2\mu_N(x) + c\mu_{N}(x/c) + \frac{(N-1)(c^2-c)}{2}\delta_0(x),
		\end{align}
		which is a rewrite of the equation \eqref{eq:KL-mazur}.
	\end{proof}
	
	We conclude with the following proposition that explains the distribution relations of the Bernoulli--Hurwitz measure.
	\begin{prop}[Addition formula at CM points]
		\label{prop:addition}
		Let $(\omega_1,\omega_2)\in \C^2$ be such that $\im(\omega_1/\omega_2)>0$ and $\fo\cdot (\Z\omega_1+\Z\omega_2)\subseteq \Z\omega_1+\Z\omega_2$. Let $n\in \Z_{>0}$, and suppose further that $\omega_1\in \varpi^n(\Z\omega_1+\Z\omega_2)$. Then for all $a\in \Z$, we have
		\begin{align}
			\sum_{0\le \lambda<p^e} \Psi^{(k)}_{a+\lambda p^{e(n-1)},en}(\omega_1,\omega_2) = \barpi^{k+1}\Psi^{(k)}_{a\varrho^{-1},e(n-1)}(\omega_1,\omega_2).
		\end{align}
	\end{prop}
	We will give two proofs of the addition formula; the first is elementary but technical, while the second is conceptually simpler and requires understanding the action of diamond operators on the forms $\Psi_{a,n}^{(k)}$. Both proofs are based on the following
	\begin{lem}\label{lem:cm-points}
		Let $\omega_1,\omega_2\in \C^2$ be such that $\im(\omega_1/\omega_2)>0$ and $\fo\cdot (\Z\omega_1+\Z\omega_2)\subseteq \Z\omega_1+\Z\omega_2$. Write $\sL = \Z\omega_1+\Z\omega_2$, and suppose $\omega_1\in \fp^n\sL$. Then for all $1\le r\le n$, we have
		\begin{align}\label{eq:cm-points}
			\{\lambda\omega_1/p^r \colon 0\le \lambda<p^r\} = \fp^{-r}\sL/\sL, \qquad
			\Z(\omega_1/p^r)+\Z\omega_2 = \fp^{-r}\sL.
		\end{align}
	\end{lem}
	\begin{proof}
		By our assumption, $\omega_1\in \fp^n\sL\subseteq \fp^r\sL$, so $\lambda\omega_1/p^r \in \bar{\fp}^{-r}\sL$ for all $\lambda\in \Z$. Now, suppose $(\lambda-\lambda')\omega_1/p^r \in \sL$ for $0\le \lambda,\lambda'<p^r$, then $(\lambda-\lambda')\omega_1\in p^r\sL$. By the linear independence of $\omega_1,\omega_2$, this forces $p^r\mid (\lambda-\lambda')$, whereby $\lambda=\lambda'$. This proves the first equality. As for the second, we observe that $\sL\subset \Z(\omega_1/p^r)+\Z\omega_2$, so the second equality follows from the first.
	\end{proof}
	\begin{proof}[First proof of Proposition \ref{prop:addition}]
		Let $\zeta = \zeta_{en}$. Put $\sL = \Z\omega_1+\Z\omega_2$. For all $b\in \Z$, we have
		\begin{align}
			\Psi^{(k)}_{b,en}(\omega_1,\omega_2) &= \frac{1}{12p^{en}}
			\sum_{0\le l_1<p^{e(n-1)}}\sum_{0\le \lambda<p^e}\sum_{0\le l_2<p^{en}}\frac{d^{k+1}}{dz^{k+1}}
			\log\Lambda_{c,N}\left(\frac{l_1\omega_1+l_2\omega_2}{p^{en}} +\frac{\lambda\omega_1}{p^e}, \sL\right)\zeta^{-bl_2}\\
			&= \frac{\barpi^{k+1}}{12p^{en}}
			\sum_{0\le l_1<p^{e(n-1)}}\sum_{0\le l_2<p^{en}}\frac{d^{k+1}}{dz^{k+1}}
			\log\Lambda_{c,N}\left(\frac{\barpi(l_1\omega_1+l_2\omega_2)}{p^{en}}, \sL\right)\zeta^{-bl_2},
		\end{align}
		where in the second equality we used Lemma \ref{lem:prod-formula} and \eqref{eq:cm-points} with $r=e$.
		
		Recall that by our choice, $\zeta_{en}^{p^e} = \zeta_{e(n-1)}$. As such,
		\begin{align}
			\sum_{0\le \lambda<p^e}\Psi^{(k)}_{a+\lambda p^{e(n-1)},en}(\omega_1,\omega_2) &= \frac{\barpi^{k+1}}{12p^{en}}\sum_{0\le l_1<p^{e(n-1)}}\sum_{0\le l_2<p^{en}}\frac{d^{k+1}}{dz^{k+1}}\log\Lambda_{c,N}\left(\frac{\barpi(l_1\omega_1 + l_2\omega_2)}{p^{en}}, \sL\right)
			\sum_{0\le \lambda<p^e}\zeta_{en}^{-(a+\lambda p^{e(n-1)})l_2}\\
			&=\frac{\barpi^{k+1}}{12p^{e(n-1)}}\sum_{0\le l_1,l_2<p^{e(n-1)}}
			\frac{d^{k+1}}{dz^{k+1}}\log\Lambda_{c,N}\left(\frac{\barpi(l_1\omega_1 + p^e l_2\omega_2)}{p^{en}}, \sL\right)\zeta_{e(n-1)}^{-al_2}.
		\end{align}
		
		To proceed we need a surgery. For $m\in \Z_{> 0}$ and $x\in \fo/\fp^{em}$, write $x^\flat_{\varpi^m} = x^\flat_{\fp^{em}}$. Similarly, for $x\in \fo/\bar{\fp}^{em}$, write $x^\flat_{\barpi^m} = x^\flat_{\bar\fp^{em}}$. Consider the following severances:
		\begin{itemize}
			\item $\barpi = (\barpi)^\flat_{\varpi^{n-1}} + x$ for some $x\in \fp^{e(n-1)}$.
			\item $\varpi^{-1} = (\varpi^{-1})^\flat_{\barpi^{n-1}} + y$ for some $y\in \fp^{-e}\bar{\fp}^{e(n-1)}$.
		\end{itemize}
		Then,
		\begin{align}
			\barpi\frac{l_1\omega_1+p^e l_2\omega_2}{p^{en}} = \left(\frac{p^e}{\varpi}\frac{l_1\omega_1}{p^{en}} + x\frac{l_2\omega_2}{p^{e(n-1)}}\right) +  \frac{(\barpi)^\flat_{\varpi^{n-1}}l_2\omega_2}{p^{e(n-1)}} = 
			\left(\frac{(\varpi^{-1})^\flat_{\barpi^n}l_1\omega_1}{p^{e(n-1)}} + y \frac{l_1\omega_1}{p^{e(n-1)}} + x\frac{l_2\omega_2}{p^{e(n-1)}}\right) +  \frac{(\barpi)^\flat_{\varpi^{n-1}}l_2\omega_2}{p^{e(n-1)}}.
		\end{align}
		We note that $y \frac{l_1\omega_1}{p^{e(n-1)}}\in \fp^{-e}\bar{\fp}^{e(n-1)}\fp^{en}p^{-e(n-1)}\sL = \sL$ by our assumption $\omega_1\in\varpi^n\sL$. Moreover, given $l_2$, since $\varpi$ is invertible modulo $\barpi^{n-1}$, we have
		\begin{align}
			\left\{\frac{(\varpi^{-1})^\flat_{\barpi^{n-1}}l_1\omega_1}{p^{e(n-1)}} + x\frac{l_2\omega_2}{p^{e(n-1)}}\colon 0\le l_1<p^{e(n-1)}\right\} = \barpi^{-n+1}\sL/\sL
		\end{align}
		by \eqref{eq:cm-points} again. Therefore,
		\begin{align}
			&\sum_{0\le l_1,l_2<p^{e(n-1)}}
			\frac{d^{k+1}}{dz^{k+1}}\log\Lambda_{c,N}\left(\frac{\barpi(l_1\omega_1 + p^el_2\omega_2)}{p^{en}}, \sL\right)\zeta_{e(n-1)}^{-al_2}\\
			=&\sum_{0\le l_1,l_2<p^{e(n-1)}}
			\frac{d^{k+1}}{dz^{k+1}}\log\Lambda_{c,N}\left(
			\frac{(\varpi^{-1})^\flat_{\barpi^{n-1}}l_1\omega_1}{p^{e(n-1)}} + x\frac{l_2\omega_2}{p^{e(n-1)}} +  \frac{(\barpi)^\flat_{\varpi^n}l_2\omega_2}{p^{e(n-1)}}, \sL\right)\zeta_{e(n-1)}^{-al_2}\\
			=&\sum_{0\le l_2<p^{e(n-1)}}\sum_{\rho\in \barpi^{-n+1}\sL/\sL}
			\frac{d^{k+1}}{dz^{k+1}}\log\Lambda_{c,N}(\rho +  l_2\omega_2/p^{e(n-1)}, \sL)\zeta_{e(n-1)}^{-a\varrho^{-1}l_2}\\
			=&\sum_{0\le l_1,l_2<p^{e(n-1)}}
			\frac{d^{k+1}}{dz^{k+1}}\log\Lambda_{c,N}\big((l_1\omega_1 + l_2\omega_2)/p^{e(n-1)}, \sL\big)\zeta_{e(n-1)}^{-a\varrho^{-1}l_2}.
		\end{align}
	\end{proof}
	\begin{proof}[Second proof of Proposition \ref{prop:addition}]
		By (MF1) we may suppose $(\omega_1,\omega_2) = (2\pi i\tau,2\pi i)$, where $\tau\in \cH$. Writing $\sL = \Z\tau+\Z$, our assumption says $\tau\in \fp^{en}\sL$. We note that $\sum_{0\le \lambda<p} \sigma_{a+\lambda p^{n-1},n}(m) = \sigma_{a,n-1}(m)$; for $m\ge 1$ this is clear, and for $m=0$ it follows from the distribution relation of $\mu_{\KL}$. This shows
		\begin{align}
			\sum_{0\le \lambda<p^e} \Psi^{(k)}_{a+\lambda p^{e(n-1)},en}(2\pi i \tau,2\pi i) = \Psi^{(k)}_{a,e(n-1)}(2\pi i\tau/p^e,2\pi i)
		\end{align}
		by Proposition \ref{prop:q-expansion}. Switching now to the moduli notation. We note that there is an isomorphism of triples over $\C$:
		\begin{align}\label{eq:diamond}
			\left(\C/2\pi i(\Z\cdot\tau/p^e +\Z),\frac{2\pi i\tau}{p^{en}},dz\right) =
			\left(\C/2\pi i\barpi^{-1}\sL,\frac{2\pi i\tau}{p^{en}},dz\right)
			\xrightarrow[\sim]{\times \barpi}
			\left(\C/2\pi i\sL,\barpi\frac{2\pi i\tau}{p^{en}},\barpi^{-1}dz\right).
		\end{align}
		Here the first equality uses the fact that $\Z(\tau/p^e) +\Z = \barpi^{-1}\sL$ from Lemma \ref{lem:cm-points}. It follows that
		\begin{align}
			\Psi^{(k)}_{a,e(n-1)}\left(\C/2\pi i(\Z\cdot\tau/p^e +\Z),\frac{2\pi i\tau}{p^{en}},dz\right) = \barpi^{k+1}\Psi_{a,e(n-1)}^{(k)}\left(\C/2\pi i\sL,\barpi\frac{2\pi i\tau}{p^{en}},dz\right).
		\end{align}
		For $u\in \Z_p^\times$, denote by $\chx{u}$ the diamond operator on level-$\Gamma^1(p^{e(n-1)})$ forms:
		\begin{align}
			f|\chx{u}(\omega_1,\omega_2) = f(a\omega_1+b\omega_2,c\omega_1+d\omega_2),
		\end{align}
		where $\gamma = \mat{a}{b}{c}{d}\in \SL_2(\Z)$ is any matrix congruent to $\mat{u^{-1}}{0}{*}{u}\bmod p^{e(n-1)}$. Take $u=\varrho$ and $\gamma = \mat{a}{b}{c}{d}$ a such matrix. By \eqref{eq:cm-points}, $\tau/p^{e(n-1)}\in \varpi\barpi^{-n+1}\sL$, so 
		\begin{align}
			\frac{\barpi\tau}{p^{en}} = \varpi^{-1}\frac{\tau}{p^{e(n-1)}} \equiv (\varpi^{-1})^\flat_{\barpi^{n-1}}\frac{\tau}{p^{e(n-1)}} \equiv 
			\frac{a\tau + b}{p^{e(n-1)}}\bmod \sL.
		\end{align}
		We then have
		\begin{align}
			\Psi_{a,e(n-1)}^{(k)}\left(\C/2\pi i\sL,\barpi\frac{2\pi i\tau}{p^{en}},dz\right) &=\Psi_{a,e(n-1)}^{(k)}\left(\C/2\pi i[\Z(a\tau+b)+\Z(c\tau+d)],\frac{2\pi i(a\tau+b)}{p^{e(n-1)}},dz\right) \\
			&= \Psi_{a,e(n-1)}^{(k)}|\chx{\varrho}(2\pi i\tau,2\pi i)\\
			&= \Psi_{a\varrho^{-1},e(n-1)}^{(k)}(2\pi i\tau,2\pi i).
		\end{align}
		The last equality is by (MF2).
	\end{proof}
	
	\section{Interlude: A revisit of the Eisenstein periods}
	\label{sec:katz}
	
	Keep the notation from previous sections, and in particular fix coprime $c,N\in \Z_{>1}$ that are prime to $p$. The goal of this section is twofold: The first is to rewrite the period formula using the form $\Phi_{a,n}(\tau) = \Psi_{a,n}(p^n\tau)$ (see \eqref{eq:epf-Phi} below); this leads naturally to the second part, where we explain how these new formulas are specializations of those of Katz's $p$-adic Eisenstein measure. Altogether, this furnishes an explicit comparison of Lichtenbaum's and Katz's methods.
	
	\subsection{Bernoulli--Hurwitz period formulas revised}
	For $a\in \Z_p,n\in \Z_{\ge 0}$ and $k\in \Z_{\ge 0}$, put
	\begin{align}
		\Phi^{(k)}_{a,n}(\omega_1,\omega_2) = \Psi^{(k)}_{a,n}(p^n\omega_1,\omega_2).
	\end{align}
	It is readily verified that $\Phi^{(k)}_{a,n}$ is a modular form of level $\Gamma_1(p^n)$, and, by Proposition \ref{prop:q-expansion}, we have the $\qq$-expansion:
	\begin{align}\label{eq:q-exp-Phi}
		\Phi^{(k)}_{a,n}(2\pi i\tau,2\pi i) = \sum_{m\ge 0}\sigma^{(k)}_{a,n}(m)\qq^m.
	\end{align}
	Now, on the space of level $\Gamma_1(p^n)$-modular forms, we have the usual diamond operators by $\Z_p^\times$ \cite[\S5.2]{diamond-shurman}. More precisely, if $u\in \Z_p^\times$ and $\gamma\in \SL_2(\Z)$ is such that $\gamma \equiv \mat{u^{-1}}{*}{0}{u}\bmod p^n$, then $f|\chx{u}(\umega) = f(\gamma\umega)$. In terms of the moduli interpretation, if $(\EE,P,\omega)$ is a point of evaluation for $f$, by a computation similar to that conducted in the second proof of Proposition \ref{prop:addition}, we have
	\begin{align}\label{eq:diamond-gamma_1}
		f|\chx{u}(\EE,P,\omega) = f(\EE,uP,\omega).
	\end{align}
	\begin{lem}\label{lem:diamond-Phi}
		For $a\in \Z_p,u\in \Z_p^\times, n\in \Z_{\ge 0}$ and $k\in \Z_{\ge 0}$, we have
		\begin{align}
			\Phi_{a,n}^{(k)}|\chx{u} = \Phi_{au^{-1},n}^{(k)}.
		\end{align}
	\end{lem}
	\begin{proof}
		Let $\gamma = \mat{a}{b}{c}{d}\in \SL_2(\Z)$ be such that $c\equiv 0\bmod p^n$ and $d\equiv u\bmod p^n$. For $\umega = (\omega_1,\omega_2)\in (\C^\times)^2$ with $\omega_1/\omega_2\in \cH$, we have
		\begin{align}
			\Phi_{a,n}^{(k)}(\gamma \umega) = \Psi_{a,n}^{(k)}\left(
			\mat{p^n}{0}{0}{1}
			\gamma\umega
			\right)
			= \Psi_{a,n}^{(k)}\left(
			\mat{a}{p^nb}{c/p^n}{d}
			\mat{p^n}{0}{0}{1}\umega
			\right) = \Psi_{a,n}^{(k)}|\chx{u}(p^n\omega_1,\omega_2).
		\end{align}
		The result then follows from (MF2) of Proposition \ref{prop:modularity} (see also the second proof of Proposition \ref{prop:addition}, where the diamond operators for $\Gamma^1(p^n)$-forms are introduced).
	\end{proof}
	Recall $\iota:\hat{\G}_m\to \hat{E}$ is the isomorphism chosen via \eqref{ass:iota}. Write $\pi_n = \iota(\zeta_n)\in \hat{E}[p^n] = E[\fp^n]$. 
	\begin{prop}\label{prop:Psi-Phi}
		Suppose $a\in \Z_p$, $n\in \Z_{\ge 0}$, $n'\ge n$ is the smallest integer divisible by $e$ and $k\in \Z_{\ge 0}$. We have
		\begin{align}\label{eq:epf2}
			\Omega_\infty^{-k-1}\barpi^{-(k+1)n'/e}
			\Psi_{a\varrho^{n'/e},n}^{(k)}\left(\frac{\varsigma+(-\varsigma)^\flat_{\fp^{n'}}}{p^{n'-n}}\right) = 
			\Phi_{a,n}^{(k)}(\C/L,\pi_n\cdot\Omega_\infty,dz).
		\end{align}
		In particular, $\Omega_\infty^{-k-1}\barpi^{-(k+1)n'/e}
		\Psi_{a,n}^{(k)}\left(\frac{\varsigma+(-\varsigma)^\flat_{\fp^{n'}}}{p^{n'-n}}\right)\in F_\fP(E[\fp^n])$.
	\end{prop}
	\begin{proof}
		Keeping \S\ref{subsec:conventions-modform} in mind, we find
		\begin{align}
			\Psi_{a\varrho^{n'/e},n}^{(k)}\left(\frac{\varsigma+(-\varsigma)^\flat_{\fp^{n'}}}{p^{n'-n}}\right) &= \Phi_{a\varrho^{n'/e},n}^{(k)}\left(\frac{\varsigma+(-\varsigma)^\flat_{\fp^{n'}}}{p^{n'}}\right)\\
			&= \Phi^{(k)}_{a\varrho^{n'/e},n}(\C/ \bar{\fp}^{-n'}\fa,1/p^n,dz)\\
			&= \Omega_\infty^{k+1}\Phi^{(k)}_{a\varrho^{n'/e},n}(\C/\barpi^{-n'/e}L,\Omega_\infty/p^n,dz)\\
			&=(\barpi^{n'/e}\Omega_\infty)^{k+1}\Phi^{(k)}_{a,n}|\chx{\varrho^{-n'/e}}\left(\C/L,\Omega_\infty\cdot \frac{p^{n'-n}}{\varpi^{n'/e}},dz\right)\\
			&=(\barpi^{n'/e}\Omega_\infty)^{k+1}\Phi^{(k)}_{a,n}\left(\C/L,\Omega_\infty\cdot \varrho^{-n'/e}\frac{p^{n'-n}}{\varpi^{n'/e}},dz\right)\\
			&=(\barpi^{n'/e}\Omega_\infty)^{k+1}\Phi^{(k)}_{a,n}(\C/L,\Omega_\infty \pi_n,dz).
		\end{align}
		Here the second equality uses Lemma \ref{lem:cm-points}, the fourth uses Lemma \ref{lem:diamond-Phi}, the fifth uses \eqref{eq:diamond-gamma_1}, and the last uses the congruence
		\begin{align}
			\pi_n\equiv p^{n'-n}\pi_{n'} = p^{n'-n}(\varrho^{-n'/e}\varpi^{-n'/e}) \pmod \fa
		\end{align}
		from Remark \ref{rem:pi_n}.
	\end{proof}
	We also record the additive property of $\Phi_{a,n}$'s, which can be directly seen from the $\qq$-expansions. (Recall that if $(\EE,\PP,\omega)$ is a point of evaluation for $\Gamma_1(p^n)$-forms and $f$ is a $\Gamma_1(p^{n-1})$-form, then $f(\EE,\PP,\omega) = f(\EE,p\PP,\omega)$.)
	\begin{lem}\label{lem:addition2}
		Let $\EE$ be an elliptic curve, $r\in \Z_{\ge 0}$, $\PP\in \EE[p^{r+1}]$ and $\omega$ be a nowhere vanishing global differential of $\EE$. For $a\in \Z_p$, we have
		\begin{align}
			\sum_{0\le b<p}\Phi_{a+bp^r,r+1}(\EE,\PP,\omega) = \Phi_{a,r}(\EE,p\PP,\omega).
		\end{align}
	\end{lem}
	
	\subsection{Trueness of Eisenstein periods}
	
	Thanks to Proposition \ref{prop:Psi-Phi}, we have the sleeker formula
	\begin{align}\label{eq:epf-Phi}
		\mu_{\BH}(a+p^n\Z_p) = \Phi_{a,n}(\C/L,\pi_n\cdot \Omega_\infty,dz),
	\end{align}
	which is the gateway to the explicit period formulae of Katz's Eisenstein measure, as we now explain. 
	
	Let $\bV$ denote the space of generalized modular forms defined over $W = W(\bar{\F}_p)$ \cite[\S1.1]{katz-eisenstein-measure}. Recall these are functions on the tuples $(\EE,\varphi)$, where $\EE$ is an ordinary elliptic curve and $\varphi: \hat{\EE}\to \hat{\G}_m$ is an isomorphism of formal group schemes, and their $\qq$-expansions at $(\Tate(\qq),\varphi_\can)$ land in $W[[\qq]]$ ($\varphi_{\can}$ is the isomorphism induced by the canonical inclusion $\mu_{p^\infty}\to \Tate(\qq)$). By a result of Igusa, the topology of $\bV$ is determined by the $\qq$-expansion at the single cusp $(\Tate(\qq),\varphi_\can)$ \cite[p.~498, bottom paragraph]{katz-moduli}. Now, recall from \cite[\S1.5]{katz-eisenstein-measure} that a true modular form of level $\Gamma_{00}(p^n)$ is a function on the triples $(\EE,\alpha_n,\omega)$, where $\EE$ is an elliptic curve, $\alpha_n:\mu_{p^n} \to \EE[p^n]$ is an embedding of groups schemes, and $\omega$ is a nowhere vanishing global K\"ahler differential on $\EE$. A such form $f$ can be regarded as a generalized modular form $f^\gen$ via the correspondence
	\begin{align}
		f^\gen(\EE,\varphi) = f(\EE,\varphi^{-1}|_{\mu_{p^n}},\varphi^*(dt/t)),
	\end{align}
	where, as in \S\ref{subsec:change-of-variable}, $t$ is the variable for $\hat{\G}_m$. We can restate the period formula \eqref{eq:epf-Phi} (\textit{cf.}, equations (2.5.2), (2.6.15), (2.7.17), \textit{op.~cit.}; see also \cite[\S XI]{katz-moduli}).
	\begin{cor}\label{cor:BH-Phigen}
		For $a\in \Z_p$ and $n\in \Z_{\ge 0}$, we have
		\begin{align}
			\mu_{\BH}(a+p^n\Z_p) = \Omega_p\Phi_{a,n}^\gen(E,\iota^{-1}).
		\end{align}
	\end{cor}
	\begin{proof}
		By Remark \ref{rem:omega-p}, $(\iota^{-1})^*(dt/t) = \Omega_p \omega_E$. So
		\begin{align}
			\Phi^\gen_{a,n}(E,\iota^{-1}) & =  \Phi_{a,n}(\C/L,\pi_n\cdot\Omega_\infty,\Omega_pdz)\\
			& = \Omega_p^{-1}\Phi_{a,n}(\C/L,\pi_n\cdot\Omega_\infty,dz).
		\end{align}
		Now, use \eqref{eq:epf-Phi}.
	\end{proof}
	We now define a $p$-adic measure $\mu_\Eis\in \Mes(\Z_p,\mathbb{V})$ by
	\begin{align}
		\mu_\Eis(a+p^n\Z_p) = \Phi_{a,n}^\gen.
	\end{align}
	From the $\qq$-expansion, the distribution relation is easily verified. We contend that $\mu_\Eis$ is actually the $c$-regularized $p$-adic Eisenstein  measure of Katz. In order to state their precise relation, we first recall the classical Eisenstein series. For $k\in \Z_{\ge 1}$ and $\tau\in \cH$, put
	\begin{align}
		G_k(\tau) = \begin{cases}
			\sum_{m\in \Z}\sum_{n\in \Z,(m,n)\ne (0,0)} (m\tau+n)^{-k} & \text{if }k\in 2\Z;\\
			0 & \text{otherwise}.
		\end{cases}
	\end{align}
	We note that the summation defining $G_2(\tau)$ is conditionally convergent \cite[Corollary 9.2.6]{complex-analysis}.
	\begin{prop}\label{prop:eisenstein}
		Let $k\in \Z_{\ge 0}$. Then
		\begin{align}\label{eq:Psi-eisenstein}
			\Psi^{(k)}_{0,0}(\tau) = (c^2-c^{k+1})(1-N^{k+1})k!G_{k+1}(\tau).
		\end{align}
	\end{prop}
	\begin{proof}
		To begin with, observe that $\Psi_{0,0}^{(k)}=0$ for $k\in 2\Z_{\ge0}$. Indeed, when $k$ is even, the coefficients $\sigma_{0,0}^{(k)}(m)$ are automatically zero for $m>0$, since
		\begin{align}
			\sum_{d\mid n} \sgn(d)d^k = \frac{1}{2}\sum_{d\mid n}( \sgn(d)d^k + \sgn(-d)(-d)^k )= 0.
		\end{align}
		The constant term $\sigma^{(k)}_{0,0}(0)$ is also zero, as the measure $\mu_{\KL}$ is odd; see \eqref{eq:epf-KL}. It follows that both sides of \eqref{eq:Psi-eisenstein} are zero.
		
		Now suppose $k$ is odd. By the Poisson summation formula (see, e.g., \cite[Lemma 9.2.4]{complex-analysis}), we have for $k\ge 1$,
		\begin{align}\label{eq:Gk-q-exp}
			\sum_{m\in \Z}\sum_{\substack{n\in \Z\\ (m,n)\ne (0,0)}} \frac{1}{(m\tau+n)^{k+1}} = 2\zeta(k+1) + \frac{(2\pi i)^{k+1}}{k!}\sum_{m\ge 1} \qq^m\sum_{d\mid m}\sgn(d)d^k.
		\end{align}
		By the functional equation, we have
		\begin{align}
			k!(2\pi i)^{-k-1}2\zeta(k+1) = \zeta(-k).
		\end{align}
		Hence,
		\begin{align}
			(2\pi i)^{-k-1}(c^2-c^{k+1})(1-N^{k+1})k!G_{k+1}(\tau)
			= (c^2-c^{k+1})(1-N^{k+1})\zeta(-k) + \sum_{m\ge 1}\qq^m\sigma^{(k)}_{0,0,c,N}(m).
		\end{align}
		As $\int_{\Z_p}x^k\mu_{\KL}(x) = (c^2-c^{k+1})(1-N^{k+1})\zeta(-k)$, the right hand side is exactly $\Psi_{0,0}^{(k)}(2\pi i\tau,2\pi i)$ by Proposition \ref{prop:q-expansion}.
	\end{proof}
	\begin{rem}\label{rem:Gk}
		As explained in \cite[Example 4.5.4]{katz-modular-scheme}, the function $G_k$ is a $p$-adic modular form of weight $k$ and full level (true modular if $k\ge 4$). Hence we may use the moduli interpretation and evaluate it on tuples of the form $(\EE,\omega)$, where $\EE$ is an ordinary elliptic curve over a $p$-adically complete algebra and $\omega$ is a nowhere vanishing differential on $\EE$.
	\end{rem}
	Let $\mu_\Katz\in \Mes(\Z_p,\bV)$ be the one-variable $p$-adic Eisenstein measure, which is denoted by $2\mathbf{H}^{N,0}$ in \cite[\S3.3]{katz-eisenstein-measure}, and is characterized by the interpolation property
	\begin{align}
		\int_{\Z_p}x^k\mu_{\Katz}(x) = (1-N^{k+1})k!G_{k+1}^\gen.
	\end{align}
	\textit{Caveat.} Note that Katz's $G_k$ is our $(k!/2)G_k$.
	\begin{thm}\label{thm:Eis-Katz}
		Let $c,N\in \Z_{>1}$ be coprime integers with $p\nmid cN$. Then
		\begin{align}\label{eq:Eis-Katz}
			\mu_{\Eis}(a+p^n\Z_p) = c^2\mu_{\Katz}(a+p^n\Z_p) - c\mu_{\Katz}(a/c+p^n\Z_p).
		\end{align}
	\end{thm}
	\begin{proof}
		Denote by $\nu_\Eis$ the $p$-adic measure whose periods are given by the right hand side of \eqref{eq:Eis-Katz}. Then it suffices to show that $\int_{\Z_p}x^k \mu_{\Eis} = \int_{\Z_p} x^k\nu_{\Eis}$ for all $k\in \Z_{\ge 0}$ \cite[Corollary 3.3.1]{Hida}. By definition,
		\begin{align}
			\int_{\Z_p}x^k\nu_\Eis(x) = c^2\int_{\Z_p}x^k\mu_{\Katz}(x) - c\int_{\Z_p}x^k\mu_\Katz(x/c) = (c^2-c^{k+1})(1-N^{k+1})k!G_{k+1}^\gen.
		\end{align}
		On the other hand,
		\begin{align}
			\int_{\Z_p}x^k\mu_{\Eis}(x) = \lim_{n\to\infty} \sum_{0\le a<p^n} a^k\Phi_{a,n}^\gen = \lim_{n\to \infty} \sum_{0\le a<p^n}\Phi_{a,n}^{(k),\gen} = \Phi_{0,0}^{(k),\gen}.
		\end{align}
		Here, the second equality comes from the congruence of $\qq$-expansions, $a^k\Phi_{a,n}(\qq) \equiv \Phi_{a,n}^{(k)}(\qq) \bmod p^n$. Indeed, for all $b\in (\Z/p^n)^\times$,
		\begin{align}
			a^k\sum_{\substack{d\mid m\\ d\equiv a/b\bmod p^n}}\sgn(d) \equiv b^k\sum_{\substack{d\mid m\\ d\equiv a/b\bmod p^n}}\sgn(d)d^k \pmod {p^n}.
		\end{align}
		It follows that, for all $m\in \Z_{>0}$,
		\begin{align}
			a^k\sigma_{a,n}(m) \equiv \sigma^{(k)}_{a,n}(m) \pmod{p^n}.
		\end{align}
		The same identity holds for $m=0$, since
		\begin{align}
			\sigma^{(k)}(0) = \int_{a+p^n\Z_p}x^k\mu_{\KL}(x) \equiv a^k\int_{a+p^n\Z_p}\mu_{\KL}(x) = a^k\sigma_{a,n}(0) \pmod{p^n};
		\end{align}
		the last equality is by the definition of $\mu_\KL$. We can then finish the proof using Proposition \ref{prop:eisenstein}.
	\end{proof}
	Note that this asserts that, after the $c$-regularization, the Eisenstein measure is valued in the space of true modular forms; it is exactly the problematic form $G_2$ that gets killed in the process.

	\section{Interpolation via congruences of Eisenstein series}\label{sec:interpolation}

	The goal of this section is to define the Bernoulli--Hurwitz $p$-adic zeta functions, and prove their interpolation properties (\textit{cf.}~\cite[p.~276, top formula]{katz-eisenstein-measure} and \cite[Theorem 6.4]{lichtenbaum}). As mentioned in the introduction, our approach makes crucial use of the explicit period formulas \eqref{eq:epf}, \eqref{eq:epf-Phi}, and the existence of certain weight one modular invariants constructed in Appendix \ref{app:hasse}, thereby departing from the existing strategies in the literature \cite{lichtenbaum,coates-wiles-aus,cassou-nogues-CM}. In fact, thanks to the versatility of this approach, we are allowed to work with any elliptic curve $\EE$ over $W$ having a good ordinary reduction instead of just CM ones.
	
	\subsection{Setup}\label{subsec:setup}
	Below, fix an ordinary elliptic curve $\cE$ over $W$ and a generator $\omega_{\cE}$ of $H^0(\cE,\Omega^1_{\cE/W})$. Choose a trivialization $\varphi_\cE: \hat{\cE}\xrightarrow{\sim}\hat{\G}_m$ over $W$, and put $\Omega_{p,\cE} = \varphi_\cE^*(dt/t)/\omega_\cE\in W^\times$. Denote $\varphi_{n,\cE} = \varphi_\cE|_{\hat{\cE}[p^n]}$, and put $\cP_n = \varphi_{\cE}^{-1}(\zeta_n)$. The $p$-canonical subgroup of $\cE$ is then given by $H_{\cE} = \varphi_\cE^{-1}(\mu_p)$, whose geometric points as a group is generated by $\cP_1$. Define the $(c,N)$-regularized $p$-adic Bernoulli--Hurwitz measure $\mu_{\BH,\cE} = \mu_{c,N,\cE,\varphi_\cE,\omega_\cE}\in \Mes(\Z_p,W)$ by the period formula
	\begin{align}\label{eq:mu-BH-ord}
		\mu_{\BH,\cE}(a+p^n\Z_p) = \Phi_{a,n}(\cE,\cP_n,\omega_\cE);
	\end{align}
	often we drop $\cE$ from the notation of $\mu_{\BH,\cE}$ when the data $(\cE,\varphi_\cE,\omega_\cE)$ is clear. By the same argument as in Corollary \ref{cor:BH-Phigen}, we see that
	\begin{align}
		\mu_{c,N,\cE,\varphi_\cE,\omega_\cE} = \Omega_{p,\cE}\cdot \mu_{\Eis}(\cE,\varphi_\cE).
	\end{align}
	When $(\cE,\varphi_\cE,\omega_\cE) = (E,\iota^{-1},\omega_E)$, we recover the Bernoulli--Hurwitz measure considered in the previous sections by the explicit period formula \eqref{eq:epf-Phi}.
	
	Next, for $i\in \Z/(p-1)$, define the Bernoulli--Hurwitz $p$-adic zeta function as
	\begin{align}\label{eq:bh-zeta}
		\zeta^\BH_{p}(s,\omega^i) = \zeta^\BH_{p,\cE}(s,\omega^i) =(c^2 - \omega^i(c)\chx{c}^{1-s})^{-1}(1-\omega^i(N)\chx{N}^{1-s})^{-1}\int_{\Z_p^\times} \omega^{i-1}(x)\chx{x}^{-s}\mu_{\BH}(x).
	\end{align}
	When $i=0$, we will write $\zeta_p^\BH(s,\omega^0)$ simply as $\zeta_p^\BH(s)$. We will see from the interpolation formula, Theorem \ref{thm:zeta-interpolation} below, that $\zeta^\BH_p(s,\omega^i)$ does not depend on $c$ or $N$. 
	
	To state the interpolation property of $\zeta_p^\BH$, we will need to recall the level raising operator, which sends a full level $p$-adic modular form $f$ to a $\Gamma_0(p)$-level form $f^\dagger$: If $\EE$ is an ordinary elliptic curve, $\HH\subset \EE[p]$ is a subgroup-scheme of order $p$ with $\pi_\HH: \EE\to \EE/\HH$ the natural projection, and $\omega$ a basis of the K\"ahler 1-forms of $\EE$, then
	\begin{align}
		f^\dagger(\EE,\HH,\omega) = f(\EE/\HH,\hat{\pi}_\HH^*\omega),
	\end{align}
	with $\hat{\pi}_\HH: \EE/\HH\to \EE$ the dual isogeny (Recall that $\EE/\HH$ is still ordinary; see, e.g., \cite[Proposition 8.4.4.(b)]{diamond-shurman}). The effect of $\dagger$ on the standard $\qq$-expansion at $(\Tate(\qq),\mu_{p^n},\omega_\can)$ can be seen to be $f^\dagger(\qq) = f(\qq^p) = f(\Tate(\qq^p),\omega_\can)$. If $f$ is of weight $k$ and $\HH\subset\EE$ is the canonical subgroup, we define the $p$-stabilization
	\begin{align}\label{eq:level-raising}
		f^*(\EE,\omega) = f(\EE,\omega) - p^{k-1}f^\dagger(\EE,\HH,\omega) = f(\EE,\omega) - p^{k-1}f(\EE/\HH,\hat{\pi}_\HH^*\omega).
	\end{align}
	
	\subsection{Interpolations}
	\begin{thm}\label{thm:zeta-interpolation}
		For all $k\in \Z_{\ge 0}$, we have
		\begin{align}\label{eq:bh-ord-interpol}
			\zeta_{p,\cE}^\BH(-k,\omega^{k+1}) =\Omega_{p,\cE}^{-k}\cdot k!G_{k+1}^*(\cE,\omega_\cE).
		\end{align}
	\end{thm}
	
	\begin{lem}\label{lem:p-form}
		Let $k\in \Z_{\ge 0}$. For $\tau\in\cH$, we have
		\begin{align}
			\Phi^{(k)}_{0,1}(2\pi i\tau,2\pi i) = p^k\Phi^{(k)}_{0,0}(2\pi ip\tau,2\pi i).
		\end{align}
		Thus $\Phi^{(k)}_{0,1} = p^k \Phi^{(k),\dagger}_{0,0}$.
	\end{lem}
	\begin{proof}
		Note that for $m\in \Z_{>0}$,
		\begin{align}
			\sum_{\substack{d\mid m\\ d\equiv 0\bmod p}} d^k = p^k\sum_{d\mid m/p} d^k, 
		\end{align}
		so $\sigma_{0,1}^{(k)}(m) = p^k \sigma^{(k)}_{0,0}(m/p)\1_{p\mid m}$. For $m=0$, we can manually verify this identity, by noting that $\mu_{\KL}(pa+p^n\Z_p) = \mu_{\KL}(a+p^{n-1}\Z_p)$ for all $a\in \Z_p$ and $n\in \Z_{>0}$, thanks to \eqref{eq:epf-KL} and \eqref{eq:KL-mazur}. Hence, by the $\qq$-expansion \eqref{eq:q-exp-Phi},
		\begin{align}
			\Phi^{(k)}_{0,1}(2\pi i\tau,2\pi i) = \sum_{m\ge 0}\sigma_{0,1}^{(k)}(m)\qq^m = p^k\sum_{m\ge 0}\sigma^{(k)}_{0,0}(m)\qq^{mp} = p^k\Phi_{0,0}^{(k)}(2\pi ip\tau,2\pi i).
		\end{align}
	\end{proof}
	\begin{proof}[Proof of Theorem \ref{thm:zeta-interpolation}, part 1]
		We have $\int_{\Z_p^\times}\omega^{k+1-1}(x)\chx{x}^k\mu_{\BH}(x) =
		\int_{\Z_p^\times} x^k\mu_{\BH}(x)$, which is approximated by the Riemann sums
		\begin{align}
			\sum_{\substack{0\le a<p^n\\ p\nmid a}} a^k \mu_{\BH}(a+p^n\Z_p) &=\sum_{\substack{0\le a<p^n\\ p\nmid a}} a^k \Phi_{a,n}(\cE, \cP_n,\omega_\cE).
		\end{align}
		up to $p^n$. As explained in the proof of Theorem \ref{thm:Eis-Katz}, the $\qq$-expansions of weight-$1$ modular forms $a^k\Phi_{a,n}$ and $A_n^{-k}\Phi^{(k)}_{a,n}$ are congruent, where $A_n$ is the $W/p^nW$-valued weight one form constructed in Appendix \ref{app:hasse}. By the $\qq$-expansion principle, Proposition \ref{prop:q-exp-principle}, we have
		\begin{align}
			a^k\Phi_{a,n}(\cE, \cP_n,\omega_\cE)
			&= a^k\Phi_{a,n}(\cE,\varphi_{n,\cE},\omega_\cE)\\
			&\equiv \left[A_n^{-k}\Phi_{a,n}^{(k)}\right](\cE,\varphi_{n,\cE},\omega_\cE)\pmod{p^n}\\
			&=\Omega_{p,\cE}^{-k}\Phi_{a,n}^{(k)}(\cE, \cP_n,\omega_\cE).
		\end{align}
		Here, the first equality is by regarding $\Phi_{a,n}$ as a form on $Z_{00}(p^n)$, and the third equality is by the definition of $A_n$ and $\Omega_{p,\cE}$. Therefore,
		\begin{align}
			\int_{\Z_p^\times}x^k\mu_{c,N}(x) &\equiv \Omega_{p,\cE}^{-k} \sum_{\substack{0\le a<p^n\\ p\nmid a}}
			\Phi^{(k)}_{a,n}(\cE,\cP_n,\omega_\cE) \pmod{p^n}\\
			&=\Omega_{p,\cE}^{-k} \sum_{1\le a<p} \Phi^{(k)}_{a,n}(\cE,\cP_1,\omega_\cE)\\
			&=\Omega_{p,\cE}^{-k}\left[\Phi^{(k)}_{0,0}(\cE,\omega_\cE) - \Phi^{(k)}_{0,1}(\cE,\cP_1,\omega_\cE)\right]\\
			\label{eq:Phi-integral}
			&=\Omega_{p,\cE}^{-k}\Phi^{(k),*}_{0,0}(\cE,\omega_\cE).
		\end{align}
		Here, the second and third equalities are by Lemma \ref{lem:addition2}, and the last one is by Lemma \ref{lem:p-form}. As
		\begin{align}\label{eq:mu-c-N-interpol}
			\int_{\Z_p^\times}x^k\mu_{c,N}(x) = \zeta^\BH_p(-k,\omega^i)(c^2-c^{k+1})(1-N^{k+1}),
		\end{align}
		the proof for $k\ne 1$ is concluded by invoking Proposition \ref{prop:eisenstein}.
	\end{proof}
	\begin{proof}[Proof of Theorem \ref{thm:zeta-interpolation}, part 2]
		We now deal with the exceptional case $k=1$. From \eqref{eq:Phi-integral} and Proposition \ref{prop:eisenstein}, we see that $\int_{\Z_p^\times}x \mu_{c,N}(x) = 0$, so $\zeta_p^{\BH}(s,\omega^2)$ is regular at $s=-1$. Differentiating \eqref{eq:bh-zeta}, we find
		\begin{align}\label{eq:zeta-1}
			\int_{\Z_p^\times} x\log_p x\mu_{c,N}(x) = -c^2\log_p c(1-N^2)\zeta_p^\BH(-1,\omega^2),
		\end{align}
		where $\log_p$ is the $p$-adic logarithm. To compute the left hand side, we use the approximation
		\begin{align}\label{eq:riemann-sum}
			\int_{\Z_p^\times} x\log_p x\mu_{c,N}(x) \equiv \sum_{\substack{0\le a<p^n\\ p\nmid a}} (a\log_p a) \Phi_{a,n}(\cE,\cP_n,\omega_\cE) \pmod {p^{n-1}}.
		\end{align}
		For $m\in \Z_{>0}$, we have 
		\begin{align}
			&a\log_p a \cdot\sigma_{a,n}(m) \equiv 
			c^2\sum_{\substack{d\mid m\\ d\equiv a}}\sgn(d)d \log_p d - c\sum_{\substack{d\mid m\\ d\equiv a/c}} \sgn(d)(dc)\log_p(dc)\\
			&\quad -c^2N\sum_{\substack{d\mid m\\ d\equiv a/N}}\sgn(d)(dN)\log_p(dN) +cN\sum_{\substack{d\mid m\\ d\equiv a\varrho^n/cN}} \sgn(d)(dcN)\log_p(dcN) \pmod{p^{n-1}}.
		\end{align}
		Consequently,
		\begin{align}
			\sum_{\substack{0\le a<p^n\\ p\nmid a}}
			a\log_pa\cdot\sigma_{a,n}(m) &\equiv  
			c^2\sum_{d\mid m, p\nmid d} \sgn(d)d\cdot [\log_p d -\log_p(dc) -N^2\log_p(dN) +N^2\log_p(dcN)]\pmod{p^{n-1}}\\
			&=-c^2\log_p c(1-N^2)\sum_{d\mid m, p\nmid d}\sgn(d)d.
		\end{align}
		Turning now to the constant term, the decomposition \eqref{eq:KL-mazur2} gives
		\begin{align}
			\sum_{\substack{0\le a<p^n\\ p\nmid a}} a\log_p a\cdot \sigma_{a,n}(0) & = \sum_{\substack{0\le a<p^n\\ p\nmid a}} a\log_p a\cdot \mu_{\KL}(a+p^n\Z_p)\\
			&= \sum_{\substack{0\le a<p^n\\ p\nmid a}} a\log_p a\cdot \left[-c^2\mu_{N}(a+p^n\Z_p) + c\mu_N(a/c+p^n\Z_p)\right]\\
			& \equiv \sum_{\substack{0\le a<p^n\\ p\nmid a}} \left[-c^2a\log_p a +ac^2\log_p(ac)\right] \mu_{N}(a+p^n\Z_p)  \pmod{p^{n-1}}\\
			& = c^2\log_p c\sum_{\substack{0\le a<p^n\\ p\nmid a}} a \mu_{N}(a+p^n\Z_p)\\
			& \equiv c^2\log_p c\int_{\Z_p^\times}
			x\mu_N(x) \pmod{p^{n-1}}\\
			& = -c^2\log_p c(1-N^2)(1-p)\zeta(-1).
		\end{align}
		Here, the second equality is by \eqref{eq:KL-mazur} and the last equality comes from the interpolation property of $\mu_N$.
		
		Consider now the $\qq$-expansion
		\begin{align}
			G_2(\qq) = -\frac{1}{12} + \sum_{m\ge 1}\qq^m\sum_{d\mid m}\sgn(d)d
		\end{align}
		of the $p$-adic modular form $G_2$ by \eqref{eq:Gk-q-exp}, and the $p$-stabilization
		\begin{align}
			G_2^*(\qq) = -\frac{1-p}{12} + \sum_{m\ge 1}\qq^m\sum_{d\mid m,p\nmid d}\sgn(d)d = G_2(\qq) - pG_2(\qq^p).
		\end{align}
		The congruences we established above then says that
		\begin{align}\label{eq:congruence-wt-2}
			\sum_{\substack{0\le a<p^n\\ p\nmid a}} a\log_p a \Phi_{a,n}(\qq) \equiv -c^2\log_p c(1-N^2) G_2^*(\qq) \equiv -c^2\log_p c(1-N^2) G_2^*(\qq)A_n^{-1}(\qq)\pmod{p^{n-1}}.
		\end{align}
		
		Applying the $\qq$-expansion principle, \eqref{eq:congruence-wt-2} gives the congruence
		\begin{align}
			\sum_{\substack{0\le a<p^n\\ p\nmid a}} a\log_p a \Phi_{a,n}(\cE,\cP_n,\omega_\cE) &\equiv -c^2\log_p c(1-N^2) \left[A_n^{-1}G_2^*\right](\cE,\cP_n,\omega_\cE) \pmod {p^{n-1}}\\
			&= -c^2\log_p c(1-N^2)\Omega_p^{-1}G_2^*(\cE,\cP_n,\omega_\cE)\\
			&\label{eq:dagger}
			= -c^2\log_p c(1-N^2)\Omega_p^{-1}G_2^*(\cE,\omega_\cE).
		\end{align}
		Combining \eqref{eq:zeta-1}, \eqref{eq:riemann-sum}, \eqref{eq:dagger}, we have
		\begin{align}
			\zeta_p^\BH(-1,\omega^2) \equiv \Omega_p^{-1} G_2^*(\cE,\omega_E) \pmod{p^{n-1}}.
		\end{align}
		Let $n\to \infty$, and we are done.
	\end{proof}

	\subsection{The case of CM elliptic curves}
	\label{subsec:cm-interpolation}
	
	Let $(\cE,\varphi_\cE,\omega_\cE)$ be as above; recall $H_\cE\subset \cE[p]$ denotes the canonical subgroup. Put $\cE' = \cE/H_{\cE}$ and $\omega_{\cE'} = \hat{\pi}_{H_\cE}^*\omega_\cE$. Consider the factorization of $[p]$ for formal groups
	\begin{align}
		[p]\colon \hat{\cE}\xrightarrow{\pi_{H_\cE}} \hat{\cE}'\simeq \hat{\cE}/\hat{\cE}[p] \xrightarrow{\ \theta_\cE\ } \hat{\cE},
	\end{align}
	then we see that $\theta_\cE$ is an isomorphism, and we have thus an induced trivialization $\varphi_{\cE'} = \varphi_\cE\circ\theta_{\cE}: \hat{\cE}'\xrightarrow{\sim} \hat{\G}_m$. Put $\Omega_{p,\cE'} = \varphi_{\cE'}^*(dt/t)/\omega_{\cE'}$.
	\begin{lem}\label{lem:omega-prime}
		We have $\Omega_{p,\cE'} = \Omega_{p,\cE}$. Thus $\omega_{\cE'}$ is a generator of $H^0(\cE',\Omega_{\cE'/W}^1)$.
	\end{lem}
	\begin{proof}
		Pulling back, we find $\pi_{H_\cE}^*\varphi_{\cE'}^*(dt/t) = \Omega_{p,\cE'}\pi_{H_\cE}^*\omega_{\cE'}$. Now,
		\begin{align}
			\pi_{H_\cE}^*\varphi_{\cE'}^*(dt/t) = \pi_{H_\cE}^*\theta_\cE^*\varphi_\cE^*(dt/t) = (\theta_\cE\circ\pi_{H_\cE})^* (\Omega_{p,\cE}\omega_\cE) = p\Omega_{p,\cE}\cdot\omega_\cE.
		\end{align}
		For the other side, we find
		\begin{align}
			\Omega_{p,\cE'}\pi_{H_\cE}^*\omega_{\cE'} = \Omega_{p,\cE'}\pi_{H_\cE}^*\hat{\pi}_{H_\cE}^*\omega_\cE = p\Omega_{p,\cE'}\cdot\omega_\cE.
		\end{align}
		Equate the rightmost quantities and we get $\Omega_{p,\cE'} = \Omega_{p,\cE}$. The last claim then follows because, by definition, $\Omega_{p,\cE}$ is in $W^\times$. Hence $\omega_{\cE'}$ is a generator of $H^0(\cE',\Omega^1_{\cE'/W})$ as $\varphi^*_{\cE'}(dt/t)$ is.
	\end{proof}

	Now, let $(\cE,\varphi_{\cE},\omega_{\cE}) = (E,\iota^{-1},\omega_E)$ be the CM elliptic curve that is fixed throughout. For $i\in \Z_{\ge 0}$, denote by $E_i$ the CM elliptic curve $E/E[\fp^i]$, so $E_0 = E$ and $E_{i+1} = E_i/H_{E_i}$. By the discussion above, we have the extra data $(\varphi_i = \varphi_{E_i},\omega_i = \omega_{E_i})$ of $E_i$ constructed recursively starting from $(E_0,\varphi_0,\omega_0) = (E,\iota^{-1},\omega_E)$. By Lemma \ref{lem:omega-prime}, we can form the $p$-adic measure $\mu_{\BH,E_i}$ and $\zeta_{p,E_i}^\BH$ for all $i\ge 0$. The fact that $E$ is CM means the chain of isogenies $E_0\to E_1\to \cdots$ is cyclic; i.e., we have an isomorphism $\theta_\varpi: E_e= E/E[\varpi]\to E$ given by multiplication by $\varpi$. The question then is to compare $\varphi_e$ with $\varphi_0$ and $\omega_e$ with $\omega_0$,
	\begin{lem}\label{lem:cm-cyclicity}
		We have 
		\begin{align}
			\varphi_e\circ\theta_\varpi^{-1} = \barpi \cdot \varphi_0,\quad
			\omega_e = \barpi\cdot \theta_\varpi^*\omega_0.
		\end{align}
	\end{lem}
	\begin{proof}
		It suffices to prove the second identity; the first follows since, \textit{a priori}, there is some $\lambda\in \Z_p^\times$ with $\varphi_e\circ\theta_\varpi^{-1} = \lambda \varphi_0$, whereby Lemma \ref{lem:omega-prime} forces $\lambda = \omega_e/\theta_\varpi^*\omega_0$.
		
		We now turn to differentials. Denote by $\pi_e: E\to E_e$ the natural projection. Then $[\varpi] = \theta_\varpi\circ\pi_e$. Dualizing, we have $[\barpi] = \hat{\pi}_e\circ\hat{\theta}_\varpi$. As $\omega_e = \hat{\pi}_e^*\omega_0$,
		\begin{align}
			\hat{\theta}_\varpi^*\omega_e = [\barpi]^*\omega_0 = \barpi \omega_0.
		\end{align}
		Now pullback both sides via $\theta_\varpi^*$.
	\end{proof}
	
	\begin{cor}\label{cor:interpolation-cm}
		For all $k\in \Z_{\ge 0}$ and $i\in \Z_{\ge 0}$, we have
		\begin{align}
			\zeta_{p,E_i}^\BH(-k,\omega^{k+1}) = \Omega_p^{-k}k!\left[G_{k+1}(E_i,\omega_i) - p^k G_{k+1}(E_{i+1},\omega_{i+1})\right],
		\end{align}
		and
		\begin{align}\label{eq:euler-factor}
			\sum_{0\le i<e} p^{ik}\zeta_{p,E_i}^\BH(-k,\omega^{k+1}) = \Omega_p^{-k}(1-\varpi^{k+1}p^{-e})k!G_{k+1}(E,\omega_E).
		\end{align}
	\end{cor}
	\begin{proof}
		The first identity follows from the interpolation formula \eqref{eq:bh-ord-interpol}, the description of the level raising operator \eqref{eq:level-raising} and Lemma \ref{lem:omega-prime}. As such,
		\begin{align}
			\sum_{0\le i<e} p^{ik}\zeta_{p,E_i}^\BH(-k,\omega^{k+1}) &=
			\Omega_p^{-k}k!\left[G_{k+1}(E,\omega_E) - p^{ke}G_{k+1}(E_e,\omega_e)\right]\\
			&=
			\Omega_p^{-k}k!\left[G_{k+1}(E,\omega_E) - p^{ke}G_{k+1}(E,(\theta_\varpi^{-1})^*\omega_e)\right]\\
			&=\Omega_p^{-k}(1-p^{ke}\barpi^{-k-1})G_{k+1}(E,\omega_E).
		\end{align}
		The last equality used Lemma \ref{lem:cm-cyclicity}.
	\end{proof}
	\begin{rem}
		Let $(\cE,\varphi_\cE,\omega_\cE)$ be as in \S\ref{subsec:setup}. Let $\cE' = \cE/H_{\cE}$ and $(\varphi_{\cE'},\omega_{\cE'})$ be the attached datum. Note that
		\begin{align}\label{eq:indra}
			\mu_{\BH,\cE,\varphi_{\cE},\omega_\cE}(pa+p^{n+1}\Z_p) = \mu_{\BH,\cE',\varphi_{\cE'},\omega_{\cE'}}(a+p^n\Z_p).
		\end{align}
		To see this, observe that
		\begin{align}
			\Phi_{pa,n+1}(\qq) = \Phi_{a,n}(\qq^p).
		\end{align}
		So for any $(\EE,\PP_{n+1},\omega)$ a point of evaluation for $\Gamma_1(p^{n+1})$-forms,
		\begin{align}
			\Phi_{pa,n+1}(\EE,\PP_{n+1},\omega) = \Phi_{a,n}(\EE/\HH,\pi_\HH(\PP_{n+1}),\hat{\pi}_\HH^*\omega),
		\end{align}
		with $\HH\subset \EE[p]$ the subgroup generated by $p^n \PP_{n+1}$ and $\pi_\HH: \EE\to \EE/\HH$ the natural projection. Plugging in $(\EE,\PP_{n+1},\omega) = (\cE,\cP_{n+1} = \varphi_\cE^{-1}(\zeta_{n+1}),\omega_\cE)$, we find
		\begin{align}
			\Phi_{pa,n+1}(\cE,\cP_{n+1},\omega_\cE)= \Phi_{a,n}(\cE',\pi_{\HH_\cE}\varphi_{\cE}^{-1}(\zeta_{n+1}),\hat{\pi}_{H_\cE}^*\omega_\cE).
		\end{align}
		The identity \eqref{eq:indra} is concluded by noting
		\begin{align}
			\pi_{\HH_\cE}\varphi_{\cE}^{-1}(\zeta_{n+1}) = \theta_{\cE}^{-1}\theta_\cE\pi_{H_\cE}\varphi_\cE^{-1}(\zeta_{n+1}) = \theta_\cE^{-1}\varphi_\cE^{-1}(\zeta_n) = \varphi_{\cE'}^{-1}(\zeta_n)
		\end{align}
		and $\hat{\pi}_{H_\cE}^*\omega_\cE = \omega_{\cE'}$.
		Therefore, identifying $\Z_p$ with $\fo_\fp$, \eqref{eq:euler-factor} can be written as
		\begin{align}
			(c^2-c^{k+1})^{-1}(1-N^{k+1})^{-1}\int_{\fo\setminus \varpi\fo}x^k\mu_{\BH,E,\iota^{-1},\omega_E}(x) = 
			\Omega_p^{-k}(1-\varpi^{k+1}p^{-e})k!G_{k+1}(E,\omega_E).
		\end{align}
	\end{rem}

	\section{The non-critical point $s=1$}
	\label{sec:FLF}
	
	Keep the notation from the last section; in particular fix the triple $(\cE,\varphi_\cE,\omega_\cE)$. As a testing ground of our modular approach, we will start by proving a residue formula which is not present in \cite{katz-eisenstein-measure} nor \cite{lichtenbaum} (\textit{cf}., however, \cite[Theorem 35.(iv)]{colmez-schneps}).
	\begin{thm}\label{thm:residue}
		Let $\cE$ be an ordinary elliptic curve over $W$, $\varphi_\cE: \hat{\cE}\xrightarrow{\sim} \hat{\G}_m$ an isomorphism of formal group schemes, and $\omega_\cE$ a generator of $H^0(\cE,\Omega^1_{\cE/W})$. Let $\zeta^\BH_p(s)$ be the $p$-adic Bernoulli--Hurwitz zeta function attached to $(\cE,\varphi_\cE,\omega_\cE)$. Then
		\begin{align}\label{eq:residue}
			\lim_{s\to 1}(s-1)\zeta^\BH_{p,\cE}(s) = \Omega_{p,\cE}(1-1/p).
		\end{align}
	\end{thm}
	Actually, we will give two proofs. The first is similar to that of Lichtenbaum \cite[\S9]{lichtenbaum} (see also \cite[\S III]{cassou-nogues-CM}), and makes crucial use of the elliptic units and their product formulas; this method confines us to the CM curve $E$ with the extra condition $e=1$, namely $\fp = \varpi\fo$. The second, which is elliptic-unit-free, works directly with Eisenstein series via the explicit period formula \eqref{eq:epf-Phi}, and reduces the proof to some simple combinatorial computations. In fact, the second approach gives us more, as we will see in \S\ref{subsec:kronecker} that it also allows us to understand the constant term of $\zeta_p^\BH$ at $s=1$.
	
	\subsection{The first proof of the residue formula}
	
	For the duration of this subsection, we will assume $(\cE,\varphi_\cE,\omega_\cE) = (E,\iota^{-1},\omega_E)$, and that $\fp$ is principal, so $\fp = \varpi\fo$ and $F_\fP = K_\fp$ by class field theory. In this case, the measure $\mu_{c,N}$ can be constructed by the elliptic function as explained in \S\ref{sec:measure}. By the integral representation \eqref{eq:bh-zeta}, we are interested in computing $\int_{\Z_p^\times}x^{-1}\mu_{c,N}(x)$. Thanks to the $p$-adic Fourier transform, we can instead compute a primitive of the power series $f_{c,N}$ under the differential operator $t\frac{d}{dt}$. Before stating the next theorem, recall that in step 2 of the proof of Proposition \ref{prop:der-interpol}, we have shown that $\Lambda_{c,N}(z)\in K_\fp[[z]]\setminus zK_\fp[[z]]$. Therefore we may take the $p$-adic logarithm $\log\Lambda_{c,N}(z)\in K_\fp[[z]]$, and we have $\frac{d}{dz}(\log\Lambda_{c,N}) = \Lambda_{c,N}'/\Lambda_{c,N}$ (the right-hand side was formally denoted by $\frac{d}{dz}\log\Lambda_{c,N}$ earlier).
	
	\begin{lem}\label{lem:gcN}
		Notation as in \S\ref{sec:measure}. Write $g_{c,N} = \frac{1}{12}(\varepsilon^{-1})^*(\log\Lambda_{c,N})$. Then $g_{c,N}(t) \in \hat{K}^\ur_\fp[[t-1]]$, and
		\begin{align}
			t\frac{d}{dt}g_{c,N} = \Omega_p^{-1}f_{c,N}.
		\end{align}
	\end{lem}
	\begin{proof}
		We have seen in the proof of Proposition \ref{prop:der-interpol}, step 2, that $(\xi^{-1})^*\log\Lambda_{c,N}$ is in $K_\fp[[w]]$, whereby $(\varepsilon^{-1})^*\log\Lambda_{c,N}\in \hat{K}^\ur_\fp[[t-1]]$. Since $f_{c,N} = \frac{1}{12} (\varepsilon^{-1})^*\left(\frac{d}{dz}\log\Lambda_{c,N}\right)$, this boils down to the relation that for $F(z) \in \C_p[[z]]$,
		\begin{align}
			t\frac{d}{dt}\left[(\varepsilon^{-1})^*F\right] = \Omega_p^{-1}(\varepsilon^{-1})^*\left(\frac{dF}{dz}\right),
		\end{align}
		which can be directly verified.
	\end{proof}
	
	Consider now the restriction operator $\mu\mapsto \mu|_{\Z_p^\times}$ on $\Mes(\Z_p,\Our)$, which under the Fourier transform is
	\begin{align}
		f\mapsto f^{\diamond} = f(t) - \frac{1}{p}\sum_{\zeta^p=1}f(\zeta t).
	\end{align}
	Clearly the restriction commutes with $\mu\mapsto x\mu$, therefore $\diamond$ commutes with $t\frac{d}{dt}$ on $\Our[[t-1]]$.
	\begin{lem}\label{lem:convergence}
		Suppose $f(t) = \sum_{n\ge 0}a_n(t-1)^n \in \C_p[[t-1]]$ satisfies the growth condition
		\begin{align}
			\lim_{n\to \infty} \left[\ord_p(a_n) + n/(p-1)\right] = \infty.
		\end{align}
		Then for all $\zeta\in \mu_p$,
		\begin{align}
			f(\zeta t) = \sum_{n\ge 0} (t-1)^n\zeta^n \sum_{l\ge n} (\zeta-1)^{l-n}\binom{l}{n}a_l
		\end{align}
		defines an element in $\C_p[[t-1]]$, so $f^\diamond$ is well-defined. Moreover, $t\frac{d}{dt}(f^\diamond) = (t\frac{d}{dt}f)^\diamond$.
	\end{lem}
	\begin{proof}
		The expansion of $f(\zeta t)$ follows from the binomial expansion
		\begin{align}
			(\zeta t-1)^l = \sum_{0\le n\le l}\binom{l}{n}\zeta^n(t-1)^n(\zeta-1)^{l-n}.
		\end{align}
		Put $|x|_p = p^{-\ord_p(x)}$ for $x\in \C_p^\times$, then
		\begin{align}
			\left|\binom{l}{n}(\zeta-1)^{l-n} a_l\right|_p \le p^{(n-l)/(p-1)}|a_l|_p\to 0 \quad(l\to \infty).
		\end{align}
		So the coefficients of $f(\zeta t)$ converge. The commutation of $\diamond$ and $t\frac{d}{dt}$ can be checked manually.
	\end{proof}
	
	We contend that, writing $g_{c,N} = \sum_{n\ge 0} b_n(t-1)^n$, then $|b_n|_p\le |n+1|_p^{-1}$, so Lemma \ref{lem:convergence} applies. Indeed, put
	\begin{align}
		t^{-1}f_{c,N}(t) = \sum_{n\ge 0} a_n(t-1)^n,
	\end{align}
	then $a_n \in \Our$. By Lemma \ref{lem:gcN}, we know $b_{n+1} = \Omega_p^{-1}a_n/(n+1)$ for all $n\in \Z_{\ge 0}$, so $|b_{n+1}|_p\le |n+1|_p^{-1}$ as claimed. As such,
	\begin{align}
		t\frac{d}{dt}(g_{c,N}^\diamond) = \Omega_p^{-1}f_{c,N}^\diamond.
	\end{align}
	Note that this shows $g_{c,N}^\diamond\in \Our[[t-1]]$, since the operator $t\frac{d}{dt}$ is invertible on the image of $\diamond$ (alternatively, the operator $\mu\mapsto x\mu$ is invertible on measures supported on $\Z_p^\times$). Denote by $\nu_{c,N}\in \Mes(\Z_p,\Our)$ the Fourier inverse of $g_{c,N}^\diamond$. and we find
	\begin{align}\label{eq:limit-pre}
		\Omega_p^{-1}\int_{\Z_p^\times} x^{-1}\mu_{c,N}(x) = \int_{\Z_p^\times} \nu_{c,N}(x) = g_{c,N}(1) - \frac{1}{p}\sum_{\zeta^p=1} g_{c,N}(\zeta).
	\end{align}

	Next, recall the product formula of elliptic units
	\begin{lem}
		Let $\alpha\in \fo$ and $c\in \Z_{>1}$ be such that $\gcd(c,6\alpha) = 1$. Then for $z\in \C/L$,
		\begin{align}
			\prod_{P\in \alpha^{-1}L/L} \Theta_c(z+P,L) = \Theta_c(\alpha z,L).
		\end{align}
		Consequently, for $c,N\in \Z_{>1}$ with $\gcd(c,6Np)=\gcd(N,p)=1$,
		\begin{align}
			\sum_{\zeta^p=1} g_{c,N}(\zeta) = \frac{1}{12}\sum_{P\in \varpi^{-1}L/L}\log_p\Lambda_{c,N}(P,L) = \frac{1}{12}\log_p\Lambda_{c,N}(0,L).
		\end{align}
	\end{lem}
	\begin{proof}
		The first identity is classical; see \cite[Theorem 3, Theorem 8]{coates-cm}. Given this,
		\begin{align}
			\sum_{0\ne \rho\in N^{-1}L/L}\sum_{P\in \varpi^{-1}L/L}\log_p\Theta_c(\rho+P,L) = \sum_{0\ne \rho\in N^{-1}L/L}\log_p\Theta_c(\varpi \rho,L)
			= \sum_{0\ne \rho\in N^{-1}L/L}\log_p\Theta_c(\rho,L) 
			= \log_p\Lambda_{c,N}(0,L).
		\end{align}
	\end{proof}
	\begin{proof}[First proof of Theorem \ref{thm:residue}]
		Choose coprime $c,N\in \Z_{>1}$ such that $\gcd(cN,6p) = 1$. By \eqref{eq:limit-pre}, we see that
		\begin{align}
			\Omega_p^{-1}(c^2-1)\lim_{s\to 1}(1-\chx{N}^{1-s})\zeta^\BH_p(s) &= \frac{1}{12}(1-1/p)\log\Lambda_{c,N}(0,L)\\
			&= \frac{1}{12}(1-1/p)\log_p\left[\left. \frac{\Theta_c(Nz,L)}{\Theta_c(z,L)}\right|_{z=0}\right]\\
			&= (1-1/p)(c^2-1)\log_pN.
		\end{align}
		Here the last equality can be seen from the simple computation (denote by $\lambda$ the first coefficient of $\theta(z,L)$ at $z=0$):
		\begin{align}
			\lim_{z\to 0}\frac{\Theta_c(Nz,L)}{\Theta_c(z,L)} = \lim_{z\to 0} \frac{\lambda^{c^2-1}N^{12(c^2-1)}z^{12} + O(z^{13})}{\lambda^{c^2-1}z^{12} + O(z^{13})} = N^{12(c^2-1)}.
		\end{align}
	\end{proof}
	
	\subsection{The second proof}
	We now work with any ordinary triple $(\cE,\varphi_{\cE},\omega_\cE)$. Take the Riemann sum
	\begin{align}
		\sum_{\substack{0\le a<p^n\\ p\nmid a}}
		a^{-1}\mu_\BH(a+p^n\Z_p) = \sum_{\substack{0\le a<p^n\\ p\nmid a}}
		a^{-1}\Phi_{a,n}(\cE,\cP_n,\omega_\cE),
	\end{align}
	which is congruent to $\int_{\Z_p^\times}x^{-1}\mu_\BH(x)$ modulo $p^n$. For $m\in \Z_{>0}$, we observe the congruence
	\begin{align}
		a^{-1}\sigma_{a,n}(m) & \equiv c^2\sum_{\substack{d\mid m\\ d\equiv a}}\frac{\sgn(d)}{d} - c^2N\sum_{\substack{d\mid m\\ d\equiv a/N}}\frac{\sgn(d)}{dN}
		-c\sum_{\substack{d\mid m\\ d\equiv a/c}}\frac{\sgn(d)}{dc}
		+cN\sum_{\substack{d\mid m\\ d\equiv a/cN}}
		\frac{\sgn(d)}{cNd} \pmod{p^n}\\
		& = c^2\sum_{\substack{d\mid m\\ d\equiv a}}\frac{\sgn(d)}{d} - c^2\sum_{\substack{d\mid m\\ d\equiv a/N}}\frac{\sgn(d)}{d}
		-\sum_{\substack{d\mid m\\ d\equiv a/c}}\frac{\sgn(d)}{d}
		+\sum_{\substack{d\mid m\\ d\equiv a/cN}}
		\frac{\sgn(d)}{d} \pmod{p^n}\\
		&=0\pmod{p^n}.
	\end{align}
	When $m=0$, we find
	\begin{align}
		a^{-1}\sigma_{a,n}(0) = a^{-1} \mu_{\KL}(a+p^n\Z_p) \equiv \int_{a+p^n\Z_p} x^{-1}\mu_\KL(x) \pmod{p^n}.
	\end{align}
	It follows that
	\begin{align}
		\sum_{\substack{0\le a<p^n\\ p\nmid a}}a^{-1}\Phi_{a,n}(\qq) \equiv 
		\int_{\Z_p^\times} x^{-1}\mu_{\KL}(x) = (1-c^2)\int_{\Z_p^\times}x^{-1}\mu_{N}(x)  \pmod{p^n}.
	\end{align}
	By Leopoldt's formula \cite[Theorem 3.2]{iwasawa-lectures}, we have
	\begin{align}
		\int_{\Z_p^\times}x^{-1}\mu_N(x) = \lim_{s\to 1} -(1-\chx{N}^{1-s})\zeta_p(s) =  -(1-1/p)\log_pN,
	\end{align}
	where $\zeta_p(s)$ is the Kubota--Leopoldt zeta function of trivial character. We therefore have a congruence of weight-one forms
	\begin{align}
		\sum_{\substack{0\le a<p^n\\ p\nmid a}}a^{-1}\Phi_{a,n}(\qq) \equiv 
		(c^2-1)\log_pN(1-1/p)A_n(\qq)  \pmod{p^n}.
	\end{align}
	Applying the $\qq$-expansion principle, we have
	\begin{align}
		\int_{\Z_p^\times}x^{-1}\mu_\BH(x)\equiv (c^2-1)\log_pN(1-1/p)A_n(\cE,\varphi_{n,\cE},\omega_\cE) \equiv \Omega_{p,\cE}(c^2-1)\log_pN(1-1/p)\pmod{p^n}.
	\end{align}
	Since $\int_{\Z_p^\times}x^{-1}\mu_\BH(x) = (c^2-1)\log_pN\lim_{s\to 1}(s-1)\zeta_p^\BH(s)$, letting $n\to\infty$ we get the full proof of the residue formula \eqref{eq:residue}.
	
	\subsection{The $p$-adic Kronecker's first limit formula}
	\label{subsec:kronecker}
	The modular approach tells us more; below, we will prove the $p$-adic first limit formula of $\zeta_p^\BH(s)$ (\textit{cf}., \cite[\S20.4]{lang-elliptic}, \cite[\S X]{katz-real-analytic-eisenstein}). We start with some preliminaries. Let $\Delta$ be the full level weight-12 modular form with $\Delta(\qq) = \qq\prod_{m\ge 1}(1-\qq^m)^{24}$. Consider then the quotient
	\begin{align}
		\Delta^{(p)}(\omega_1,\omega_2) = \frac{\Delta(\omega_1,\omega_2)^p}{\Delta(p\omega_1,\omega_2)},
	\end{align}
	which is a modular form defined over $\Z$ and of weight $12(p-1)$ and level $\Gamma_0(p)$ (see, e.g., \cite[proof of Theorem 2.1]{DokDokLocal}, where the two cusps of $X_0(p)$ are flipped). Below we will regard $\Delta^{(p)}$ as a full level $p$-adic modular form via the canonical subgroups of ordinary elliptic curves. We prove first the key lemma.
	\begin{lem}
		For any $p$-adically complete $W=W(\bar{\F}_p)$-algebra $B$ and $(\EE,\varphi_n,\omega)$ defined over $B$, where $\EE$ is ordinary, $\varphi_n: \hat{\EE}[p^n]\xrightarrow{\sim} \mu_{p^n}$ is an isomorphism and $\omega$ is a nowhere vanishing K\"ahler differential of $\EE$, the value $\Delta^{(p)}A^{-12(p-1)}_n(\EE,\varphi_n,\omega) \in B/p^nB$ lands in $1+(pB/p^nB)$. It follows that, for all $n\in \Z_{>0}$, we can define a weight-zero $(W/p^{n-1}W)$-valued modular form on $Z_{00}(p^n)$ by
		\begin{align}
			G_{0,n}(\EE,\varphi_n,\omega) = -\frac{1}{12p}\log\left[\Delta^{(p)}A_n^{-12(p-1)}(\EE,\varphi_n,\omega)\right] = -\frac{1}{12p}\log\left[\frac{\Delta^{(p)}(\EE,\omega)}{A_n^{12(p-1)}(\EE,\varphi_n,\omega)}\right],
		\end{align}
		where for $B$ a $p$-adically complete ring, $\frac{1}{p}\log:1+pB \to B$ is defined by
		\begin{align}
			\frac{1}{p}\log(1+px) = \sum_{n\ge 1,p\nmid n}(-1)^{n-1}\frac{p^{n-1}x^n}{n}
			+ \sum_{n\ge 1,p\nmid n}(-1)^{pn-1}\frac{p^{np-2}x^{np}}{n} + \sum_{n\ge 1,p\nmid n}(-1)^{p^2n-1}\frac{p^{np^2-3}x^{np^2}}{n}+\cdots.
		\end{align}
		Moreover, the $\qq$-expansion of $G_{0,n}$ is given by
		\begin{align}
			\sum_{m\ge 1}\qq^m\sum_{d\mid m,p\nmid d}\sgn(d)/d \in W/p^{n-1}[[\qq]].
		\end{align}
	\end{lem}
	\begin{proof}
		We have
		\begin{align}
			\Delta^{(p)}(\qq) = \prod_{m\ge 1}\frac{(1-\qq^m)^{24p}}{(1-\qq^{pm})^{24}} \equiv 1 \equiv A_n^{12(p-1)}(\qq)\pmod p.
		\end{align}
		By the $\qq$-expansion principle, for any triple $(\EE,\varphi_n,\omega)$ as in the statement, we have
		\begin{align}
			\Delta^{(p)}(\EE,\varphi_n,\omega) - A_n^{12(p-1)}(\EE,\varphi_n,\omega) \in pB/p^nB.
		\end{align}
		By definition, $A_n(\EE,\varphi_n,\omega)$ is valued in $(B/p^nB)^\times$, whereby we find $\Delta^{(p)}A^{-12(p-1)}_n(\EE,\varphi_n,\omega) -1 \in pB/p^nB$.
		
		Next, clearly $G_{0,n}$ is weight-$0$ and valued in $W/p^{n-1}$. To show that it is a modular form on $Z_{00}(p^n)$, it suffices to show that
		\begin{align}
			-\frac{1}{12p}\log\left[\frac{\Delta(\Tate(\qq),\omega_\can)}{\Delta(\Tate(\qq^p),\omega_\can)A_n^{12(p-1)}(\Tate(\qq),\varphi_{n,\can},\omega_\can)^{}}\right] = -\frac{1}{12p}\log[\Delta(\qq)/\Delta(\qq^p)] \in W/p^{n-1}W[[\qq]].
		\end{align}
		We may lift everything to $W[1/p][[\qq]]$ and compute there:
		\begin{align}
			-\frac{1}{12p}\log\left[\Delta(\qq)/\Delta(\qq^p)\right] &= 2\sum_{m\ge 1}\left[-\log(1-\qq^m) + \frac{1}{p}\log(1-\qq^{mp})\right]\\
			&= 2\sum_{m\ge 1}\sum_{d\ge 1}\left[\frac{\qq^{md}}{d} - \frac{\qq^{mdp}}{dp}\right]\\
			&=\sum_{m\ge 1}\qq^m\sum_{d\mid m, p\nmid d}\sgn(d)/d.
		\end{align}
	\end{proof}
	\begin{thm}\label{thm:first-limit-formula}
		We have
		\begin{align}
			\zeta_{p,\cE}^\BH(s) = \Omega_{p,\cE}\left[\frac{1-1/p}{s-1} + (1-1/p)\gamma_p + (1-1/p)\log_p\Omega_{p,\cE} -\frac{1}{12p}\log_p \Delta^{(p)}(\cE,\omega_\cE) + O(s-1)\right].
		\end{align}
	\end{thm}
	Here $\gamma_p$ is the $p$-adic Euler constant given by $(1-1/p)^{-1}\lim_{s\to 1} \left[\zeta_p(s) - \frac{1-1/p}{s-1}\right]$; see, e.g., \cite[p.~266]{koblitz78}.
	\begin{proof}
		We are left with computing the constant term. Since
		\begin{align}
			\omega^{-1}(x)\chx{x}^{-s} = x^{-1} -(s-1)x^{-1}\log_px + O((s-1)^2)),
		\end{align}
		we have
		\begin{align}
			(1-\chx{N}^{1-s})\zeta_p^\BH(s) = \int_{\Z_p^\times}x^{-1}\mu_\BH(x) - (s-1)\int_{\Z_p^\times}x^{-1}\log_p x\mu_{\BH}(x) + O((s-1)^2).
		\end{align}
		In light of Theorem \ref{thm:residue}, through direct computations we find $\lim_{s\to 1}\left[\zeta_p^\BH(s) - \frac{\Omega_{p,\cE}(1-1/p)}{s-1}\right]$ is given by
		\begin{align}\label{eq:limit-1}
			\frac{1}{(1-c^2)\log_pN}\left\{
			\Omega_{p,\cE}\log_pN(1-1/p)\left[\log_pc + \frac{(1-c^2)\log_pN}{2}\right] + \int_{\Z_p^\times}x^{-1}\log_px\mu_\BH(x)\right\}.
		\end{align}
		The game now is to evaluate $\int_{\Z_p^\times} x^{-1}\log_px \mu_\BH(x)$. Again, let $\varphi_{n,\cE} = \varphi_\cE|_{\hat{\cE}[p^n]}$. We consider the Riemann sum
		\begin{align}
			\sum_{1\le a<p^n,p\nmid a}a^{-1}(\log_p a)\Phi_{a,n}(\cE,\varphi_{n,\cE},\omega_\cE) \equiv \int_{\Z_p^\times}x^{-1}\log_px\mu_\BH(x) \pmod{p^{n-1}}.
		\end{align}
		By a computation similar to part 2 of the proof of Theorem \ref{thm:zeta-interpolation}, we have for $m\in\Z_{\ge 1}$,
		\begin{align}
			\sum_{\substack{0\le a<p^n\\ p\nmid a}} a^{-1}\log_pa\cdot \sigma_{a,n}(m) \equiv (1-c^2)\log_pN\sum_{d\mid m, p\nmid d}\sgn(d)/d\pmod{p^{n-1}}.
		\end{align}
		For $m = 0$,
		\begin{align}
			\sum_{\substack{0\le a<p^n\\ p\nmid a}} a^{-1}\log_pa\cdot \sigma_{a,n}(0) & \equiv
			\sum_{\substack{0\le a<p^n\\ p\nmid a}}\left[(1-c^2)a^{-1}\log_p a + a^{-1}\log_pc\right]\mu_N(a + p^n\Z_p)\pmod{p^{n-1}}\\
			&\equiv (1-c^2)\int_{\Z_p^\times} \frac{\log_px}{x}\mu_N(x) +
			\log_p c \int_{\Z_p^\times}\frac{\mu_N(x)}{x}\pmod{p^{n-1}}.
		\end{align}
		As $\int_{\Z_p^\times}x^{-1}\chx{x}^{1-s}\mu_N(x) = (\chx{N}^{1-s}-1)\zeta_p(s)$, we have
		\begin{align}
			\int_{\Z_p^\times}x^{-1}\mu_N(x) = -\log_pN(1-1/p)\quad \text{and}\quad \int_{\Z_p^\times}x^{-1}\log_p x\mu_N(x) = (1-1/p)\left[\gamma_p\log_p N - \frac{(\log_pN)^2}{2}\right].
		\end{align}
		Consequently,
		\begin{align}
			\sum_{\substack{0\le a<p^n\\ p\nmid a}} a^{-1}\log_pa\cdot \sigma_{a,n}(0) \equiv
			(1-c^2)(1-1/p)\gamma_p\log_pN - (1-1/p)\log_pN\left[\frac{(1-c^2)\log_pN}{2}
			+\log_p c\right] \pmod{p^{n-1}}.
		\end{align}
		Denote the right hand side above as $C$. Then we have shown that
		\begin{align}
			\sum_{1\le a<p^n,p\nmid a}
			a^{-1}\log_pa\Phi_{a\varrho^n,n}(\qq) \equiv \left[C + (1-c^2)\log_pN G_{0,n}(\qq)\right]A_n(\qq)\pmod{p^{n-1}}.
		\end{align}
		Applying the $\qq$-expansion principle and evaluate on $(\cE,\varphi_{n,E},\omega_\cE)$, we find
		\begin{align}\label{eq:limit-2}
			\int_{\Z_p^\times}x^{-1}\log_px\mu_{\BH}(x)\equiv \Omega_{p,\cE}\left\{C + (1-c^2)\log_pN \frac{1}{12p}\log_p\left[
			\frac{\Omega_{p,\cE}^{12(p-1)}}{\Delta^{(p)}(\cE,\omega_\cE)}\right]\right\}\pmod{p^{n-1}},
		\end{align}
		where we used the fact that, on $1+pW$, $\log = \log_p$, the usual $p$-adic logarithm. The limit formula then follows by combining \eqref{eq:limit-1} and \eqref{eq:limit-2}, and letting $n\to \infty$.
	\end{proof}
	
	\begin{cor}\label{cor:FLF-cm}
		Notation as in \S\ref{subsec:cm-interpolation}. Then
		\begin{align}
			\sum_{0\le i<e} \zeta_{p,E_i}^\BH(s)/p^i =
			\Omega_p\left\{ (1-1/p^e)
			\left[\frac{1}{s-1} + \gamma_p + \log_p\Omega_p + \frac{1}{12}\log_p\Delta(E,\omega_E)\right]
			+\frac{1}{p^e}\log_p\barpi
			\right\} + O(s-1).
		\end{align}
	\end{cor}
	\begin{proof}
		By Theorem \ref{thm:first-limit-formula} and Lemma \ref{lem:omega-prime}, we have
		\begin{align}
			\sum_{0\le i<e} \zeta_{p,E_i}^\BH(s)/p^i &= 
			\Omega_p(1-1/p^e)\left(\frac{1}{s-1} + \gamma_p + \log_p\Omega_p\right)
			-\frac{\Omega_p}{12p^e}\log_p\left[
			\frac{\Delta(E_0,\omega_0)^{p^e}}{\Delta(E_e,\omega_e)}
			\right] + O(s-1)\\
			&= 
			\Omega_p(1-1/p^e)\left(\frac{1}{s-1} + \gamma_p + \log_p\Omega_p\right)
			-\frac{\Omega_p}{12p^e}\log_p\left[
			\frac{\Delta(E,\omega_E)^{p^e}}{\Delta(E,\barpi\omega_E)}
			\right] + O(s-1)\\
			&=\Omega_p(1-1/p^e)\left[\frac{1}{s-1} + \gamma_p + \log_p\Omega_p + \frac{1}{12}\log_p\Delta(E,\omega_E)\right] + (\Omega_p/p^e)\log_p\barpi + O(s-1).
		\end{align}
		The second equality is by Lemma \ref{lem:cm-cyclicity}.
	\end{proof}
	\appendix
	
	\section{Bernoulli measures as constant terms}
	\label{app:bernoulli}
	
	We will now verify \eqref{eq:interpolation-KL}. Thus let $c,N\in \Z_{\ge 1}$ be as in \S\ref{subsec:notation}. For $a\in \Z_p$ and $n\in \Z_{\ge 0}$, put
	\begin{align}\label{eq:KL-period}
		\mu_{\KL}(a+p^n\Z_p) = -\frac{c(c-1)(N-1)}{2}\1_{a\not\equiv 0\bmod p^n} + c^2\frac{-a^\flat_{p^n}+N(a/N)^\flat_{p^n}}{p^n} + c\frac{(a/c)^\flat_{p^n} - N(a/cN)^\flat_{p^n}}{p^n}.
	\end{align}
	Recall the $N$-regularized Bernoulli periods are given by the formula \cite[p.~38]{Lang}
	\begin{align}\label{eq:epf-KL}
		\mu_N(a+p^n\Z_p) = \frac{a^\flat_{p^n} - N(a/N)^\flat_{p^n}}{p^n} + \frac{N-1}{2}.
	\end{align}
	For all $k\in \Z_{\ge 0}$, we have (Theorem 2.3 \textit{op.~cit.})
	\begin{align}
		\int_{\Z_p} x^k\mu_N(a+p^n\Z_p) = (-1)^k(1-N^{k+1})\zeta(-k).
	\end{align}
	Comparing the formulas, we see that for $C = \frac{(c^2-c)(N-1)}{2}$, we have
	\begin{align}\label{eq:KL-mazur}
		\mu_{\KL}(a+p^n\Z_p) = -c^2\mu_N(a+p^n\Z_p) + c\mu_N(a/c+p^n\Z_p) + C\1_{a\equiv 0\bmod p^n}.
	\end{align}
	It follows that
	\begin{align}
		\int_{\Z_p}x^k\mu_{\KL}(x) &\equiv \sum_{0\le a<p^n}a^k[-c^2\mu_N(a+p^n\Z_p) + c\mu_N(a/c+p^n\Z_p) + C\1_{a\equiv 0\bmod p^n}] \pmod{p^n}\\
		& \equiv  -c^2\sum_{0\le a<p^n}a^k\mu_N(a+p^n\Z_p) + c^{k+1}\sum_{0\le a<p^n}(a/c)^k\mu_N(a/c+p^n\Z_p)  + C\1_{k = 0} \pmod{p^n}\\
		& \equiv -(c^2-c^{k+1})\int_{\Z_p}x^k\mu_N(x) + C\1_{k=0} \pmod{p^n}\\
		&= (-1)^{k-1}(c^2-c^{k+1})(1-N^{k+1})\zeta(-k) +\frac{(c^2-c)(N-1)}{2}\1_{k=0}.
	\end{align}
	When $k=0$, the above is identically zero, as $\zeta(0) = -1/2$. When $k>0$, the above is nonzero only when $k$ is odd, thereby $(-1)^{k-1}=1$.

	\section{Weight one modular invariants on the ordinary loci}\label{app:hasse}

	The purpose of this section is to explain the construction of some weight-one modular forms defined on the ordinary loci of the modular curves $X_{00}(p^n)$ for $n\in \Z_{>1}$. We will start with a discussion of this ordinary locus, and recall a version of $\qq$-expansion principle on it. In the second part, we explain the construction of the weight one modular invariants on these ordinary loci, and justify why they should be viewed as Hasse invariants with levels. 
	
	\subsection{The ordinary locus}
	Let $W = W(\bar{\F}_p)$ be the ring of Witt vectors of $\bar{\F}_p$. Fix below an integer $n\in \Z_{>0}$, and recall from \cite[Chapter II]{katz-real-analytic-eisenstein} the moduli space $X_{00}(p^n) = X(\Gamma_{00}(p^n))$ parametrizes isomorphism classes of
	\begin{align}
		\left\{(\EE,\alpha_n)\colon \EE\text{ is an elliptic curve}, \alpha_n: \mu_{p^n}\to \EE[p^n]\text{ is a closed immersion of group schemes}\right\}.
	\end{align}
	Put $Z_{00}(p^n)$ the moduli space parameterizing isomorphism classes of
	\begin{align}
		\{(\EE,\varphi_n)\colon \EE\text{ is an ordinary elliptic curve}, \varphi_n: \hat{\EE}[p^n] \to \mu_{p^n}\text{ is an isomorphism of group schemes}\}.
	\end{align}
	There is then a natural inclusion $Z_{00}(p^n) \to X_{00}(p^n)$ sending $(\EE,\varphi_n)$ to $(\EE,i_{\EE}\circ\varphi_n^{-1})$, where $i_\EE: \hat{\EE}\to \EE$ is the natural inclusion. If $B$ is a pro-artinian $W$-algebra, we have $Z_{00}(p^n)(B)\subset X_{00}(p^n)(B)$ is exactly the subset cut out by the ordinary condition; for this reason we loosely refer to $Z_{00}(p^n)$ as the ``ordinary locus''. As for the geometry, recall from \cite[\S X]{katz-moduli} that $Z_{00}(p^n)$ is represented by a stack over $W$, which is irreducible by a classical result of Igusa \cite{igusa}.
	
	Next, after Katz \cite{katz-modular-scheme}, for $k\in \Z$, a weight-$k$ modular form $f$ on $Z_{00}(p^n)$ is a function on the triples $(\EE,\varphi_n,\omega)$, where $(\EE,\varphi_n) \in Z_{00}(p^n)$ and $\omega\in H^0(\EE,\Omega^1_\EE)$ is nowhere vanishing, such that
	\begin{enumerate}
		\item[(i)] If $B$ is a $p$-adically complete separated $W$-algebra, and $(\EE,\varphi_n,\omega)$ is defined over $B$, then $f(\EE,\varphi_n,\omega) \in B$. Also, if $\beta: B\to B'$ is a morphism of $W$-algebras, let $(\EE_{B'},\varphi_{n,B'},\omega_{B'})$ be the base change of $(\EE,\varphi_n,\omega)$ to $B'$ via $\beta$. Then
		\begin{align}
			f(\EE_{B'},\varphi_{n,B'},\omega_{B'}) = \beta(f(\EE,\varphi_n,\omega)).
		\end{align}
		Finally, $f(\EE,\varphi_n,\omega) = f(\EE',\varphi_n',\omega')$ if there is an isomorphism $\sigma: \EE\to \EE'$  such that $\varphi_n = \varphi_n'\circ\sigma$ and $\omega = \sigma^*\omega'$.
		
		\item[(ii)] Suppose $(\EE,\varphi_n,\omega)$ is defined over a $p$-adically complete separated $W$-algebra $B$. For all $\lambda \in B^\times$, $f(\EE,\varphi_n,\lambda\omega) = \lambda^{-k}f(\EE,\varphi_n,\omega)$.
		
		\item[(iii)] Let $\widehat{W\llrrparen{\qq}}$ be the completion of $W\llrrparen{\qq}$, so it is $p$-adically complete and separated. For the Tate curve $\Tate(\qq)/\widehat{W\llrrparen{\qq}}$, let $\varphi_{n,\can}: \widehat{\Tate(\qq)}[p^n]\to \mu_{p^n}$ be the isomorphism induced from the natural inclusion $\mu_{p^n}\to \Tate(\qq)$. Then $f(\Tate(\qq),\varphi_{n,\can},\omega_\can)\in W[[\qq]]$. 
	\end{enumerate}
	We are ready to state the $\qq$-expansion principle for $Z_{00}(p^n)$.
	\begin{prop}\label{prop:q-exp-principle}
		Let $f$ be a weight-$k$ modular form on $Z_{00}(p^n)$ that vanishes on the triple $(\Tate(\qq),\varphi_{n,\can},\omega_{\can})$. Then $f$ is identically zero.
	\end{prop}
	\begin{proof}
		One way to see this is to use the full tower $M^{\triv}$ from \cite[\S X]{katz-moduli} that parametrizes isomorphism classes of $(\EE,\varphi)$, where $\varphi: \hat{\EE}\to \hat{\G}_m$ is an isomorphism of formal groups. That is, we can regard a form $f$ on $Z_{00}(p^n)$ as a generalized modular form $f^\gen$ on $M^\triv$, by setting
		\begin{align}
			f^\gen(\EE,\varphi) = f(\EE,\varphi|_{\hat{\EE}[p^n]},\varphi^*(dt/t));
		\end{align}
		\textit{cf}., \cite[\S1.5]{katz-eisenstein-measure}. The association $f\mapsto f^\gen$ is faithful, i.e., $f^\gen = 0$ implies $f = 0$. To see this, simply note that for any point of evaluation $(\EE,\varphi_n,\omega)$, there exist $\varphi:\hat{\EE}\xrightarrow{\sim} \hat{\G}_m$ extending $\varphi_n$ and a scalar $\lambda$ such that
		\begin{align}
			(\EE,\varphi_n,\omega) = (\EE,\varphi|_{\hat{\EE}[p^n]},\lambda\varphi^*(dt/t)).
		\end{align}
		To conclude the proof, we use the $\qq$-expansion principle on $M^\triv$ \cite[p.~499, bottom lemma]{katz-moduli}.
	\end{proof}
	
	\subsection{Construction of the invariants}
	
	We now prove the existence of a family of Hasse-type modular invariants on the ordinary loci $Z_{00}(p^n)$'s.
	\begin{lem}
		For all $n\in \Z_{>0}$, there exists a weight-one modular form $A_n$ valued in $W/p^n$ on $Z_{00}(p^n)$, such that $A_n(\qq) = A_n(\Tate(\qq),\varphi_{n,\can},\omega_\can) = 1$.
	\end{lem}
	\begin{proof}
		Let $B$ be a $p$-adically complete separated $W$-algebra, and $(\EE,\varphi_n,\omega)$ be a triple over $B$, with $\EE$ ordinary, $\varphi_n:\hat{\EE}[p^n]\xrightarrow{\sim}\mu_{p^n}$ and $\omega\in H^0(\EE,\Omega^1_{\EE/B})$ a generator. Recall that an ordinary elliptic curve $\EE$ over $B$ has a trivialization $\varphi: \hat{\EE}\xrightarrow{\sim} \hat{\G}_m$ \cite[Corollary 4.3.3]{lubin}, and that any other $\varphi'$ is a $\Z_p^\times$-multiple of $\varphi$. As such, we can choose a $\varphi$ such that $\varphi|_{\hat{\EE}[p^n]} = \varphi_n$; this determines $\varphi$ up to $1+p^n\Z_p$. Now, form
		\begin{align}
			A_n(\EE,\varphi_n,\omega) = \varphi^*(dt/t)/\omega\in B/p^nB,
		\end{align}
		which, as mentioned earlier, is well-defined. We contend that $A_n$ satisfies the properties (i-iii) above. In fact, the verifications for (i) and (ii) are formal, so we omit them. For (iii), we prove the stronger statement that $A_n(\qq) = 1$. 
		
		Let $\varphi_{\can} : \widehat{\Tate(\qq)}\to \hat{\G}_m$ be the inverse limit of $\varphi_{r,\can}$ for $r\in \Z_{>0}$, which we will use to compute $A_n(\qq)$. The map $\varphi_{r,\can}^{-1}$ is a factor of the natural inclusion $i_r: \mu_{p^r}\to \Tate(\qq)$. On the affine part of $\Tate(\qq)$, $i_r$ corresponds to the $W_\qq = \widehat{W\llrrparen{\qq}}$-algebra morphism
		\begin{align}
			\frac{W_\qq[x,y]}{(y^2 + xy - x^3 - a_4(\qq)x -a_6(\qq))} \longrightarrow W_\qq[t]/(t^{p^r}-1), \quad x\mapsto \fx(t), y\mapsto \fy(t).
		\end{align}
		Here (see \cite[Theorem V.1.1]{silverman-advanced}),
		\begin{align}
			a_4(\qq) &= -5\sum_{m,l>0}\qq^{ml}l^3,\quad
			a_6(\qq) = -\sum_{m,l>0}\qq^{ml}\cdot \frac{5l^3 + 7l^5}{12};\\
			\fx(t) &= \sum_{l\in \Z}\frac{\qq^lt}{(1-\qq^lt)^2} - 2\sum_{m,l>0}\qq^{ml},\quad
			\fy(t) = \sum_{l\in \Z}\frac{(\qq^l t)^2}{(1-\qq^lt)^3} + \sum_{m,l>0} \qq^{ml}.
		\end{align}
		It can then be verified by hand that $i_r^*\omega_{\can} = d\fx(t)/(\fx(t) + 2\fy(t)) = dt/t$. Let $r$ vary and we get $\varphi_\can^*(dt/t) = \omega_\can$.
	\end{proof}
	
	\begin{rem}
		Let $A$ be the classical full-level Hasse invariant defined over $\Z/p$ \cite[\S2.0]{katz-modular-scheme}. We contend that $A_1^{p-1} = A|_{Z_{00}(p)}$. More generally, denote by $E_{p-1}$ the normalized Eisenstein series of weight $p-1$ (\textit{ibid.},  \S2.1), then we have $A_{n+1}^{(p-1)p^n} \equiv (E_{p-1}|_{Z_{00}(p^n)})^{p^n}\bmod p^{n+1}$ for all $n\in \Z_{\ge 0}$. To see this, it suffices to compare their $\qq$-expansions as both forms are of weight $(p-1)p^n$. The classical Clausen--von-Staudt congruence implies that $E_{p-1}(\qq)\equiv 1\bmod p$, and thus $E_{p-1}(\qq)^{p^n}\equiv 1\equiv A_{n+1}^{(p-1)p^n}(\qq)\bmod p^{n+1}$. Proposition \ref{prop:q-exp-principle} then extends the congruence to all of $Z_{00}(p^n)$.
	\end{rem}
	\bibliographystyle{alpha}
	\bibliography{references.bib}
\end{document}